\documentclass[11pt]{amsart}
\usepackage[utf8]{inputenc}
\usepackage{cite}

\usepackage{epsfig}
\usepackage{amssymb,amsmath,amsthm,amscd}
\usepackage{latexsym}
\usepackage[margin=1in]{geometry}
\usepackage[all,knot]{xy}
\usepackage{color}
\usepackage{mathrsfs}
\usepackage{eucal}
\usepackage{subcaption}
\usepackage{etoolbox}

\usepackage{graphicx}
\usepackage{epstopdf}
\usepackage[colorlinks=true,allcolors=blue]{hyperref}

\DeclareMathAlphabet{\mathpzc}{OT1}{pzc}{m}{it}

\numberwithin{equation}{section}

\newtheorem{thm}{Theorem}[section]
\newtheorem{prop}[thm]{Proposition}
\newtheorem{cor}[thm]{Corollary}
\newtheorem{lem}[thm]{Lemma}

\newtheorem{conj}{Conjecture}[section]

\theoremstyle{definition}
\newtheorem{defn}[thm]{Definition}

\newtheorem{rem}[thm]{Remark}
\newcommand{\Z}{\mathbb{Z}}
\newcommand{\R}{\mathbb{R}}
\newcommand{\C}{\mathbb{C}}

\newcommand{\F}{\mathbb{F}}

\newcommand{\scrR}{\mathscr{R}}
\newcommand{\G}{\mathscr{G}}

\newcommand{\A}{\mathscr{A}}

\newcommand{\rk}{\mathrm{rk}}

\newcommand{\Hom}{\mathrm{Hom}}

\newcommand{\Tr}{\operatorname{Tr}}

\newcommand{\Tor}{\operatorname{Tor}}

\newcommand{\im}{\operatorname{im}}

\newcommand{\gr}{\operatorname{gr}}
\newcommand{\id}{\operatorname{id}}

\newcommand{\CF}{\mathrm{CF}}

\newcommand{\HF}{\mathrm{HF}}
\newcommand{\HFhat}{\widehat{\mathbf{HF}}}

\newcommand{\SI}{\operatorname{SI}}
\newcommand{\CSI}{\operatorname{CSI}}

\newcommand{\bfalpha}{\boldsymbol{\alpha}}
\newcommand{\bfbeta}{\boldsymbol{\beta}}
\newcommand{\bfgamma}{\boldsymbol{\gamma}}
\newcommand{\bfdelta}{\boldsymbol{\delta}}
\newcommand{\bfx}{\mathbf{x}}

\newcommand{\calH}{\mathcal{H}}
\newcommand{\bfi}{\mathbf{i}}
\newcommand{\bfj}{\mathbf{j}}
\newcommand{\bfk}{\mathbf{k}}
\newcommand{\bbL}{\mathbb{L}}
\newcommand{\bbK}{\mathbb{K}}

\newcommand{\Khr}{\operatorname{Khr}}

\newcommand{\Hol}{\operatorname{Hol}}
\newcommand{\calD}{\mathcal{D}}
\newcommand{\CKh}{\operatorname{CKh}}
\newcommand{\CKhr}{\operatorname{CKhr}}
\newcommand{\Kh}{\operatorname{Kh}}
\newcommand{\bfC}{\mathbf{C}}
\newcommand{\bfD}{\boldsymbol{\partial}}
\newcommand{\Cone}{\operatorname{Cone}}
\newcommand{\bfeta}{\boldsymbol{\eta}}
\newcommand{\Hd}{\operatorname{Hd}}
\newcommand{\calM}{\mathcal{M}}
\newcommand{\calE}{\mathcal{E}}
\newcommand{\calI}{\mathcal{I}}
\newcommand{\calS}{\mathcal{S}}
\newcommand{\calB}{\mathcal{B}}
\newcommand{\CS}{\operatorname{CS}}
\newcommand{\IC}{\mathrm{IC}}
\newcommand{\I}{\mathrm{I}}
\newcommand{\scrC}{\mathscr{C}}

\newcommand{\SU}{\mathrm{SU}}
\newcommand{\SO}{\mathrm{SO}}

\newcommand\todo[1]{\textbf{\textcolor{red}{#1}}}
\newcommand\hide[1]{}

\title[Traceless Character Varieties, Spectral Sequences, and Khovanov Homology]{Traceless Character Varieties, A Link Surgeries Spectral Sequence, and Khovanov Homology}

\author{Henry T. Horton}

\begin{document}

\begin{abstract}
In \cite{horton1}, we constructed a well-defined Lagrangian Floer invariant for any closed, oriented $3$-manifold $Y$ via the symplectic geometry of so-called traceless $\SU(2)$-character varieties. This invariant, $\SI(Y)$, which we refer to as the {\bf symplectic instanton homology} of $Y$, was also shown to satisfy an exact triangle for Dehn surgeries on knots which is typical of Floer-theoretic invariants of $3$-manifolds.

In this article, we demonstrate further structural properties of this symplectic instanton homology. For example, Floer theories are expected to roughly satisfy the axioms of a topological quantum field theory (TQFT), so that in particular they should be functorial with respect to cobordisms. Following a strategy used by Ozsv\'ath and Szab\'o in the context of Heegaard Floer homology, we prove that our theory is functorial with respect to connected $4$-dimensional cobordisms, so that cobordisms induce homomorphisms between symplectic instanton homologies. We also generalize the surgery exact triangle by proving that Dehn surgeries on a \emph{link} $L$ in a $3$-manifold $Y$ induce a spectral sequence of symplectic instanton homologies -- the $E^2$-page is isomorphic to a direct sum of symplectic instanton homologies of all possible combinations of $0$- and $1$-surgeries on the components of $L$, and the spectral sequence converges to $\SI(Y)$. For the branched double cover $\Sigma(L)$ of a link $L \subset S^3$, we show there is a link surgery spectral sequence whose $E^2$-page is isomorphic to the reduced Khovanov homology of $L$ and which converges to the symplectic instanton homology of $\Sigma(L)$.
\end{abstract}

\maketitle

\tableofcontents

\addtocontents{toc}{\protect\setcounter{tocdepth}{1}}

\section{Introduction}

The study of anti-self dual $\SU(2)$-connections on smooth $4$-manifolds led to remarkable development in low-dimensional topology in the 1980's. At the end of that decade, Andreas Floer \cite{instanton-invariant} defined an invariant of any homology $3$-sphere $Y$, the {\bf instanton homology group} $I(Y)$, which is obtained from a chain complex generated by nontrivial flat $\SU(2)$-connections on $Y$ with differential counting anti-self dual connections on $Y \times \R$ interpolating between given flat connections. The study of these groups led to more unprecedented advances in low-dimensional topology.

In a similar vein, Floer \cite{Floer1} also defined a similar homological invariant in the context of symplectic geometry. Given a pair of Lagrangian submanifolds $L_0$, $L_1$ in some symplectic manifold $(M, \omega)$, this {\bf Lagrangian Floer homology} is computed from a chain complex with generators the intersection points of $L_0$ and $L_1$ and the differential  counts pseudoholomorphic strips $u: \R \times [0,1] \longrightarrow M$ with $u(\R, 0) \subset L_0$ and $u(\R, 1) \subset L_1$ that tend asymptotically to elements of $L_0 \cap L_1$.

Despite the different contexts in which instanton and Lagrangian Floer homology arise, there is a tantalizing symplectic interpretation of the objects appearing in the instanton chain complex. If one fixes a genus $g$ Heegaard splitting $Y = H_0 \cup_{\Sigma_g} H_1$, then the $\SU(2)$-character variety $M(\Sigma_g)$ is naturally a (stratified) symplectic manifold, and by restriction to the boundary the $\SU(2)$-character varieties of the handlebodies $H_i$ embed as Lagrangian submanifolds $L(H_i)$ in $M(\Sigma_g)$. The intersection points $L(H_0) \cap L(H_1)$ are in one-to-one correspondence with flat $\SU(2)$-connections on $Y$. Furthmore, holomorphic curves in $M(\Sigma_g)$ can be shown to give rise to anti-self dual connections near the Heegaard surface (\emph{e.g.} on $(\Sigma_g \times \R) \times \R$). These observations led Atiyah \cite{atiyah-floer} to posit the {\bf Atiyah-Floer conjecture}, which says that in this setting, $I(Y) \cong \HF(L(H_0), L(H_1))$.

While inspiring, the Atiyah-Floer conjecture is not quite well-posed. The problem is that $M(\Sigma_g)$ is only a stratified symplectic manifold, so that globally it has singularities. As a result, the Lagrangian Floer homology $\HF(L(H_0),L(H_1))$ is not obviously well-defined, and a lot of work is required to make sense of the relevant holomorphic curve counts. Salamon and Wehrheim have a long-active program to rigorously define $\HF(L(H_0),L(H_1))$ \cite{salamon-af,salamon-wehrheim,wehrheim-af}, but another approach would be to work in a setting where we can be sure to avoid any singularities in the relevant symplectic manifolds. On the gauge theory side, Kronheimer and Mrowka \cite{yaft} have described techniques by which one may obtain smooth configuration spaces of flat connections. In \cite{horton1}, we adapted a variant of Kronheimer-Mrowka's constructions to the symplectic setting, obtaining a so-called {\bf symplectic instanton homology} $\SI(Y)$.

The idea of our construction is roughly as follows. Let $Y$ be a closed, oriented $3$-manifold (not necessarily a homology sphere) with a Heegaard splitting $Y = H_0 \cup_{\Sigma_g} H_1$. If we choose a basepoint $x \in \Sigma_g$, then in a small $3$-ball neighborhood of $x$ we may remove a regular neighborhood of a $\theta$-graph (\emph{i.e.} a graph with two vertices and three edges connecting them) from $Y$. Here the $\theta$-graph is ``standardly'' embedded, so that each edge intersects $\Sigma_g$ once and each handlebody $H_0$, $H_1$ contains one of the vertices. The intersections of the pieces of the Heegaard decomposition with the complement of this $\theta$-graph are written $\Sigma_g^\theta$, $H_\alpha^\theta$, and $H_\beta^\theta$.

The key point is that instead of looking at the usual $\SU(2)$-character variety of each piece of the Heegaard splitting, we instead add the condition that meridians of the edges of the $\theta$-graph should be sent to the conjugacy class of traceless $\SU(2)$ matrices. This means that we associate to $\Sigma_g^\theta$ the {\bf traceless $\SU(2)$-character variety}
\[
	\scrR_{g,3} = \left.\left\{ A_1, B_1, \dots, A_g, B_g, C_1, C_2, C_3 \in \SU(2) ~\left|~ \begin{array}{c}\prod_{k = 1}^g [A_k, B_k] = C_1C_2C_3, \\ \operatorname{Tr}(C_k) = 0 \end{array}\right\}\right.\right/\text{conj}.
\]
$\scrR_{g,3}$ is naturally a smooth symplectic manifold, and the images $L_0$, $L_1$ of the traceless character varieties of $H_0^\theta$, $H_1^\theta$ in $\scrR_{g,3}$ are smooth Lagrangian submanifolds. In \cite{horton1} we proved that
\[
	\SI(Y) = \HF(L_0, L_1),
\]
the {\bf symplectic instanton homology} of $Y$, is a natural homeomorphism invariant of $Y$.

More generally, given a homology class $\omega \in H_1(Y; \Z/2)$, there is a symplectic instanton homology $\SI(Y,\omega)$ related to the moduli space of flat connections on the $\SO(3)$-bundle $P \longrightarrow Y$ with second Stiefel-Whitney class Poincar\'e dual to $\omega$, see \cite[Section 7]{horton1}. The main structural properties of $\SI(Y)$ established in \cite{horton1} are the K\"unneth principle for connected sums,
\[
	\SI(Y \# Y') \cong (\SI(Y) \otimes \SI(Y')) \oplus \Tor(\SI(Y),\SI(Y')),
\]
and the existence of an exact triangle for Dehn surgeries,
\[
	\xymatrix{\SI(Y,\omega_K) \ar[rr] & & \SI(Y_\lambda(K)) \ar[dl] \\
	 & \SI(Y_{\lambda + \mu}(K)) \ar[ul] & }
\]
where $\omega_K$ is the mod $2$ homology class of the knot $K$ the Dehn surgery is performed along.

The purpose of the present article to to exhibit further structural properties of these symplectic instanton homology groups. For example, one further expects that a Floer theoretic invariant for $3$-manifolds associates to a cobordism $W: Y \longrightarrow Y'$ between $3$-manifolds a homomorphism between their Floer homology groups, (almost) giving the structure of a $(3+1)$-dimensional TQFT. We show that this is indeed the case for our symplectic instanton homology groups (see Theorem \ref{thm:cobordism-maps}):

\begin{thm}
Associated to any connected cobordism $W: Y \longrightarrow Y'$ of connected, closed, oriented $3$-manifolds $Y$, $Y'$, there is a homomorphism $\SI(W): \SI(Y) \longrightarrow \SI(Y')$ which is a diffeomorphism invariant of $W$. The assignment $W \mapsto \SI(W)$ is functorial in the sense that $\SI(Y \times [0,1]) = \id_{\SI(Y)}$ and for any second cobordism $W': Y' \longrightarrow Y''$, $\SI(W \cup W^\prime) = \SI(W^\prime) \circ \SI(W)$.
\end{thm}

We may also define homomorphisms between the symplectic instanton homology groups for nontrivial $\SO(3)$-bundles that are induced by $4$-dimensional cobordisms $W$ equipped with a mod $2$ homology class $\Omega \in H_2(W, \partial W; \F_2)$:

\begin{thm}
Given a compact, connected, oriented cobordism $W: Y \longrightarrow Y'$ of closed, connected, oriented $3$-manifolds and a mod $2$ homology class $\Omega \in H_2(W, \partial W; \F_2)$, there is an induced homomorphism
\[
	\SI(W, \Omega): \SI(Y, \partial \Omega|_Y) \longrightarrow \SI(Y', \partial \Omega|_{Y^\prime})
\]
that is an invariant of the $(W, \Omega)$. Furthermore, the assignment $(W, \Omega) \mapsto \SI(W, \Omega)$ is functorial in the sense that
\begin{itemize}
	\item[(1)] $\SI(Y \times [0,1], \omega \times [0,1]) = \id_{\SI(Y,\omega)}$.
	\item[(2)] $\SI(W', \Omega') \circ \SI(W,\Omega) = \SI(W \cup_{Y^\prime} W', \Omega + \Omega')$.
\end{itemize}
\end{thm}

With the functoriality of $\SI(Y,\omega)$ in place, we can identify the maps in surgery exact triangle for symplectic instanton homology:

\begin{thm}
\textup{(cf. \cite[Theorem 1.3]{horton1})} For any framed knot $(K, \lambda)$ in a closed, oriented $3$-manifold $Y$, there is an exact triangle
\[
	\xymatrix{\SI(Y, \omega_K) \ar[rr]^{\SI(W,\Omega_K)} & & \SI(Y_\lambda(K)) \ar[dl]^{\phantom{p}\SI(W_\lambda)} \\
	 & \SI(Y_{\lambda + \mu}(K)) \ar[ul]^{\SI(W_{\lambda + \mu},\Omega_K^\prime)\phantom{ppp}} & }
\]
where $\omega_K$ is the mod $2$ homology class in $Y$ represented by the knot $K$ and $\Omega_K$ is the relative mod $2$ homology class in $W$ represented by the core of the $2$-handle attached to $K$ (similarly for $\Omega_K^\prime$ in $W_{\lambda + \mu}$).
\end{thm}

Note that the surgery exact triangle is for Dehn surgery on a \emph{knot}. For general \emph{links}, the surgery exact triangle must be replaced with a spectral sequence.

\begin{thm}
For any framed link $(L, \lambda)$ in a closed, oriented $3$-manifold $Y$, there is a spectral sequence whose $E^1$-page is a direct sum of symplectic instanton homologies of all possible combinations of $\lambda_k$- and $(\lambda_k + \mu_k)$-surgeries on the components $L_k$ of $L$ that converges to $\SI(Y, \omega_L)$.
\end{thm}

As a special application of the above link surgeries spectral sequence, we establish a relationship between Khovanov homology and symplectic instanton homology.

\begin{thm}
For any link $L \subset S^3$, there is a spectral sequence with $E^2$-page isomorphic to $\Kh(m(L); \F_2)$ converging to $\SI(\Sigma(L))$. In particular, we have a rank inequality $\rk \Kh(m(L); \F_2) \geq \rk \SI(\Sigma(L))$.
\end{thm}

\textbf{Acknowledgements.} We thank Paul Kirk, Dylan Thurston, and Chris Woodward for useful conversations on various aspects of this work. This article consists of material which makes up part of the author's Ph.D thesis, written at Indiana University.
\section{Review of Symplectic Instanton Homology}

\label{chap:review}

Here we will recall the basic definitions and properties of the symplectic instanton homology groups $\SI(Y)$ established in \cite{horton1}.

\subsection{Traceless Character Varieties}

Suppose $M$ is a compact $3$-manifold (possibly with boundary) or a closed surface (without boundary) and let $T \subset M$ be a properly embedded codimension-$2$ submanifold of $M$. There is a meridian $\mu_i \in \pi_1(M \setminus T)$ associated to each connected component $T_i$ of $T$. We define the {\bf traceless $\SU(2)$-character variety} of $(M,T)$ to be the representation space
\[
	\scrR(M,T) = \{\rho: \pi_1(M \setminus T) \longrightarrow \SU(2) \mid \Tr(\rho(\mu_i)) = 0\}/\text{conjugation}.
\]
This is well-defined as each choice of meridian $\mu_i$ is unique up to conjugation. The traceless character varieties of punctured surfaces are given the special notation
\[
	\scrR_{g,n} = \scrR(\Sigma_g, \{n \text{ pts}\}).
\]
By fixing a standard basis for $\pi_1(\Sigma_g \setminus \{n \text{ pts}\})$, we obtain the {\bf holonomy description}
\[
	\scrR_{g,n} = \left.\left\{ A_1, \dots, A_g, B_1, \dots, B_g, C_1, \dots, C_n \in \SU(2) ~\left|~ \begin{array}{c} \displaystyle \prod_{j = 1}^g [A_j, B_j] \prod_{k = 1}^n C_k = I \\ \Tr(C_k) = 0 \end{array} \right\}\right/\text{conj.}\right.,
\]
which allows us to denote elements of $\scrR_{g,n}$ by $[A_1, \dots, A_g, B_1, \dots, B_g, C_1, \dots, C_n]$.

The essential features of the smooth and symplectic topology of $\scrR_{g,n}$ are as follows:
\begin{itemize}
	\item If $n$ is odd, then $\scrR_{g,n}$ is a smooth manifold of dimension $6g - 6 + 2n$.
	\item $\scrR_{g,n}$ is connected and simply connected.
	\item $\scrR_{0,3} = \{[\bfi, \bfj, -\bfk]\}$, a single point.
	\item There is a natural symplectic form $\omega_{g,n}$ on $\scrR_{g,n}$ (defined analogously to the classical Goldman symplectic form on $\scrR_{g,0}$, see \cite{goldman} and \cite{atiyah-bott}).
	\item $(\scrR_{g,n}, \omega_{g,n})$ is monotone with monotonicity constant $\tau = \tfrac{1}{4}$ and minimal Chern number $1$.
\end{itemize}

\subsection{The Symplectic Instanton Homology Groups}

Let $Y$ be a closed, oriented $3$-manifold and suppose we have a genus $g$ Heegaard splitting $Y = H_0 \cup_{\Sigma_g} H_1$. Fix a basepoint $x \in \Sigma_g \subset Y$, and in a small $3$-ball neighborhood of $x$ remove a standardly embedded $\theta$-graph with one vertex in each handlebody $H_i$ and each edge intersecting the Heegaard surface transversely in a single point (see Figure \ref{fig:tripod} for an illustration of the intersection of the complement of the $\theta$-graph in the $3$-ball with one of the handlebodies).

\begin{figure}[t]
	\centering
	\includegraphics[scale=1.5]{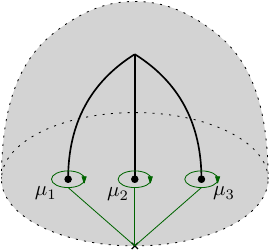}
	\caption{Half of the $\theta$-graph lying in one of the handlebodies $H_i$, with meridians labeled.}
	\label{fig:tripod}
\end{figure}

This setup gives us a diagram of inclusions of pairs and an induced diagram of traceless character varieties:
\[
	\xymatrix{ & (H_0, H_0 \cap \theta) \ar@{^(->}[dl] & & & & \scrR(H_0, H_0 \cap \theta) \ar@{_(->}[dr] &  \\
	(Y, \theta) & & (\Sigma_g, \Sigma_g \cap \theta) \ar@{_(->}[ul] \ar@{^(->}[dl] & \xrightarrow{~\scrR~}& \scrR(Y, \theta) \ar[ur] \ar[dr] & & \scrR_{g,3} \\
	 & (H_1, H_1 \cap \theta) \ar@{_(->}[ul] & & & & \scrR(H_1, H_1 \cap \theta) \ar@{^(->}[ur] & }
\]
Each $\scrR(H_i, H_i \cap \theta)$ is a copy of $(S^3)^g$ embedding in $(\scrR_{g,3}, \omega_{g,3})$ as a monotone Lagrangian submanifold (see \cite[Section 3]{horton1}). We write $L_i = \im(\scrR(H_i, H_i \cap \theta) \longrightarrow \scrR_{g,3})$. One has an identification
\[
	L_0 \cap L_1 \cong \Hom(\pi_1(Y), \SU(2)).
\]

The Lagrangian Floer homology
\[
	\SI(Y) = \HF(L_0, L_1)
\]
is the {\bf symplectic instanton homology} of $Y$. By \cite[Section 5]{horton1} it is a natural homeomorphism invariant of $(Y,x)$. The most important computations of $\SI(Y)$ for the purposes of this article are:
\begin{itemize}
	\item $\SI(S^3) \cong \Z$. Indeed, the Floer chain complex is generated by a single point (corresponding to the trivial representation $\pi_1(S^3) \longrightarrow \SU(2)$) and the differential is necessarily trivial.
	\item $\SI(S^2 \times S^1) \cong H^\ast(S^3)$ as an algebra, not just a group. The product structure on Floer homology is given as usual by the triangle product, and the corresponding product on $H^\ast(S^3)$ is the cup product. Concretely, $\SI(S^2 \times S^1)$ has two generators $\theta$ and $\Theta$ (corresponding to the trivial representation and the nontrivial central representation of $\pi_1(S^2 \times S^1)$ in $\SU(2)$, respectively) with $\Theta$ acting as the unit and $\theta$ squaring to zero.
\end{itemize}

We remark that while symplectic instanton homology can be defined with $\Z$ coefficients, we work with $\F_2$ coefficients throughout this article.

It is often convenient to encode a Heegaard splitting by a {\bf Heegaard diagram} $(\Sigma_g, \bfalpha, \bfbeta, z)$, where $\bfalpha$ is a $g$-tuple of curves linearly independent in $H_1(\Sigma_g)$ (similarly for $\beta$) and $z \in \Sigma_g$ is a basepoint. In a small disk neighborhood of $z$ one places three punctures with meridians $\mu_1, \mu_2, \mu_3$ and we have Lagrangians
\[
	L_\alpha = \left\{\rho: \pi_1(\Sigma_{g,3}) \longrightarrow \SU(2) \left| \begin{array}{ll} \rho(\alpha_k) = I, & k = 1, \dots, g \\ \rho(\mu_1\mu_2\mu_3) = I & \end{array} \right\} \right./\text{conjugation},
\]
\[
	L_\beta = \left\{\rho: \pi_1(\Sigma_{g,3}) \longrightarrow \SU(2) \left| \begin{array}{ll} \rho(\beta_k) = I, & k = 1, \dots, g \\ \rho(\mu_1\mu_2\mu_3) = I & \end{array} \right\} \right./\text{conjugation}.
\]
In terms of the previously defined Lagrangians we have $L_\alpha = L_0$ and $L_\beta = L_1$ if $H_0$ is the handlebody obtained by adding $2$-handles along the $\alpha$-curves and $H_1$ is the handlebody obtained by adding $2$-handles along the $\beta$-curves.

Finally, given a homology class $\omega \in H_1(Y; \Z/2)$, one may represent it by an embedded, possibly disconnected curve in the Heegaard surface, which we also denote $\omega$. We may then define a modified Lagrangian
\[
	L_\alpha^\omega = \left\{\rho: \pi_1(\Sigma_{g,3}) \longrightarrow \SU(2) \left| \begin{array}{ll} \rho(\alpha_k) = (-1)^{\alpha_k \cdot \omega}I, & k = 1, \dots, g \\ \rho(\mu_1\mu_2\mu_3) = I & \end{array} \right\} \right./\text{conjugation},
\]
where $\alpha_k \cdot \omega$ is the unoriented intersection number. We then define
\[
	\SI(Y,\omega) = \HF(L_\alpha^\omega, L_\beta).
\]
This variant of symplectic instanton homology corresponds to $\SO(3)$ representations of the fundamental group with second Stiefel-Whitney class Poincar\'e dual to $\omega$.
\section{The K\"unneth Principle for Triangle Maps}

% \todo{NEW: Implement a rigorous treatment via seam swap maps. Also, the K\"unneth theorem holds for general Floer field theory by the critical point switch condition. Can probably define ``unital'' FFT and establish the K\"unneth principle for polygons in full generality.}

Many of the results in this article rely on a technical result which we refer to as the ``K\"unneth principle,'' which allows us to ``localize'' certain naturally-occurring maps defined on the symplectic instanton homology of a connected sum to one of the connect summands. In this section we expand upon what exactly is meant by this.

\subsection{Relative Invariants of Pseudoholomorphic Quilts}

We start by reviewing how counting certain pseudoholomorphic maps defines chain maps between quilted Floer chain groups. We first must introduce some terminology.

A {\bf surface with strip-like ends} consists of the following data:
\begin{itemize}
	\item A surface with boundary $\overline{S}$, and an enumeration of the boundary components $\partial \overline{S} = C_1 \amalg \cdots \amalg C_m$.
	\item For each boundary component $C_k$, a finite set (possibly empty) of $d_k$ points $z_{k,1}, \dots,\allowbreak z_{k,d_k} \in C_k$, labelled cyclically according to the induced orientation of $C_k$. We denote the indexing set for all such marked points
	\[
		\calE(S) = \{e = (k, l) \mid 1 \leq k \leq m, 1 \leq l \leq d_k\}.
	\]
	We will write $e \pm 1 = (k, l \pm 1)$ to denote the next/previous label adjacent to $e$ with respect to the cyclic ordering. $I_{k,l} \subset C_k$ will denote the open arc in $C_k$ between $z_{k,l}$ and $z_{k,l+1}$ (or $I_{k,0} = C_k$ if $d_k = 0$).
	\item A complex structure $j$ on the punctured surface $S = \overline{S} \setminus \{z_e\}_{e \in \calE(S)}$.
	\item For each $e \in \calE(S)$, a set of embeddings
	\[
		\epsilon_{S,e}: \R^\pm \times [0,1] \longrightarrow S,
	\]
	where $\R^+ = (0,\infty)$ and $\R^- = (-\infty, 0)$ (a sign is chosen for each $e$, independent of the previous data), such that
	\[
		\lim_{t \to \pm \infty} \epsilon_{S,e}(t,s) = z_e
	\]
	and $\epsilon_{S,e}^\ast j_S$ is the canonical complex structure on the strip $\R^\pm \times [0,1]$. The $\epsilon_{S,e}$ are called {\bf strip-like ends} for $S$; ends of the form $\epsilon_{S,e}: (-\infty, 0) \times [0,1] \longrightarrow S$ are called {\bf incoming ends} and ends of the form $\epsilon_{S,e}: (0, \infty) \times [0,1] \longrightarrow S$ are called {\bf outgoing ends}. The set of labels has a natural partition $\calE(S) = \calE_-(S) \amalg \calE_+(S)$ depending on whether the corresponding label is associated to an incoming or outgoing end.
	\item Orderings of the sets of incoming and outgoing ends,
	\[
		\calE_-(S) = (e_1^-, \dots, e_{N_-}^-), \quad\quad \calE_+(S) = (e_1^+, \dots, e_{N_+}^+).
	\]
\end{itemize}

Surfaces $S$ with strip-like ends represent the domains of certain pseudoholomorphic curves with Lagrangian boundary conditions. Indeed, let $(M, \omega)$ be a closed, monotone symplectic manifold and suppose that $\mathbf{L} = \{L_e\}_{e \in \calE(S)}$ is a collection of pairwise transverse, simply connected (for simplicity), monotone Lagrangian submanifolds of $M$. We write $\calI_+(\mathbf{L})$ for the set of tuples of points $\mathbf{x}^+ = \{x_e^+ \in L_{e-1} \cap L_e\}_{e \in \calE_+(S)}$ and $\calI_-(\mathbf{L})$ for the set of tuples of points $\mathbf{x}^- = \{x_e^- \in L_e \cap L_{e-1}\}_{e \in \calE_-(S)}$. We may then define the moduli space $\calM_S(\bfx^-, \bfx^+)$, which consists of finite energy pseudoholomorphic maps $u: S \longrightarrow M$ satisfying the following boundary conditions and asymptotics:
\begin{itemize}
	\item $u(I_e) \subset L_e$ for all $e \in \calE(S)$.
	\item $\displaystyle \lim_{s \to \pm \infty} u(\epsilon_{S,e}(s,t)) = x_e^\pm$ for all $e \in \calE_\pm(S)$.
\end{itemize}
For generic almost complex structures on $(M, \omega)$, $\calM_S(\bfx^-, \bfx^+)$ is a smooth, oriented manifold whose zero-dimensional component $\calM_S(\bfx^-, \bfx^+)_0$ is a finite set of points. Hence the surface with strip-like ends $S$ determines a {\bf relative invariant} $\Phi_S$ in Floer homology defined on the chain level by
\[
	C\Phi_S: \bigotimes_{e \in \calE_-(S)} \CF(L_e, L_{e-1}) \longrightarrow \bigotimes_{e \in \calE_+(S)} \CF(L_{e-1}, L_e),
\]
\[
	C\Phi_S\left(\bigotimes_{e \in \calE_-(S)} x_e^-\right) = \sum_{\bfx^+ \in \calI_+(\mathbf{L})} \#\calM_S(\bfx^-,\bfx^+)_0 \bigotimes_{e \in \calE_+(S)} x_e^+.
\]

The above ideas extend to define relative invariants for \emph{quilted} Floer homology with just a little work. A {\bf quilted surface with strip-like ends} consists of the following data:
\begin{itemize}
	\item A collection of surfaces with strip-like ends, $\underline{S} = \{(S_k, j_k)\}_{k = 1}^m$, called the {\bf patches} of the quilt.
	\item A collection of {\bf seams} $\calS$, where a seam is a $2$-element set
	\[
		\sigma = \{(k_\sigma, e_\sigma), (k_\sigma^\prime, e_\sigma^\prime)\} \subset \bigcup_{k = 1}^m \{k\} \times \calE(S_k).
	\]
	The collection of seams is subject to the condition that $\sigma \cap \sigma' = \varnothing$ for any distinct $\sigma, \sigma' \in \calS$ (\emph{i.e.} all seams are pairwise disjoint).
	\item For each seam $\sigma \in \calS$, an identification of the corresponding pair of boundary arcs/ components,
	\[
		\varphi_\sigma: I_{k_\sigma, e_\sigma} \longrightarrow  I_{k_\sigma^\prime, e_\sigma^\prime}
	\]
	The $\{\varphi_\sigma\}_{\sigma \in \calS}$ are required to be compatible with the strip-like ends, meaning the following:
	\begin{itemize}
		\item Both $e_\sigma$ and $e_\sigma^\prime - 1$ are incoming and $\varphi_\sigma(\epsilon_{k_\sigma,e_\sigma}(t,0)) = \epsilon_{k_\sigma^\prime,e_\sigma^\prime - 1}(t,1)$, \emph{or} they are both outgoing and $\varphi_\sigma(\epsilon_{k_\sigma,e_\sigma}(t,1)) = \epsilon_{k_\sigma^\prime,e_\sigma^\prime - 1}(t,0)$.
		\item Both $e_\sigma - 1$ and $e_\sigma^\prime$ are incoming and $\varphi_\sigma(\epsilon_{k_\sigma,e_\sigma-1}(t,1)) = \epsilon_{k_\sigma^\prime,e_\sigma^\prime}(t,0)$, \emph{or} they are both outgoing and $\varphi_\sigma(\epsilon_{k_\sigma,e_\sigma-1}(t,0)) = \epsilon_{k_\sigma^\prime,e_\sigma^\prime}(t,1)$.
	\end{itemize}
	\item An {\bf end} for $\underline{S}$ is a maximal sequence of ends $\underline{e} = \{(k_i, e_i)\}$ such that $\epsilon_{k_i,e_i}(\cdot,1) = \epsilon_{k_{i+1},e_{i+1}}(\cdot,0)$ and $\{(k_i, e_i), (k_{i+1}, e_{i+1})\}$ is a seam for each $i$. As part of the data of $\underline{S}$ we require orderings of the sets of such incoming and outgoing ends,
	\[
		\calE_-(\underline{S}) = (\underline{e}_1^-, \dots, \underline{e}_{N_-(\underline{S})}^-), \quad\quad \calE_+(\underline{S}) = (\underline{e}_1^+, \dots, \underline{e}_{N_+(\underline{S})}^+).
	\]
\end{itemize}
The {\bf true boundary components} of $\underline{S}$ are the boundary components not included in the seams, which are indexed by
\[
	\calB(\underline{S}) = \left.\left(\bigcup_{k = 1}^m \{k\} \times \calE(S_k)\right)\right\backslash \left(\bigcup_{\sigma \in \calS(\underline{S})} \sigma\right).
\]

The precise definition given above is a bit cumbersome, and an intuitive description with a picture is helpful. Figure \ref{fig:quilt} gives a picture of a quilted surface with strip-like ends, along with Lagrangian boundary conditions (defined in the next paragraph).

\begin{figure}[h]
	\centering
	\includegraphics[scale=.15]{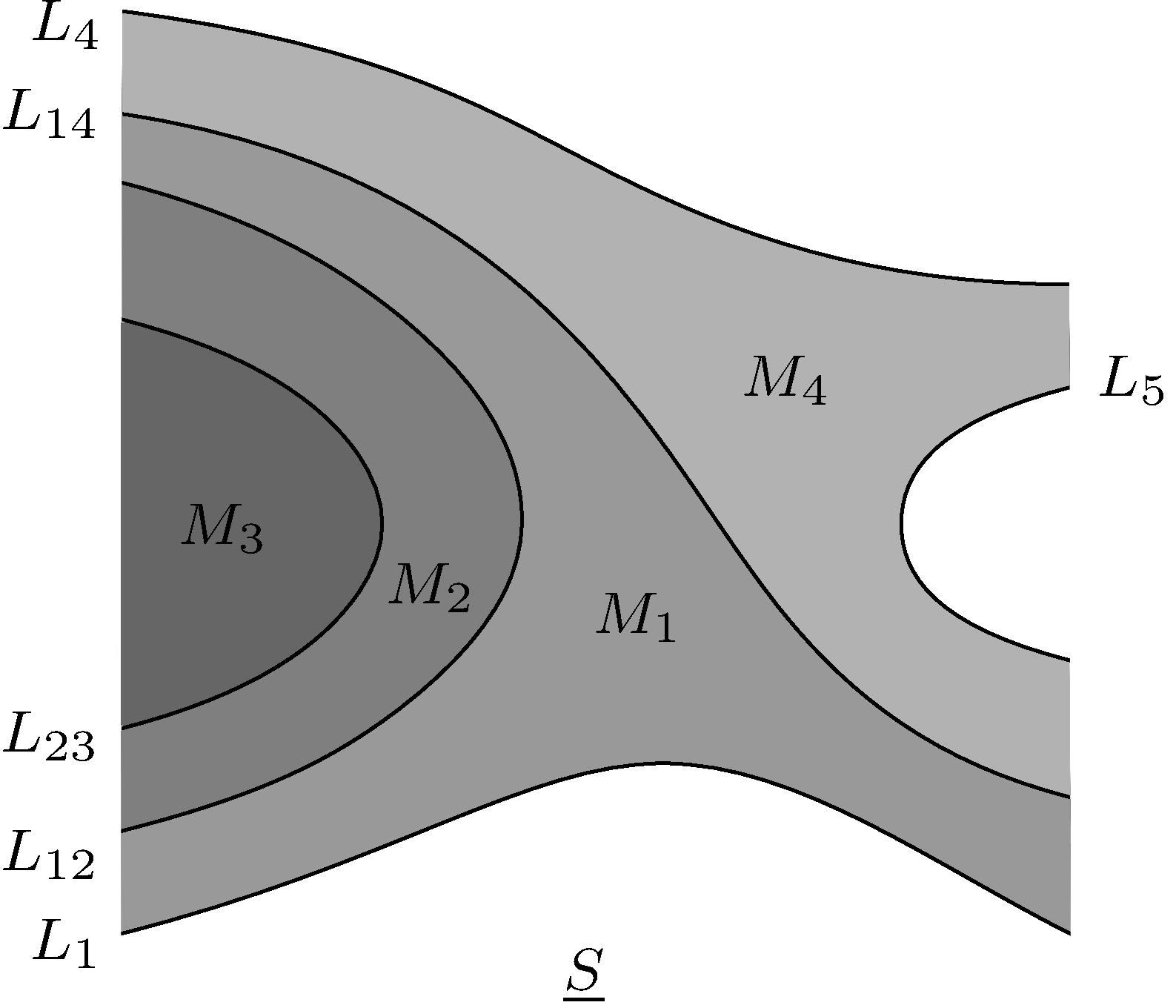}
	\caption{A quilted surface with $4$ patches, $3$ seams, $3$ true boundary components, $1$ negative end, and $2$ positive ends.}
	\label{fig:quilt}
\end{figure}

Quilted surfaces $\underline{S}$ with strip-like ends represent the domains of \emph{tuples} pseudoholomorphic curves with Lagrangian boundary conditions and certain compatibility conditions determined by the seams. Let $(M_1, \omega_1), \dots, (M_m, \omega_m)$ be a collection of closed, monotone symplectic manifolds, one for each patch of $\underline{S}$. Suppose that
\[
	\mathbf{L} = \{L_{(k_\sigma,e_\sigma),(k_\sigma^\prime,e_\sigma^\prime)} \subset M_{k_\sigma}^- \times M_{k_\sigma^\prime}\}_{\sigma \in \calS(\underline{S})} \cup \{L_{(k,e)} \subset M_k\}_{(k,e) \in \calB(\underline{S})}
\]
is a collection of pairwise transverse, simply connected (for simplicity), monotone Lagrangian correspondences/submanifolds. We may then define the moduli space $\calM_{\underline{S}}(\bfx^-, \bfx^+)$, which consists of $k$-tuples of finite energy pseudoholomorphic maps $u_j: S_j \longrightarrow M_j$ satisfying the following boundary conditions, asymptotics, and seam conditions:
\begin{itemize}
	\item $u(I_{(k,e)}) \subset L_{(k,e)}$ for all $(k,e) \in \calB(\underline{S})$.
	\item $\displaystyle \lim_{s \to \pm \infty} u_{k_i}(\epsilon_{k_i,e_i}(s,t)) = x_{(k_i,e_i)}^\pm$ for all $\underline{e} = \{(k_i,e_i)\}_{i = 1}^{n_{\underline{e}}} \in \calE_\pm(\underline{S})$.
	\item $(u_{k_\sigma}, u_{k_\sigma}^\prime \circ \varphi_\sigma)(I_{k_\sigma,e_\sigma}) \subset L_{(k_\sigma,e_\sigma),(k_\sigma^\prime,e_\sigma^\prime)}$ for all $\sigma \in \calS(\underline{S})$.
\end{itemize}
Again, for generic almost complex structures on $(M_1, \omega_1), \dots, (M_m, \omega_m)$, $\calM_{\underline{S}}(\bfx^-, \bfx^+)$ is a smooth, oriented manifold whose zero-dimensional component $\calM_{\underline{S}}(\bfx^-, \bfx^+)_0$ is a finite set of points. Hence the quilted surface with strip-like ends $\underline{S}$ determines a {\bf relative invariant} $\Phi_{\underline{S}}$ in quilted Floer homology defined on the chain level by
\[
	C\Phi_{\underline{S}}: \bigotimes_{\underline{e} \in \calE_-(\underline{S})} \CF(\underline{L}_{\underline{e}}) \longrightarrow \bigotimes_{\underline{e} \in \calE_+(\underline{S})} \CF(\underline{L}_{\underline{e}}),
\]
\[
	C\Phi_{\underline{S}}\left(\bigotimes_{\underline{e} \in \calE_-(\underline{S})} \underline{x}_{\underline{e}}^-\right) = \sum_{\bfx^+ \in \calI_+(\mathbf{L})} \#\calM_{\underline{S}}(\bfx^-,\bfx^+)_0 \bigotimes_{\underline{e} \in \calE_+(\underline{S})} \underline{x}_{\underline{e}}^+.
\]

\subsection{Relative Invariants and Geometric Composition}

Something that will be important for us is the behavior of relative invariants $\Phi_{\underline{S}}$ under geometric composition of Lagrangian correspondences on adjacent seams. More precisely, let $\underline{S}$ be a quilted surface with strip-like ends such that some patch $S_k$ is diffeomorphic to a strip $\R \times [0,1]$. Let the seams of the patch $S_k$ be denoted $\{(l_-,f_-),(k,e_-)\}$ and $\{(k,e_+),(l_+,f_+)\}$; one of these may possibly be a true boundary component. Suppose $\mathbf{L}$ is a collection of Lagrangian boundary conditions for $\underline{S}$, and suppose that the correspondences
\[
	L_- = L_{(l_-,f_-),(k,e_-)} \subset M_{l_-}^- \times M_k, \quad\quad L_+ = L_{(k,e_+),(l_+,f_+)} \subset M_k^- \times M_{l_+}
\]
have embedded geometric composition $L_- \circ L_+$. Write $\underline{S}'$ for the quilted surface with strip-like ends obtained from $\underline{S}$ by removing the strip $S_k$ and replacing it with the seam $\{(l_-,f_-),(l_+,f_+)\}$. Take Lagrangian boundary conditions $\mathbf{L}'$ for $\underline{S}'$ which are identical to $\mathbf{L}$ for the seams and boundary components that $\underline{S}$ and $\underline{S}'$ have in common, and $L_{(l_-,f_-),(l_+,f_+)} = L_- \circ L_+$. The process of obtaining $(\underline{S}', \mathbf{L}')$ from $(\underline{S}, \mathbf{L})$ is referred to as {\bf strip shrinking}.

\begin{thm}
\textup{(Strip Shrinking for Quilted Surfaces)} For quilted surfaces $\underline{S}$ and $\underline{S}'$ as described above, there is a $\delta > 0$ such that if the width of the strip $S_k$ is less than $\delta$ (with respect to the conformal structure on the domain), then there is an identification $\calM(\underline{S})_0 = \calM(\underline{S}')_0$ via strip shrinking, and the induced isomorphisms $\Psi_{\underline{e}}: \CF(\underline{L}_{\underline{e}}) \longrightarrow \CF(\underline{L}_{\underline{e}}^\prime)$ intertwine the relative invariants of the quilted surfaces, up to a degree shift:
\[
	\Phi_{\underline{S}^\prime} \circ \left( \bigotimes_{\underline{e} \in \mathcal{E}_-(\underline{S})} \Psi_{\underline{e}} \right) = \left( \bigotimes_{\underline{e} \in \mathcal{E}_+(\underline{S}')} \Psi_{\underline{e}} \right) \circ \Phi_{\underline{S}}[d_k n_k].
\]
Here $d_k$ is the number of incoming ends of $S_k$ minus the number of outgoing ends of $S_k$, and $n_k$ is half the dimension of the symplectic manifold $M_k$ associated to $S_k$.
\label{thm:stripshrinking}
\end{thm}

\begin{figure}[h]
	\centering
	\includegraphics[scale=.125]{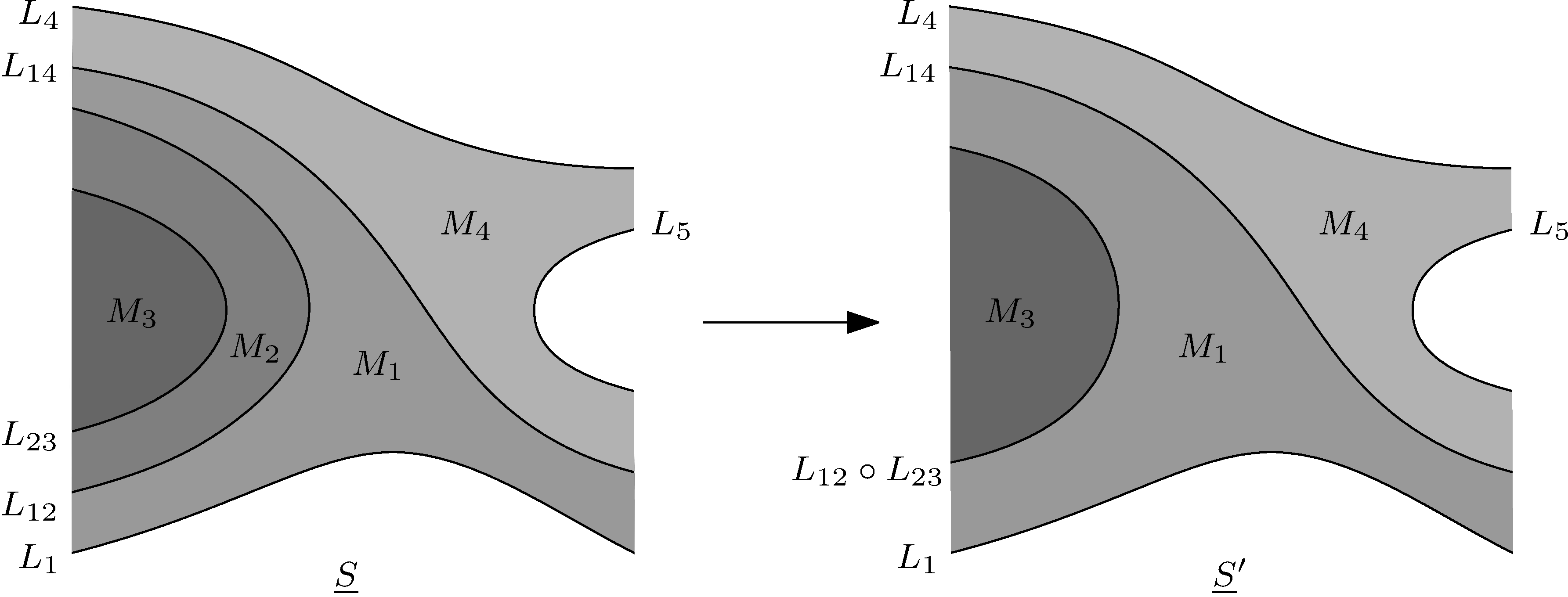}
	\caption{The process of strip shrinking in a quilted surface with strip-like ends.}
	\label{fig:stripshrinking}
\end{figure}

\subsection{Connected Sums Revisited}

% \begin{figure}[h]
% 	\includegraphics[scale=1.25]{seamswap.eps}
% 	\caption{Quilt defining the seam swap map $\mathfrak{S}^{L_0,L_1}_{L_0^\prime, L_1^\prime}$.}
% 	\label{fig:seamswap}
% \end{figure}

We first make a general definition. Suppose we have two sequences of Lagrangian correspondences
\[
	M_0 \xrightarrow{~L_0~} M_{01} \xrightarrow{~L_1~} M_1, \quad\quad M_0 \xrightarrow{~L_0^\prime~} M_{01}^\prime \xrightarrow{~L_1^\prime~} M_1,
\]
each with embedded composition, satisfying
\begin{equation}
	\label{eqn:equal-comp}
	L_0 \circ L_1 = L_0^\prime \circ L_1^\prime.
\end{equation}
We then have a strip-shrinking map
\[
	\Phi_{L_0, L_1}: \CF(L_0, L_1) \longrightarrow \CF(L_0 \circ L_1)
\]
and a ``reverse'' strip-shrinking map
\[
	\Psi_{L_0^\prime, L_1^\prime}: \CF(L_0^\prime \circ L_1^\prime) \longrightarrow \CF(L_0^\prime, L_1^\prime).
\]
In light of the condition (\ref{eqn:equal-comp}), these two maps are composable, and we define the {\bf seam swap} map $\mathfrak{S}^{L_0,L_1}_{L_0^\prime,L_1^\prime}$ by 
\[
	\mathfrak{S}^{L_0,L_1}_{L_0^\prime,L_1^\prime} = \Psi_{L_0^\prime,L_1^\prime} \circ \Phi_{L_0,L_1}: \CF(L_0, L_1) \longrightarrow \CF(L_0^\prime, L_1^\prime).
\]
By Theorem \ref{thm:stripshrinking}, the seam swap map is a chain isomorphism that respects relative invariants of quilted surfaces.

The proof that $\CSI(Y \# Y') \cong \CSI(Y) \otimes \CSI(Y')$ in \cite[Theorem 9.1]{horton1} can in fact be rephrased in terms of these seam swap maps. Given Heegaard diagrams $(\Sigma_g, \bfalpha, \bfbeta)$ for $Y$ and $(\Sigma_{g^\prime},\bfalpha',\bfbeta')$ for $Y'$, we have a Heegaard diagram $(\Sigma_{g+g^\prime}, \bfalpha \cup \bfalpha', \bfbeta \cup \bfbeta')$ for $Y \# Y'$. Consider the Lagrangian correspondences
\[
	L_\alpha = \left\{ [\rho] \in \scrR_{g,3} \left| \begin{array}{ll} \rho(\alpha_k) = I, & k = 1, \dots, g \\ \rho(\mu_1\mu_2\mu_3) = I & \end{array} \right\}\right.,
\]
\[
	L_{\alpha^\prime} = \left\{ [\rho] \in \scrR_{g^\prime,3} \left| \begin{array}{ll} \rho(\alpha_k^\prime) = I, & k = 1, \dots, g^\prime \\ \rho(\mu_1\mu_2\mu_3) = I & \end{array} \right\}\right.,
\]
\[
	L_{\alpha^\prime}^\prime = \left\{ ([\rho_1], [\rho_2]) \in \scrR_{g,3} \times \scrR_{g+g^\prime,3} \left| \begin{array}{ll} \rho_1(\alpha_k) = \rho_2(\alpha_k), & k = 1, \dots, g \\ \rho_1(\beta_k) = \rho_2(\beta_k), & k = 1, \dots, g \\ \rho_1(\mu_k) = \rho_2(\mu_k), & k = 1, 2, 3 \\ \rho_2(\alpha^\prime_k) = I, & k = 1, \dots, g' \end{array} \right\} \right.,
\]
\[
	L_{\alpha}^\prime = \left\{ ([\rho_1], [\rho_2]) \in \scrR_{g^\prime,3} \times \scrR_{g+g^\prime,3} \left| \begin{array}{ll} \rho_1(\alpha_k^\prime) = \rho_2(\alpha_k^\prime), & k = 1, \dots, g^\prime \\ \rho_1(\beta_k^\prime) = \rho_2(\beta_k^\prime), & k = 1, \dots, g^\prime \\ \rho_1(\mu_k) = \rho_2(\mu_k), & k = 1, 2, 3 \\ \rho_2(\alpha_k) = I, & k = 1, \dots, g \end{array} \right\} \right.,
\]
and similarly define $L_\beta$, $L_{\beta^\prime}$, $L_{\beta^\prime}^\prime$, and $L_{\beta}^\prime$; these just correspond to attaching handles in different orders. The key observation is that
\[
	L_\alpha \circ L_{\alpha^\prime}^\prime = L_{\alpha \cup \alpha^\prime} = L_{\alpha^\prime} \circ L_{\alpha}^\prime,
\]
\[
	L_\beta \circ L_{\beta^\prime}^\prime = L_{\beta \cup \beta^\prime} = L_{\beta^\prime} \circ L_{\beta}^\prime,
\]
and
\[
	L_{\alpha^\prime}^\prime \circ L_\beta^\prime = L_\beta \circ L_{\alpha^\prime}.
\]
Therefore we may use seam swap maps to see that
\begin{align*}
	\CF(L_{\alpha\cup\alpha^\prime},L_{\beta\cup\beta^\prime}) & \cong \CF(L_\alpha, L_{\alpha^\prime}^\prime, L_{\beta}^\prime, L_{\beta^\prime}) \\
	& \cong \CF(L_\alpha, L_{\beta}, L_{\alpha^\prime}, L_{\beta^\prime}) \\
	& \cong \CF(L_\alpha, L_\beta) \otimes \CF(L_{\alpha^\prime}, L_{\beta^\prime}),
\end{align*}
where the last identification is due to the fact that $\scrR_{0,3} = \text{pt}$. This same use of seam swap maps will appear in the proof of the K\"unneth principle.

% \begin{figure}[h]
% 	\includegraphics[scale=1]{excision.eps}
% 	\caption{Decomposing the Heegaard splitting of a connected sum through a cobordism.}
% 	\label{fig:excision}
% \end{figure}

\subsection{On $(\#^n S^3) \# (\#^{g-n} S^2 \times S^1)$}

\label{sect:S3-S2xS1}

Let $\calH = (\Sigma_g, \bfbeta, \bfgamma)$ denote the standard Heegaard diagram for $(\#^n S^3) \# (\#^{g-n} S^2 \times S^1)$, with the first $n$ $\beta$- and $\gamma$-curves corresponding to the $S^3$ summands. Write $L_\beta^\prime$ (resp. $L_\gamma^\prime$) for the Lagrangian in $\scrR_{n,3}$ corresponding to the first $n$ $\beta$- (resp. $\gamma$-) handle attachments. Let $\underline{L}_{\beta\gamma}$ denote the Lagrangian correspondence $\scrR_{n, 3} \longrightarrow \scrR_{g,3}$ corresponding to attaching the remaining $g - n$ handles (which are the same for $\beta$ and $\gamma$).

For the Heegaard diagram $\calH$ above, we have $L_\beta \cap L_{\gamma} \cong (S^3)^{g-n}$ and this intersection is clean. There is a Hamiltonian isotopy $\varphi$ of $\scrR_{g,3}$ taking $L_\gamma$ to another Lagrangian $\tilde{L}_\gamma$ such that $L_\beta \cap \tilde{L}_\gamma$ consists of $2^{g-n}$ points, and in fact this isotopy can be chosen so that
\[
	L_\beta \cap \tilde{L}_\gamma = \{[I, \dots, I, I, (-1)^{\epsilon_1}I, \dots, (-1)^{\epsilon_{g-n-1}} I, (-1)^{\epsilon_{g-n}} I] \in \scrR_{g,3} \mid (\epsilon_1, \dots, \epsilon_{g-n}) \in \{0,1\}^{g-n}\}.
\]
Under the isomorphism of unital algebras $\CF(L_\beta, \tilde{L}_\gamma) \cong \SI(\#^{g-n} S^2 \times S^1) \cong H^{3 - \ast}(S^3)^{\otimes (g - n)}$, the element $\Theta_{\beta\gamma} \in \CF(L_\beta, \tilde{L}_\gamma)$ corresponding to the intersection point with $(\epsilon_1, \dots, \epsilon_{g-n}) = (1, \dots, 1)$ maps to the unit of $H^{3 - \ast}(S^3)^{\otimes (g - n)}$.

\begin{figure}
	\centering
	\includegraphics[scale=1.25]{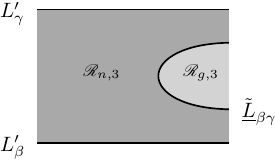}
	\caption{The quilt defining the map $\CF(L_\beta^\prime, L_\gamma^\prime) \longrightarrow \CF(L_\beta^\prime, \underline{L}_{\beta\gamma}^T, \underline{L}_{\beta\gamma}, L_\gamma^\prime): \theta \mapsto \theta \times \Theta_{\beta\gamma}$.}
	\label{fig:S3-S2xS1}
\end{figure}

Now, write $\tilde{\underline{L}}_{\beta\gamma}$ for the image of the Lagrangian correspondence $\underline{L}_{\beta\gamma}$ under the Hamiltonian isotopy $\id \times \varphi$ of $\scrR_{n,3} \times \scrR_{g,3}$ (with $\varphi$ as in the previous paragraph), and consider the quilt map pictured in Figure \ref{fig:S3-S2xS1}. Because $\Theta_{\beta\gamma}$ is the unit of $\CF(L_\beta, \tilde{L}_\gamma) \cong \CF(L_\beta^\prime, \underline{L}_{\beta\gamma}^T, \underline{L}_{\beta\gamma}, L_\gamma^\prime)$, it is easy to see that the relative invariant defined by this quilt is simply the map $\theta \mapsto \theta \times \Theta_{\beta\gamma}$, where $\theta$ is the generator of $\CF(L_\beta^\prime, L_\gamma^\prime) \cong \Z$.

More generally, if $L_\beta^\prime \cap L_\gamma^\prime = \{[\rho_0]\}$ for some representation $\rho_0$ (not necessarily the trivial representation), then the quilt map of Figure \ref{fig:S3-S2xS1} is given by $\rho_0 \mapsto \rho_0 \times \Theta_{\rho_0}$, where in this case we define $\Theta_{\rho_0}$ similarly to $\Theta_{\beta\gamma}$, except that on $\Sigma_n \setminus \text{pt} \subset \Sigma_g$, it should agree with $\rho_0$ instead of the trivial representation.

\subsection{The K\"unneth Principle}

We are now in a position to precisely state and prove the so-called K\"unneth principle for triangles.

Let $L_\beta$, $\tilde{L}_\gamma$, $L_\beta^\prime$, $L_\gamma^\prime$, and $\tilde{\underline{L}}_{\beta\gamma}$ be as in the previous subsection, and suppose we have another set of attaching curves in $\Sigma_g$, $\bfalpha$, with associated Lagrangian $L_\alpha$. We furthermore assume that $\bfalpha = (\alpha_1, \dots, \alpha_g)$ is ordered such that each of $\alpha_1, \dots, \alpha_n$ lie in the first $n$ handles of $\Sigma_g$, and each of $\alpha_{n+1}, \dots, \alpha_g$ lie in the last $g-n$ handles of $\Sigma_g$, so that the $\alpha$-curves naturally respect the direct sum decomposition $\Sigma_g = \Sigma_n \# \Sigma_{g-n}$.

We may define several more Lagrangian correspondences from our various sets of attaching curves, depending on what order we make the handle attachments in:
\[
	L_{\alpha;g-n}^\prime: \scrR_{0,3} \longrightarrow \scrR_{g-n,3},
\]
\[
	L_{\alpha;n}^\prime: \scrR_{0,3} \longrightarrow \scrR_{n,3},
\]
\[
	L_{\alpha;n}^{\prime\prime}: \scrR_{g-n,3} \longrightarrow \scrR_{g,3},
\]
\[
	\tilde{\underline{L}}_{\beta\gamma}^\prime: \scrR_{g-n,3} \longrightarrow \scrR_{0,3},
\]

\begin{thm}
\label{thm:kunneth-triangle}
\textup{(K\"unneth Principle for Triangle Maps)} The triangle map
\[
	\mu_2^{\alpha\beta\gamma}(\cdot, \Theta_{\beta\gamma}): \CF(L_\alpha, L_\beta) \longrightarrow \CF(L_\alpha, \tilde{L}_\gamma)
\]
corrresponds exactly to the map
\[
	\id \otimes \mu_2^{\alpha\beta\gamma}(\cdot, \theta): \CF(L_{\alpha;g-n}^{\prime}, \tilde{L}_{\beta\gamma}^{\prime}) \otimes \CF(L_{\alpha;n}^\prime, L_\beta^\prime) \longrightarrow \CF(L_{\alpha;g-n}^{\prime}, \tilde{L}_{\beta\gamma}^{\prime}) \otimes \CF(L_{\alpha;n}^\prime, L_\gamma^\prime).
\]
\end{thm}

\begin{proof}
We start by manipulating the triangles counted by $\mu_2^{\alpha\beta\gamma}(\cdot, \Theta_{\beta\gamma})$ as suggested by Figure \ref{fig:kunneth-principle}.

\begin{figure}
	\centering
	\includegraphics{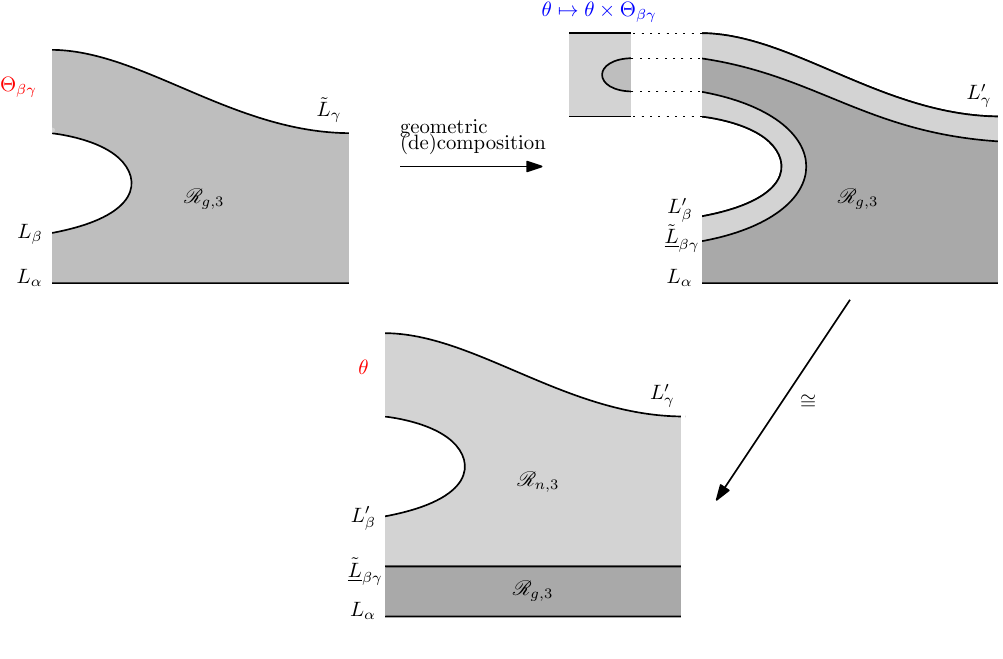}
	\caption{Quilt manipulations used to derive the K\"unneth principle for ``local'' computation of triangle maps.}
	\label{fig:kunneth-principle}
\end{figure}

Then perform a reverse strip-shrinking to convert $L_\alpha = L_{\alpha;g-n}^\prime \circ L_{\alpha;n}^{\prime\prime}$ into $(L_{\alpha;g-n}^\prime, L_{\alpha;n}^{\prime\prime})$. It is easy to see that $L_{\alpha;n}^{\prime\prime} \circ \tilde{\underline{L}}_{\beta\gamma} = \tilde{\underline{L}}_{\beta\gamma}^\prime \circ L_{\alpha;n}^\prime$, so that we may further apply a seam swap $\mathfrak{S}^{L_{\alpha;n}^{\prime\prime},\tilde{\underline{L}}_{\beta\gamma}}_{\tilde{\underline{L}}_{\beta\gamma}^\prime, L_{\alpha;n}^\prime}$. Since (reverse) strip-shrinking and seam swaps identify relative invariants of quilted surfaces, we see that the triangles counted by $\mu_2^{\alpha\beta\gamma}(\cdot, \Theta_{\beta\gamma})$ correspond to the relative invariant of the surface in Figure \ref{fig:kunneth-principle2}. Since $\scrR_{0,3} = \text{pt}$, this relative invariant is clearly $\id \otimes \mu_2^{\alpha^\prime \beta^\prime \gamma^\prime}(\cdot, \theta)$, so we are done.
\end{proof}

\begin{figure}
	\centering
	\includegraphics{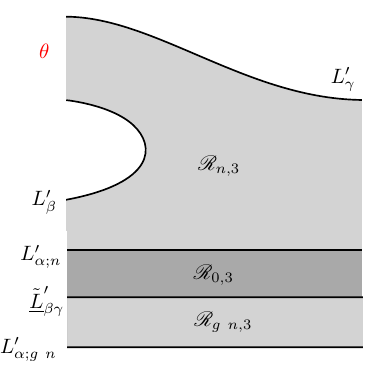}
	\caption{Result of reverse strip-shrinking and seam swapping.}
	\label{fig:kunneth-principle2}
\end{figure}

There is also a straightforward generalization of the K\"unneth Principle to certain polygon maps, which we state later as Theorem \ref{thm:kunneth-polygon} when the correct context has been established.
\section{Cobordisms and Functoriality}

\label{sect:cobordisms}

It is well-known that $3$- and $4$-manifold invariants which are gauge-theoretic in nature should fit into the framework of a topological quantum field theory. In this chapter, we explore part of that philosophy by defining maps between symplectic instanton groups induced by compact, connected cobordisms between closed, connected, oriented $3$-manifolds. We follow the approach of Ozsv\'ath-Szab\'o \cite{oz-sz-4mfd} and define the cobordism maps via handle decompositions and triangle maps. In this chapter, all homologies are taken with $\F_2$ coefficients.

\subsection{$1$-Handle Cobordisms}

Let $W_1: Y \longrightarrow Y'$ be a $4$-dimensional oriented cobordism corresponding to attaching a $1$-handle to $Y$. Then $Y' \cong Y \# (S^2 \times S^1)$. Fix a Heegaard diagram $\calH$ for $Y$ and consider the Heegaard diagram $\calH'$ for $Y' \cong Y \# (S^2 \times S^1)$ given by taking a connected sum of the diagram $\calH$ with the standard genus $1$ diagram $\calH_0$ for $S^2 \times S^1$. Then
\[
	\CSI(\calH') \cong \CSI(\calH) \otimes \CSI(\calH_0) \cong \CSI(\calH) \otimes H^\ast(S^3),
\]
by the K\"unneth principle and \cite[Proposition 9.5]{horton1}. Let $\Theta \in \CSI(\calH_0)$ be the intersection point corresponding to the generator of $H^0(S^3)$ (according to the isomorphism of \cite[Proposition 9.5]{horton1}); in particular, $\Theta$ is the unit for the triangle product on $\CSI(\calH_0)$. Then we define the chain map $\CSI(W_1)$ induced by the $1$-handle cobordism $W_1$ by
\[
	\CSI(W_1): \CSI(\calH) \longrightarrow \CSI(\calH'),
\]
\[
	\CSI(W_1)(\xi) = \xi \otimes \Theta.
\]
Write $\SI(W_1): \SI(Y) \longrightarrow \SI(Y')$ for the induced map in homology. $\SI(W_1)$ can equivalently be described as the relative invariant of the surface with strip-like ends in Figure \ref{fig:1-handle}.

\begin{figure}[h]
\begin{center}
	\includegraphics[scale=1.2]{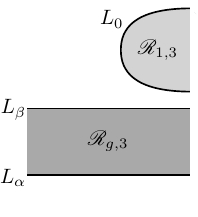}
	\caption{Surface with strip-like ends whose relative invariant is the $1$-handle map $\SI(W_1): \SI(Y) \longrightarrow \SI(Y) \otimes \SI(S^2 \times S^1)$.}
	\label{fig:1-handle}
\end{center}
\end{figure}

The first main result of this section is as follows:

\begin{thm}
\label{thm:1handlemap}
The $1$-handle attachment map $\SI(W_1)$ depends only on $Y$ in the following sense. If $\calH_1$ and $\calH_2$ are two pointed Heegaard diagrams for $Y$ differing by a Heegaard move, then there is a commutative diagram
\[
\xymatrix@C=2cm{
	\SI(\calH_1) \ar[r]^{\SI(W_1)} \ar[d]_{\Psi} & \SI(\calH_1^\prime) \ar[d]^{\Psi'} \\
	\SI(\calH_2) \ar[r]_{\SI(W_1)} & \SI(\calH_2^\prime)
}
\]
where $\Psi$ (respectively $\Psi'$) are the isomorphisms of symplectic instanton homologies induced by the Heegaard move (as in \cite[Section 5]{horton1}).
\end{thm}

\begin{proof}
Since isotopies of attaching curves and handleslides induce the identity on symplectic instanton homology, we need only check the result for when $\calH_2$ is obtained from $\calH_1$ via stabilization. In terms of quilted Floer homology, the relevant diagram we want to commute is
\[
\xymatrix{
	\HF(L_{\alpha_1}, L_{\beta_1}) \ar[d]_{\Psi} \ar[r] & \HF(L_{\alpha_1^\prime}, L_{\beta_2^\prime}) \ar[d]^{\Psi} \\
	\HF(L_{\alpha_1}, L_{\alpha_1\alpha_2}, L_{\beta_2\beta_1}, L_{\beta_1}) \ar[r] & \HF(L_{\alpha_1^\prime}, L_{\alpha_1^\prime\alpha_2^\prime}, L_{\beta_2^\prime\beta_2^\prime}, L_{\beta_2^\prime})
}
\]
where the vertical maps are inverses of strip shrinking maps and the horizontal maps are relative invariants of (quilted) triangles. By Theorem \ref{thm:stripshrinking}, this square commutes (without any grading shift since the strip we shrink has one incoming and one outgoing end).
\end{proof}

This verifies that a cobordism consisting of a single $1$-handle induces a map on symplectic instanton homology independent of the choice of Heegaard diagram. When $W_1: Y \longrightarrow Y'$ consists of $n$ $1$-handles $H_1, \dots, H_n$, we have that $Y' \cong Y \# n(S^2 \times S^1)$. We may then take $\calH' = \calH \# \calH_0 \# \cdots \# \calH_0$ as a Heegaard diagram for $Y'$, and identify
\[
	\CSI_\ast(\calH') \cong \CSI_\ast(\calH) \otimes H^{3-\ast}(S^3)^{\otimes n}.
\]
In this case, the chain map induced by the $1$-handle cobordism is
\[
	\CSI(W_1): \CSI(\calH) \longrightarrow \CSI(\calH'),
\]
\[
	\CSI(W_1)(\xi) = \xi \otimes \Theta^{\otimes n},
\]
and we again denote the induced map on homology by $\SI(W_1): \SI(Y) \longrightarrow \SI(Y')$. We should think of $\SI(W_1)$ as a composition of maps for adding a single $1$-handle; as such, we should check that it does not depend on the order in which the handles are added. Furthermore, we would like to verify that $\SI(W_1)$ is actually an invariant of $W_1$, so we should make sure it is also invariant under handleslides of the $1$-handles.

\begin{thm}
The $1$-handle cobordism map $\SI(W_1): \SI(Y) \longrightarrow \SI(Y')$ is invariant under the reordering of the $1$-handles of $W$ and handleslides among them.
\end{thm}

\begin{proof}
Invariance of the ordering is clear, since $\calH'$ is the same Heegaard diagram no matter what order the handles are added in. Handleslides of the $1$-handles in $W_1$ do not affect $\calH'$ either, so $\SI(W_1)$ is left unchanged.
\end{proof}

\subsection{$3$-Handle Cobordisms}

If $W_3: Y' \longrightarrow Y$ is a cobordism induced by attaching a single $3$-handle along some non-separating $2$-sphere in $Y'$, then $Y' \cong Y \# (S^2 \times S^1)$. There is a compatible Heegaard diagram for $Y'$ induced by the attaching $2$-sphere:

\begin{prop}
\label{prop:3handlediagram}
\textup{(Lemma 4.11 of \cite{oz-sz-4mfd})} A non-separating $2$-sphere in a $3$-manifold $Y'$ induces a split Heegaard diagram $\calH' = \calH \# \calH_0$ for $Y'$, where $\calH$ is a Heegaard diagram for the result of surgery along the $2$-sphere and $\calH_0$ is the standard diagram for $S^2 \times S^1$.
\end{prop}

Using such a diagram, we may define on the chain level a map induced by addition of a single $3$-handle:
\[
	\CSI(W_3): \CSI(\calH \# \calH_0) \longrightarrow \CSI(\calH),
\]
\[
	\CSI(W_3)(\xi \otimes \eta) = \begin{cases} \xi, & \text{if $\eta$ is the non-unit generator $\theta \in \CSI_\ast(\calH_0) \cong H^{3-\ast}(S^3)$,} \\ 0, & \text{otherwise.} \end{cases}
\]
$\SI(W_3): \SI(Y') \longrightarrow \SI(Y)$ will denote the induced map on homology. $\SI(W_3)$ can equivalently be described as the relative invariant of the surface with strip-like ends in Figure \ref{fig:3-handle}.

\begin{figure}[h]
\begin{center}
	\includegraphics[scale=1.2]{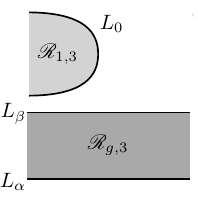}
	\caption{Surface with strip-like ends whose relative invariant is the $3$-handle map $\SI(W_3): \SI(Y) \otimes \SI(S^2 \times S^1) \longrightarrow \SI(Y)$.}
	\label{fig:3-handle}
\end{center}
\end{figure}

By an argument similar to that of Theorem \ref{thm:1handlemap}, we have the following:

\begin{thm}
The $3$-handle attachment map $\SI(W_3)$ depends only on $Y$ in the following sense. If $\calH_1$ and $\calH_2$ are two pointed Heegaard diagrams for $Y$ differing by a Heegaard move, then there is a commutative diagram
\[
\xymatrix@C=2cm{
	\SI(\calH_1 \# \calH_0) \ar[r]^{~\SI(W_3)~} \ar[d]_{\Psi} & \SI(\calH_1) \ar[d]^{\Psi'} \\
	\SI(\calH_2 \# \calH_0) \ar[r]_{~\SI(W_3)~} & \SI(\calH_2)
}
\]
where $\Psi$ (respectively $\Psi'$) are the isomorphisms of symplectic instanton homologies induced by the Heegaard move (as in \cite[Section 5]{horton1}).
\end{thm}

When the cobordism $W_3: Y' \longrightarrow Y$ consists of $n$ $3$-handles, $Y' \cong Y \# n(S^2 \times S^1)$ we can apply Proposition \ref{prop:3handlediagram} iteratively to get a Heegaard diagram $\calH \# \calH_0 \# \cdots \# \calH_0$ for $Y'$, from which we can define the cobordism map
\[
	\CSI(W_3): \CSI(\calH \# n\calH_0) \longrightarrow \CSI(\calH),
\]
\[
	\CSI(W_3)(\xi \otimes \eta) = \begin{cases} \xi, & \text{if $\eta$ is the generator $\theta^{\otimes n}$ of $\CSI_\ast(n\calH_0) \cong H^{3-\ast}(S^3)^{\otimes n}$,} \\ 0, & \text{otherwise.} \end{cases}
\]
Similar to the case of $1$-handles, the induced map on homology, $\SI(W_3): \SI(Y') \longrightarrow \SI(Y)$ is actually an invariant of the cobordism:

\begin{thm}
The $3$-handle cobordism map $\SI(W_3): \SI(Y') \longrightarrow \SI(Y)$ is invariant under the reordering of the $3$-handles of $W$ and handleslides among them.
\end{thm}

\subsection{$2$-Handle Cobordisms}

\label{sect:2-handles}

The situation for $2$-handle cobordisms is much more interesting. Recall that $4$-dimensional $2$-handles are attached along framed links $\bbL = \amalg_{i = 1}^n (L_i, \lambda_i)$, where the $L_i$ are the connected components of the link and $\lambda_i$ is a choice of longitude for $L_i$ (a homology class $\lambda_i \in H_1(\partial \mathrm{nbd}(L_i))$ with $\mu_i \cdot \lambda_i = 1$, where $\mu_i \in H_1(\partial \mathrm{nbd}(L_i))$ is the meridian of $L_i$).

For a framed link $\bbL \subset Y$, let $Y(\bbL)$ denote the result of surgery on $\bbL$ and $W(\bbL): Y \longrightarrow Y(\bbL)$ denote the trace of this surgery. One would like to have a Heegaard diagram relating $Y$ and $Y(\bbL)$; to do this, certain choices must be made.

\begin{defn}
A {\bf bouquet} for the framed link $\bbL \subset Y$ is a $1$-complex $B(\bbL)$ embedded in $Y$ with
\begin{itemize}
	\item $n + 1$ $0$-cells given by a basepoint $y_0 \in Y \setminus \amalg L_i$ and basepoints $y_i \in L_i$.
	\item $2n$ $1$-cells given by the $L_i$ and $n$ paths $\delta_i \subset Y$ satisfying $\delta_i(0) = y_0$, $\delta_i(1) = y_i$, and $\delta_i([0,1)) \cap \amalg L_j = \varnothing$.
\end{itemize}
\end{defn}

Clearly a regular neighborhood of a bouquet $B(\bbL)$ is a genus $n$-handlebody and $\amalg L_i$ is unknotted inside this handlebody. This handlebody may not give a Heegaard splitting of $Y$, but there will be some genus $g \geq n$ Heegaard splitting of $Y$ with one of the handlebodies containing this regular neighborhood. Hence we introduce the following definition.

\begin{defn}
A Heegaard triple $(\Sigma_g, \bfalpha, \bfbeta, \bfgamma, z)$ is said to be {\bf subordinate to the bouquet} $B(\bbL)$ if the following conditions are satisfied:
\begin{itemize}
	\item Attaching $2$-handles along $\{\alpha_i\}_{i = 1}^g$ and $\{\beta_i\}_{i = n+1}^g$ gives the complement of $B(\bbL)$ in $Y$.
	\item $\gamma_i = \beta_i$ for $i = n+1, \dots, g$.
	\item After surgering out $\beta_{n+1}, \dots, \beta_g$, both $\beta_i$ and $\gamma_i$ lie in the obvious punctured torus $T_i \subset \Sigma_g$ corresponding to $L_i$ for $i = 1, \dots, n$.
	\item For $i = 1, \dots, n$ the $\beta_i$ are meridians for $L_i$ and the $\gamma_i$ are the longitudes of $L_i$ specified by $\lambda_i$.
\end{itemize}
\end{defn}

Note that for such a Heegaard triple, $\calH_{\alpha\beta} = (\Sigma_g, \bfalpha, \bfbeta, z)$ is a Heegaard diagram for $Y$, $\calH_{\alpha\gamma} = (\Sigma_g, \bfalpha, \bfgamma, z)$ is a Heegaard diagram for $Y(\bbL)$, and $\calH_{\beta\gamma} = (\Sigma_g, \bfbeta, \bfgamma, z)$ is a Heegaard diagram for $\#^{g-n} (S^2 \times S^1)$. More specifically, we have the following.

\begin{prop}
\textup{(Proposition 4.3 of \cite{oz-sz-4mfd})} The $4$-manifold $X_{\alpha\beta\gamma}$ described by a Heegaard triple $(\Sigma_g, \bfalpha, \bfbeta, \bfgamma, z)$ subordinate to a bouquet $B(\bbL)$ has boundary $-Y \amalg \#^{g - n} (S^2 \times S^1) \amalg Y(\bbL)$. Filling in the $\#^{g-n} (S^2 \times S^1)$ boundary component gives the $2$-handle cobordism $W(\bbL)$:
\[
	W(\bbL) \cong X_{\alpha\beta\gamma} \cup \natural^{g-n} (D^3 \times S^1).
\]
\end{prop}

The above proposition suggests that we may use a triangle map associated to $(\Sigma_g, \bfalpha, \bfbeta, \bfgamma, z)$ in order to define a $2$-handle cobordism map associated to $W(\bbL)$. We define
\[
	f_{B(\bbL)}: \CSI(\calH_{\alpha\beta}) \longrightarrow \CSI(\calH_{\alpha\gamma}),
\]
\[
	f_{B(\bbL)}(\xi) = \mu_2^{\alpha\beta\gamma}(\xi, \Theta^{\otimes (g-n)}),
\]
where as usual $\Theta^{\otimes(g-n)}$ is the element of top degree in $\SI(\#^{g-n} (S^2 \times S^1)) \cong H^{3-\ast}(S^3)^{\otimes(g-n)}$. Since $\Theta^{\otimes(g-n)}$ is a cycle, $g_{\bbL}$ is a chain map and we get a map $F_{\bbL}: \SI(Y) \longrightarrow \SI(Y(\bbL))$ on homology.

We must justify the notation $F_{\bbL}$ by showing that this map depends only on the framed link $\bbL$. There are two levels of choices in the construction: first we pick a bouquet $B(\bbL)$ for $\bbL$, and then we choose a Heegaard triple subordinate to $B(\bbL)$. As a first step, we describe the difference between two Heegaard triples subordinate to the same bouquet.

\begin{lem}
\label{lem:bouquetdiagram}
\textup{(Lemma 4.5 of \cite{oz-sz-4mfd})} Let $\bbL$ be a framed link in a closed, oriented $3$-manifold $Y$. For a fixed bouquet $B(\bbL)$, any two Heegaard triples subordinate to $B(\bbL)$ are related by a sequence of the following moves:
\begin{itemize}
	\item Isotopies and handleslides amongst the $\bfalpha$-curves.
	\item Simultaneous isotopies and handleslides amongst the curves $\beta_{n+1}, \dots, \beta_g, \gamma_{n+1}, \dots, \gamma_g$.
	\item Isotopies and handleslides of the $\beta_i$, $1 \leq i \leq n$, over the $\beta_j$, $n + 1 \leq j \leq g$.
	\item Isotopies and handleslides of the $\gamma_i$, $1 \leq i \leq n$, over the $\gamma_j$, $n + 1 \leq j \leq g$.
	\item ``Stabilizations'' introducing the usual stabilization curves $\alpha_{g+1}$ and $\beta_{g+1}$ along with a $\bfgamma$-curve $\gamma_{g+1} = \beta_{g+1}$.
\end{itemize}
\end{lem}

With the above in place, we can prove independence of the subordinate triple.

\begin{lem}
\label{lem:indep-subdiag}
For a fixed bouquet $B(\bbL)$, the $2$-handle cobordism map $F_{\bbL}: \SI(Y) \longrightarrow \SI(Y(\bbL))$ is independent of the choice of Heegaard triple subordinate to $B(\bbL)$ in the following sense: If $\calH = (\Sigma_g, \bfalpha, \bfbeta, \bfgamma, z)$ and $\calH' = (\Sigma_{g^\prime}, \bfalpha', \bfbeta', \bfgamma', z')$ are two Heegaard triples subordinate to $B(\bbL)$, then there is a commutative diagram
\[
\xymatrix{
	\SI(\calH_{\alpha\beta}) \ar[r]^{F_{\bbL}} \ar[d]_{\Psi_1} & \SI(\calH_{\alpha\gamma}) \ar[d]^{\Psi_2} \\
	\SI(\calH_{\alpha^\prime \beta^\prime}^\prime) \ar[r]_{F_{\bbL}} & \SI(\calH_{\alpha^\prime \gamma^\prime}^\prime)
}
\]
where $\Psi_1$ and $\Psi_2$ are the isomorphisms induced by the Heegaard moves relating the respective pairs of Heegaard diagrams.
\end{lem}

\begin{proof}
We only need to check commutativity of the diagram for the moves listed in Lemma \ref{lem:bouquetdiagram}. The only such move where there is something to prove is the ``stabilization'' move. The argument is similar to that of the proof of Theorem \ref{thm:1handlemap}. In terms of quilted Floer homology, the square we are interested in is
\[
\xymatrix{
	\HF(L_\alpha, L_\beta) \ar[r] \ar[d]_{\Psi_1} & \HF(L_\alpha, L_\gamma) \ar[d]^{\Psi_2} \\
	\HF(L_\alpha, L_{\alpha\alpha^\prime}, L_{\beta^\prime\beta}, L_\beta) \ar[r] & \HF(L_\alpha, L_{\alpha\alpha^\prime}, L_{\gamma^\prime\gamma}, L_\gamma)
}
\]
where the vertical maps are inverses of strip shrinking maps and the horizontal maps are relative invariants of (quilted) triangles. By Theorem \ref{thm:stripshrinking}, this square commutes.
\end{proof}

\begin{lem}
\label{lem:indep-bouquet}
The $2$-handle cobordism map $F_{\bbL}: \SI(Y) \longrightarrow \SI(Y(\bbL))$ is independent of the choice of bouquet.
\end{lem}

\begin{proof}
It suffices to prove the independence in the case where $B(\bbL)$ and $B'(\bbL)$ are two bouquets differing only in the choice of the arc $\delta_1$ (as well as its terminal point $y_1$). But in this case, one can construct two Heegaard triples, $(\Sigma, \bfalpha, \bfbeta, \bfgamma, z)$ subordinate to $B(\bbL)$ and $(\Sigma', \bfalpha', \bfbeta', \bfgamma', z')$ subordinate to $B'(\bbL)$, such that $\bfalpha = \bfalpha'$, $\bfbeta'$ can be obtained from $\bfbeta$ via handleslides amongst the $\beta$-curves, and $\bfgamma'$ can be obtained from $\bfgamma$ via handleslides amongst the $\gamma$-curves (see the proof of Lemma 4.8 in \cite{oz-sz-4mfd}). Handleslides do not change the Lagrangians at all, so $F_{\bbL}$ is therefore independent of the choice of bouquet.
\end{proof}

Lemmas \ref{lem:indep-subdiag} and \ref{lem:indep-bouquet} immediately imply the following.

\begin{thm}
For any framed link $\bbL$ in a closed, oriented $3$-manifold $Y$, the $2$-handle cobordism map $F_{\bbL}: \SI(Y) \longrightarrow \SI(Y(\bbL))$ is independent of the choices of bouquet and subordinate Heegaard diagram used to define it.
\end{thm}

To show that $F_{\bbL}$ is actually an invariant of the $4$-manifold $W(\bbL)$, we need to check that it is invariant under handleslides and also that it is independent of the order in which we attach the $2$-handles. Invariance under handleslides is established as follows.

\begin{thm}
\label{lem:indep-2handleslide}
Let $\bbL'$ be a framed link obtained from a given framed link $\bbL$ by performing handleslides amongst the components. Then the $2$-handle cobordism maps $F_{\bbL}: \SI(Y) \longrightarrow \SI(Y(\bbL))$ and $F_{\bbL'}: \SI(Y) \longrightarrow \SI(Y(\bbL')) = \SI(Y(\bbL))$ are equal.
\end{thm}

\begin{proof}
It suffices to consider the case where $\bbL$ consists of two components $L_1, L_2$ and $\bbL' = L_1^\prime \cup L_2$ is the framed link resulting from a handleslide of $L_1$ over $L_2$. Let $\sigma: [0, 1] \longrightarrow Y$ be the arc used to define the handleslide, with $\sigma(0) \in L_1$. There is an obvious arc $\sigma': [0, 1] \longrightarrow Y$ with $\sigma'(0) = \sigma(0)$ and $\sigma'(1) \in L_1^\prime$. Now we have a bouquet $B(\bbL)$ for $\bbL$ formed from $L_1 \cup \sigma \cup L_2$ by considering $\sigma$ as a $1$-complex with $0$-simplices at $\sigma(0)$, $\sigma(\tfrac{1}{2})$, and $\sigma(1)$. We also have a bouquet $B(\bbL')$ for $\bbL'$ formed from $L_1^\prime \cup \sigma' \cup \sigma \cup L_2$. Writing $\calH = (\Sigma_g, \bfalpha, \bfbeta, \bfgamma, z)$ for the Heegaard triple subordinate to $B(\bbL)$ and $\calH' = (\Sigma_g, \bfalpha', \bfbeta', \bfgamma', z)$ for the Heegaard triple subordinate to $B(\bbL')$, it is easy to see that $\calH$ and $\calH'$ are identical aside from $\beta_1^\prime$ being a handleslide of $\beta_1$ over $\beta_2$ and $\gamma_1^\prime$ being a handleslide of $\gamma_1$ over $\gamma_2$. Therefore the moduli spaces and Lagrangians associated to $\calH$ and $\calH'$ are identical, from which it follows that $F_{\bbL}$ and $F_{\bbL'}$ are identical.
\end{proof}

To establish invariance under the ordering of the $2$-handles, we prove a composition property for the framed link invariants.

\begin{thm}
\label{lem:indep-2handleorder}
If a framed link $\bbL \subset Y$ can be decomposed as $\bbL_1 \amalg \bbL_2$, then we have that
\[
	F_{\bbL} = F_{Y(\bbL_1), \bbL_2} \circ F_{Y, \bbL_1}: \SI(Y) \longrightarrow \SI(Y(\bbL_1)) \longrightarrow \SI(Y(\bbL_1 \amalg \bbL_2)) = \SI(Y(\bbL)).
\]
\end{thm}

\begin{proof}
Fix a bouquet $B(\bbL)$ for $\bbL$ and let $\calH = (\Sigma_g, \bfalpha, \bfbeta, \bfgamma, z)$ be a Heegaard triple subordinate to $B(\bbL)$ such that $\gamma_1, \dots, \gamma_n$ are the framings for the $n$ components of $\bbL_1$. Define a fourth set of attaching curves $\bfdelta$ by setting
\[
	\delta_i = \begin{cases} \gamma_i, & i = 1, \dots, n, \\ \beta_i, & i = n+1, \dots, g. \end{cases}
\]
Note that $\calH_1 = (\Sigma_g, \bfalpha, \bfbeta, \bfdelta, z)$ is a Heegaard triple subordinate to the bouquet $B(\bbL_1)$ obtained from $B(\bbL)$ by throwing out the parts of the $1$-complex having to do with $\bbL_2$. There is a similar bouquet $B(\bbL_2)$ which can be considered as lying in $Y(\bbL_1)$, and the Heegaard triple $\calH_2 = (\Sigma_g, \bfalpha, \bfdelta, \bfgamma, z)$ is subordinate to this bouquet.

Using the bouquets $B(\bbL_1) \subset Y$ and $B(\bbL_2) \subset Y(\bbL_1)$, we have that
\[
	F_{Y,\bbL_1}(\xi) = \mu_2^{\alpha\beta\delta}(\xi \otimes \Theta_{\beta\delta}),
\]
\[
	F_{Y(\bbL_1), \bbL_2}(\xi) = \mu_2^{\alpha\delta\gamma}(\xi \otimes \Theta_{\delta\gamma}),
\]
from which we may compute
\begin{align*}
	F_{Y(\bbL_1), \bbL_2}(F_{Y, \bbL_1}(\xi)) & = \mu_2^{\alpha\delta\gamma}(\mu_2^{\alpha\beta\delta}(\xi \otimes \Theta_{\beta\delta}) \otimes \Theta_{\delta\gamma}) \\
	& = \mu_2^{\alpha\beta\gamma}(\xi \otimes \mu_2^{\beta\delta\gamma}(\Theta_{\beta\delta} \otimes \Theta_{\delta\gamma})) \\
	& = \mu_2^{\alpha\beta\gamma}(\xi \otimes \Theta_{\beta\gamma}) \\
	& = F_{\bbL}(\xi),
\end{align*}
where $\mu_2^{\beta\delta\gamma}(\Theta_{\beta\delta} \otimes \Theta_{\delta\gamma}) = \Theta_{\beta\gamma}$ by \cite[Proposition 9.5]{horton1}.
\end{proof}

\subsection{General Cobordisms and Invariance}

So far, we have only defined cobordism maps for cobordisms consisting entirely of handles of equal index. We wish to make a definition for an arbitrary cobordism. Let $W: Y \longrightarrow Y'$ be an arbitrary compact, connected cobordism of closed, connected, oriented $3$-manifolds and choose a handle decomposition of $W$ such that the handles are attached in order of increasing index. Hence we get a factorization of $W$,
\[
	W: Y \xrightarrow{~W_1~} Y_1 \xrightarrow{~W_2~} Y_2 \xrightarrow{~W_3~} Y',
\]
where each $W_k$ is a cobordism consisting entirely of $k$-handles. We wish to define the cobordism map $\SI(W): \SI(Y) \longrightarrow \SI(Y')$ via this decomposition of $W$ by using our previously defined cobordism maps:
\[
	\SI(W) = \SI(W_3) \circ \SI(W_2) \circ \SI(W_1).
\]
The goal of this section is to show that the map $\SI(W)$ is well-defined, and that it is actually a topological invariant of the $4$-manifold $W$.

Previously, we showed that $\SI(W_1)$, $\SI(W_2)$, and $\SI(W_3)$ are topological invariants of $W_1$, $W_2$, and $W_3$, respectively. Hence it only remains to inspect the behavior of these maps under Kirby moves involving handles of different indices, \emph{i.e.} cancellation of pairs of handles.

\begin{lem}
\label{lem:indep-cancel}
Let $W_1$ be a cobordism obtained by attaching a single $1$-handle to the closed, oriented $3$-manifold $Y$, and let $W_2$ be a cobordism obtained by attaching a $2$-handle to $Y \# (S^2 \times S^1)$ along a framed knot $\bbK$ such that the $2$-handle cancels the $1$-handle. Then $\SI(W_2) \circ \SI(W_1)$ is the identity.
\end{lem}

\begin{proof}
Given a Heegaard diagram $\calH = (\Sigma_g, \bfalpha, \bfbeta, z)$ for $Y$, we have
\[
	\calH_0 = (\Sigma_g, \bfalpha, \bfbeta, \bfbeta, z) \# (\Sigma_1, \alpha_0, \alpha_0, \gamma_0, z_0),
\]	
where $\alpha_0$ is the meridian of $\Sigma_1$ and $\gamma_0$ is the standard longitude of $\Sigma_1$, is a Heegaard triple representing a $2$-handle attachment along a framed knot $\bbK_0$ in $Y \# (S^2 \times S^1)$ that cancels $1$-handle addition. Note that the triangle count represented by $\calH^\prime$ represents the nearest point map $\SI(Y) \otimes \SI(S^2 \times S^1) \xrightarrow{~\cong~} \SI(Y)$, while the triangle count from $\calH_0$ represents the nearest point map $\SI(S^2 \times S^1) \otimes \SI(S^3) \xrightarrow{~\cong~} \SI(S^3)$. We may therefore compute
\[
	(\SI(W_2) \circ \SI(W_1))(\xi) = \SI(W_2)(\xi \otimes \Theta) = \xi \otimes \theta \in \SI(Y) \otimes \SI(S^3),
\]
where $\theta$ is the trivial representation in $\SI(S^3)$. This shows that $\SI(W_2) \circ \SI(W_1)$ is the map induced by stabilization of the original Heegaard diagram $\calH$.

The proof is not finished, because $\bbK_0$ is not the only framed knot such that $2$-handle attachment along it cancels the $1$-handle addition. Let $\bbK$ be an arbitrary such knot, and let $\calH' = (\Sigma_g, \bfalpha, \bfbeta, \bfbeta, z) \# (\Sigma_1, \alpha_0, \alpha_0, \delta_0, z_0)$ be the Heegaard triple corresponding to a $2$-handle addition along $\bbK$. The only difference between $\bbK$ and $\bbK_0$ is in the framings $\gamma_0$ and $\delta_0$. These two framing curves differ by a power of a Dehn twist about $\alpha_0$. The action of the Dehn twist $\tau_{\alpha_0}$ on $\Sigma_1$ induces a symplectic Dehn twist on $\scrR_{1,3}$ about the Lagrangian sphere
\[
	C_{\alpha_0} = \{[\rho] \in \scrR_{1,3} : \rho(\alpha) = -I\}
\]
(see \emph{e.g.} \cite[Theorem 3.8(b)]{fiberedtriangle}). Since $C_{\alpha_0} \cap L_{\alpha_0} = \varnothing$, the triangle count from $(\Sigma_1, \alpha_0, \alpha_0, \delta_0,z)$ is still the same in Floer homology; only the large area triangles in the count may change, but there are an even number of them. Hence this triple diagram represents the same closest point map $\SI(S^2 \times S^1) \otimes \SI(S^3) \xrightarrow{~\cong~} \SI(S^3)$ as $\calH_0$ does, so that
\[
	(\SI(W_2) \circ \SI(W_1))(\xi) = \xi \otimes \theta,
\]
just as before.
\end{proof}

A dual argument establishes the corresponding result for cancelling $2$- and $3$-handles:

\begin{lem}
Let $W_2$ be a cobordism obtained by attaching a single $2$-handle to the closed, connected, oriented $3$-manifold $Y$ along a framed knot $\bbK$, and let $W_2$ be a cobordism obtained by attaching a $3$-handle to $Y(\bbK)$ along some $2$-sphere such that the $3$-handle cancels the $2$-handle. Then $\SI(W_3) \circ \SI(W_2)$ is the identity.
\end{lem}

By combining the various lemmas throughout this section, we finally obtain the following.

\begin{thm}
\label{thm:cobordism-maps}
For any compact, connected cobordism $W: Y \longrightarrow Y'$ of closed, connected, oriented $3$-manifolds $Y$, there is a well-defined map $\SI(W): \SI(Y) \longrightarrow \SI(Y')$ between their symplectic instanton homologies which is a topological invariant of the $4$-manifold $W$.
\end{thm}

It is also useful to know that the cobordism maps are well-behaved under composition.

\begin{thm}
If $W: Y \longrightarrow Y'$ and $W': Y' \longrightarrow Y''$ are two cobordisms between connected, closed, oriented $3$-manifolds $Y$, $Y'$, and $Y''$, then
\[
	\SI(W \cup_{Y^\prime} W^\prime) = \SI(W') \circ \SI(W): \SI(Y) \longrightarrow \SI(Y'').
\]
\end{thm}

\begin{proof}
Since $\SI(W)$ and $\SI(W')$ are defined in terms of handle decompositions of $W$ and $W'$, to prove the composition law it suffices to check that the relevant maps induced by handle additions commute. We already showed that $2$-handle maps commute with each other in Theorem \ref{lem:indep-2handleorder}, and certainly we can commute $1$-handle maps with $3$-handle maps defined by nonseparating attaching spheres. Therefore it remains to show that $2$-handle maps commute with both $1$- and $3$-handle maps.

Let $\calH_0 = (\Sigma_1, \alpha_0, \beta_0, \gamma_0, z)$ be the genus $1$ Heegaard triple with $\alpha_0, \beta_0,$ and $\gamma_0$ all the meridian of $\Sigma_1$. Given a framed link $\bbL \subset Y$, a bouquet $B(\bbL)$ for $\bbL$, and a Heegaard triple $\calH$ subordinate to $B(\bbL)$, then $\calH \# \calH_0$ is a Heegaard triple subordinate to the bouquet induced by $B(\bbL)$ in $Y'$, the $3$-manifold obtained by adding a $1$-handle to $Y$. Let $W_1: Y \longrightarrow Y'$ and $W_1(\bbL): Y(\bbL) \longrightarrow Y'(\bbL)$ denote the associated $1$-handle cobordisms, and $W_2: Y \longrightarrow Y(\bbL)$, $W_2^\prime: Y' \longrightarrow Y'(\bbL)$ denote the associated $2$-handle cobordisms. A combination of Theorem \ref{thm:kunneth-triangle} and \cite[Proposition 9.5]{horton1} gives
\begin{align*}
	\SI(W_2^\prime) \circ \SI(W_1)(\xi) & = \mu_2^{\alpha\cup\alpha_0,\beta\cup\beta_0,\gamma\cup\gamma_0}(\xi \otimes \Theta_{\alpha_0\beta_0},\Theta_{\beta\gamma} \otimes \Theta_{\beta_0\gamma_0}) \\
	 & = \mu_2^{\alpha\beta\gamma}(\xi, \Theta_{\beta\gamma}) \otimes \mu_2^{\alpha_0\beta_0\gamma_0}(\Theta_{\alpha_0\beta_0},\Theta_{\beta_0\gamma_0}) \\
	 & = \SI(W_2)(\xi) \otimes \Theta_{\alpha_0\gamma_0},
\end{align*}
while on the other hand
\[
	\SI(W_1(\bbL)) \circ \SI(W_2)(\xi) = \SI(W_2)(\xi) \otimes \Theta_{\alpha_0\gamma_0}.
\]
Therefore cobordism maps for $1$- and $2$-handles commute.

To show that cobordism maps for $2$- and $3$-handles commute, we use an argument dual to the one of the previous paragraph. With $\calH_0$ as above, suppose $Y'$ is the result of adding a $3$-handle to some non-separating $2$-sphere in $Y$ that does not intersect the framed link $\bbL \subset Y$. Write $W_3: Y \longrightarrow Y'$ and $W_3(\bbL): Y(\bbL) \longrightarrow Y'(\bbL)$ for the associated $3$-handle cobordisms. Given a bouquet $B(\bbL)$ for $\bbL$ in $Y'$, there is a subordinate Heegaard triple $\calH$ such that $\calH \# \calH_0$ is a Heegaard triple subordinate to the bouquet for $\bbL$ in $Y$ induced by $B(\bbL)$. Again, Theorem \ref{thm:kunneth-triangle} and \cite[Proposition 9.5]{horton1} imply that
\begin{align*}
	\SI(W_3(\bbL)) \circ \SI(W_2)(\xi \otimes \Theta_{\alpha_0\gamma_0}) & = \SI(W_3(\bbL))(\mu_2^{\alpha\cup\alpha_0,\beta\cup\beta_0,\gamma\cup\gamma_0}(\xi \otimes \Theta_{\alpha_0\beta_0},\Theta_{\beta\gamma} \otimes \Theta_{\beta_0\gamma_0})) \\
	 & = \SI(W_3(\bbL))(\mu_2^{\alpha\beta\gamma}(\xi, \Theta_{\beta\gamma}) \otimes \mu_2^{\alpha_0\beta_0\gamma_0}(\Theta_{\alpha_0\beta_0},\Theta_{\beta_0\gamma_0})) \\
	 & = \SI(W_3(\bbL))(\SI(W_2^\prime)(\xi) \otimes \Theta_{\alpha_0\gamma_0}) \\
	 & = \SI(W_2^\prime)(\xi),
\end{align*}
while on the other hand
\[
	\SI(W_2^\prime) \circ \SI(W_3)(\xi \otimes \Theta_{\alpha_0\gamma_0}) = \SI(W_2^\prime)(\xi),
\]
which establishes the desired commutativity.
\end{proof}

\begin{rem}
As noted before, $\SI(Y)$ is really an invariant of the \emph{pointed} $3$-manifold $(Y,x)$. Therefore one should consider cobordisms $(W, \gamma): (Y,x) \longrightarrow (Y',x')$ of \emph{pointed} $3$-manifolds when discussing functoriality. Here $\gamma: [0,1] \longrightarrow W$ is a properly embedded path with $\gamma(0) = x$, $\gamma(1) = x'$. The cobordism maps constructed here should be thought of as only using the simplest possible paths $\gamma$, \emph{i.e.} for each handle attachment we think of $W = (Y \times [0,1]) \cup (k\text{-handle})$ with $\gamma(t) = (x,t)$.
\end{rem}

\subsection{Blowups}

In this section, we inspect the behavior of cobordism maps under blowups (\emph{i.e.} interior connected sums with $\overline{\C P}^2$).

In terms of framed links, blowing up corresponds to adding a $2$-handle to a $-1$-framed unknot. Let us consider the simplest case of the $-1$-framed unknot $\bbK$ in the $3$-sphere. It is clear that there is a genus $1$ Heegaard triple $\calH = (\Sigma_1, \alpha, \beta, \gamma, z)$ subordinate to the simplest possible bouquet $B(\bbK)$ for $\bbK$ with $\alpha$ and $\beta$ the standard meridian and longitude generators for $\pi_1(\Sigma_1)$, and $\gamma = \beta - \alpha$ (see Figure \ref{fig:blowup}). The corresponding Lagrangians $L_\alpha$, $L_\beta$, and $L_\gamma$ all pairwise intersect at only the trivial representation $\theta = [I,I,\bfi, \bfj, -\bfk] \in \scrR_{1,3}$.

\begin{figure}[h]
	\centering
	\includegraphics[scale=1.2]{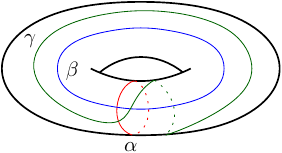}
	\caption{The genus $1$ Heegaard diagram subordinate to $-1$-surgery on the unknot in $S^3$.}
	\label{fig:blowup}
\end{figure}

\begin{lem}
For $\bbK$ the $-1$-framed unknot in $S^3$ as above and $W_2: S^3 \longrightarrow S^3$ is the associated $2$-handle cobordism, $\SI(W_2): \SI(S^3) \longrightarrow \SI(S^3)$ is the zero map.
\label{lem:S3blowup}
\end{lem}

\begin{proof}
Recall \hide{(Theorem \ref{thm:gradings}) }that there is a $\Z/2$-grading on symplectic instanton homology, and the trivial representation may be taken to sit in grading zero in $\SI(S^3)$. In general, given $x \in L_\alpha \cap L_\beta$, $y \in L_\beta \cap L_\gamma$, and $w \in L_\gamma \cap L_\alpha$, the expected dimension of the moduli space of pseudoholomorphic triangles in $\scrR_{1,3}$ through $x$, $y$, and $w$ mod $2$ is given in terms of the absolute grading by
\[
	\dim \mathcal{M}(x,y,w) = \gr(x) + \gr(y) - \gr(w) - 3 \pmod 2, \tag{$\ast$}
\]
where the $3 = \tfrac{1}{2} \dim \scrR_{1,3}$ appears since we are using Floer homology rather than Floer cohomology.

By applying a small Hamiltonian isotopy to $L_\alpha$, we can resolve the triple intersection point $\theta$ into three intersection points $\theta_1 \in L_\alpha \cap L_\beta$, $\theta_2 \in L_\beta \cap L_\gamma$, and $\theta_3 \in L_\gamma \cap L_\alpha$. The $\Z/2$-grading is preserved under this Hamiltonian isotopy, so we still have that $\gr(\theta_i) = 0$ for $i = 1, 2, 3$. Hence by $(\ast)$ we have that $\dim \mathcal{M}(\theta_1, \theta_2, \theta_3) \equiv 1\hide{-3} \pmod 2$, so the count of rigid pseudoholomorphic triangles coming from $\calH$ is necessarily zero. It follows that $\SI(W_2) \equiv 0$.
\end{proof}

\begin{thm}
If $\bbK$ is the $-1$-framed unknot in any closed, oriented $3$-manifold $Y$ and $W_2: Y \longrightarrow Y$ is the associated $2$-handle cobordism, then $\SI(W_2): \SI(Y) \longrightarrow \SI(Y)$ is the zero map.
\end{thm}

\begin{proof}
This follows from Lemma \ref{lem:S3blowup}, Theorem \ref{thm:kunneth-triangle}, and the fact that we can choose a Heegaard triple subordinate to a bouquet for $\bbK$ such that it contains the Heegaard triple from Figure \ref{fig:blowup} as a connect summand.
\end{proof}

%--------------------------------------------------------------------------------------------------------------------------

\subsection{Functoriality for Nontrivial Bundles}

Now we turn to the topic of functoriality of symplectic instanton homology with respect to cobordisms in the presence of nontrivial $\SO(3)$-bundles, \emph{i.e.} assigning to each \emph{bundle} cobordism $(W, P): (Y_0, P_0) \longrightarrow (Y_1, P_1)$ a homomorphism between the symplectic instanton homologies of the boundary components.

By a classic theorem of Dold and Whitney \cite{dold-whitney}, $\SO(3)$-bundles on a compact $4$-manifold $W$ are classified by pairs $(w_2, p_1) \in H_2(W; \Z/2) \times H_4(W; \Z)$ (the Stiefel-Whitney and Pontryagin classes of the bundle) such that the Pontryagin square of $w_2$ is the mod $4$ reduction of $p_1$. Since $H^4(W) = 0$ for $\partial W \neq \varnothing$, we therefore see that $\SO(3)$-bundles over a cobordism between non-empty $3$-manifolds are simply classified by their second Stiefel-Whitney class. By Poincar\'e-Lefschetz duality, we may then think of such a bundle $P \longrightarrow W$ as corresponding to a relative mod $2$ homology class $\Omega \in H_2(W, \partial W; \F_2)$. By naturality of Poincar\'e-Lefschetz duality, we furthermore have that if $\partial \Omega = \omega_0 \oplus \omega_1 \in H_1(Y_0; \F_2) \oplus H_1(Y_1; \F_2) \cong H_1(\partial W; \F_2)$ (where $\partial$ is the connecting homomorphism in the long exact sequence of the pair $(W, \partial W)$), then $\omega_i$ is the mod $2$ homology class representing the $\SO(3)$-bundle $P_i = i_{Y_i}^\ast P$. In this way, we eliminate direct reference to $\SO(3)$-bundles in this section.

Let $W: Y \longrightarrow Y'$ be a compact, connected, oriented cobordism of closed, connected, oriented $3$-manifolds. Given a mod $2$ relative homology class $\Omega \in H_2(W, \partial W; \F_2)$, write $\omega \oplus \omega' \in H_1(Y; \F_2) \oplus H_1(Y'; \F_2) \cong H_1(\partial W; \F_2)$ for its image under the boundary map in the long exact sequence for the pair $(W, \partial W)$. Let use write $(W, \Omega): (Y, \omega) \longrightarrow (Y', \omega')$ for such a combination of a cobordism with a mod $2$ homology class. We claim the following:

\begin{thm}
To any compact, connected, oriented cobordism with homology class $(W, \Omega): (Y, \omega) \longrightarrow (Y', \omega')$ we may associate a well-defined homomorphism
\[
	\SI(W, \Omega): \SI(Y, \omega) \longrightarrow \SI(Y', \omega')
\]	
that is functorial in the following sense:
\begin{itemize}
	\item[(1)] For any closed, connected, oriented $3$-manifold $Y$, $\SI(Y \times [0,1], \omega \times [0,1]) = \id_{\SI(Y,\omega)}$.
	\item[(2)] For any other compact, connected, oriented cobordism with homology class $(W', \Omega'): (Y', \omega') \allowbreak \longrightarrow (Y'', \omega'')$, we have
	\[
		\SI(W \cup_{Y'} W', \Omega \cup_{Y'} \Omega') = \SI(W', \Omega') \circ \SI(W, \Omega).
	\]
\end{itemize}
\end{thm}

The homomorphism $\SI(W, \Omega)$ is constructed similarly to the $\Omega = 0$ case treated, defined first for individual handle attachments and then extended to arbitrary $(W, \Omega)$ by Kirby calculus.

First, suppose $(W_1, \Omega): (Y, \omega) \longrightarrow (Y', \omega')$ consists of a single $1$-handle attachment. Then we necessarily have that $Y' \cong Y \# (S^2 \times S^1)$. A Mayer-Vietoris argument shows that in this case, $H^2(W_1; \F_2) \cong H^2(Y; \F_2)$, and therefore $H_2(W_1, \partial W_1; \F_2) \cong H_1(Y; \F_2)$. It follows that for $1$-handle attachments, we necessarily have $\omega' = \omega \cup 0 \in H_1(Y; \F_2) \oplus H_1(S^2 \times S^1; \F_2 ) \cong H_1(Y'; \F_2)$, and the handle attachment map must be of the form
\[
	\CSI(W_1, \Omega): \CSI(Y, \omega) \longrightarrow \CSI(Y, \omega) \otimes \CSI(S^2 \times S^1,0).
\]
Recall that $\CSI(S^2 \times S^1, 0) \cong H^{3 - \ast}(S^3)$ as a unital algebra; let $\Theta$ denote its unit. Then we define the $1$-handle attachment map in the same way as the $\Omega = 0$ case:
\[
	\CSI(W_1, \Omega)(\xi) = \xi \otimes \Theta.
\]

If $W_1$ consists of $m$ $1$-handle attachments, we may decompose it as $W_1 = W_{1,1} \cup \cdots \cup W_{1,m}$ and define
\[
	\CSI(W_1, \Omega) = \CSI(W_{1,m}, \Omega) \circ \cdots \circ \CSI(W_{1,1}, \Omega),
\]
where the homology classes $\Omega$ are all the same by the above discussion. We may then proceed exactly as in the $\Omega = 0$ case to obtain the following:

\begin{thm}
The map $\SI(W_1, \Omega): \SI(Y, \omega) \longrightarrow \SI(Y', \omega')$ is invariant under the ordering of the $1$-handles of $W_1$ and handleslides amongst them, and therefore is an invariant of the pair $(W_1, \Omega)$.
\end{thm}

The situation for $3$-handles is dual to the $1$-handle case. If $(W_3,\Omega): (Y, \omega) \longrightarrow (Y', \omega')$ consists of a single $3$-handle attachment, then $Y \cong Y' \# (S^2 \times S^1)$ and $H_2(W, \partial W; \F_2) \cong H_1(Y'; \F_2)$. Hence $\omega = \omega' \cup 0 \in H_1(Y'; \F_2) \oplus H_1(S^2 \times S^1; \F_2)$, and the handle attachment map is of the form
\[
	\CSI(W_3, \Omega): \CSI(Y', \omega') \otimes \CSI(S^2 \times S^1, 0) \longrightarrow \CSI(Y', \omega').
\]
We may then define the $3$-handle attachment map in the same way as the $\Omega = 0$ case:
\[
	\CSI(W_3, \Omega)(\xi \otimes \eta) = \begin{cases} \xi, & \text{if } \eta = \Theta, \\ 0, & \text{if } \eta \neq \Theta. \end{cases}
\]

If $W_3$ consists of $m$ $3$-handle attachments, we may decompose it as $W_3 = W_{3,1} \cup \cdots \cup W_{3,m}$ and define
\[
	\CSI(W_3, \Omega) = \CSI(W_{3,m}, \Omega) \circ \cdots \circ \CSI(W_{3,1}, \Omega),
\]
where the homology classes $\Omega$ are all the same by the above discussion. We may then proceed exactly as in the $\Omega = 0$ case to obtain the following:

\begin{thm}
The map $\SI(W_3, \Omega): \SI(Y, \omega) \longrightarrow \SI(Y', \omega')$ is invariant under the ordering of the $3$-handles of $W_3$ and handleslides amongst them, and therefore is an invariant of the pair $(W_3, \Omega)$.
\end{thm}

Finally, consider the case where $(W_2, \Omega): (Y, \omega) \longrightarrow (Y', \omega')$ consists of a single $2$-handle attachment. A $2$-handle is attached to $Y \times \{1\} \subset Y \times [0,1]$ along a framed link $\bbL = (L, \lambda)$ to obtain $Y' \cong Y(\bbL)$. $H_1(Y; \F_2)$ and $H_1(Y';\F_2)$ differ only possibly in the homology class represented by the meridian of the link $L$; write $\omega_L$ for the mod $2$ homology class of $L$ in $Y$ and write $\omega_L^\prime$ for the mod $2$ homology class of $L$ in $Y'$. It is easy to see that the image of any $\Omega \in H_2(W_2, \partial W_2; \F_2)$ in $H_1(Y; \F_2) \oplus H_1(Y'; \F_2)$ under the boundary map in the long exact sequence for the pair $(W_2, \partial W_2)$ must have the form
\[
	(i_\ast \omega + \epsilon \omega_L, i_\ast^\prime \omega + \epsilon' \omega_L^\prime) \in H_1(Y; \F_2) \oplus H_1(Y'; \F_2),
\]
where $\epsilon, \epsilon' \in \F_2$, $\omega$ is a mod $2$ homology class in $Y \setminus L$, and $i: Y \setminus L \hookrightarrow Y$ and $i^\prime: Y \setminus L \hookrightarrow Y'$ are the inclusion maps.

Let $Y$ be a closed, connected, oriented $3$-manifold and $\bbL$ a framed link in $Y$, write $W_2$ for the cobordism corresponding to attaching a $2$-handle to $Y$ along $\bbL$, and let $\Omega \in H_2(W_2, \partial W_2; \F_2)$ be any homology class. Recall from the above that
\[
	\partial \Omega = (i_\ast \omega + \epsilon \omega_L, i_\ast^\prime \omega + \epsilon' \omega_L^\prime) \in H_1(Y; \F_2) \oplus H_1(Y'; \F_2).
\]
Now, fixing a bouquet $B(\bbL)$ for $\bbL$ and a Heegaard triple $\calH = (\Sigma_g, \bfalpha, \bfbeta, \bfgamma, z)$ subordinate to $B(\bbL)$, we can represent the mod $2$ homology classes $i_\ast \omega + \epsilon \omega_L$ and $i_\ast^\prime \omega + \epsilon' \omega_L^\prime$ by a knot lying in the $\alpha$-handlebody of $Y$ and $Y(\bbL)$. It is clear that $L_\beta \cap L_\gamma \cong (S^3)^{g-n}$ is a clean intersection, and therefore
\[
	\HF(L_\beta, L_\gamma) = \SI(S^3, 0)^{\otimes n} \otimes \SI(S^2 \times S^1, 0)^{\otimes (g - n)} \cong H^{3-\ast}(S^3)^{\otimes(g - n)}
\]
as a unital algebra; write $\Theta_{\beta\gamma}$ for its unit. We then define the $2$-handle attachment map as a triangle map for the Lagrangians $L_\alpha^{i_\ast \omega + \epsilon \omega_L} = L_\alpha^{i_\ast^\prime \omega + \epsilon' \omega_L^\prime}$, $L_\beta$, and $L_\gamma$ (see Figure \ref{fig:2-handlemap}):
\[
	\CSI(W_2, \Omega): \CSI(Y, i_\ast\omega + \epsilon \omega_L) \longrightarrow \CSI(Y(\bbL), i_\ast^\prime\omega + \epsilon' \omega_L^\prime),
\]
\[
	\xi \mapsto \mu_2^{\alpha\beta\gamma}(\xi, \Theta_{\beta\gamma}).
\]

\begin{figure}
	\centering
	\includegraphics[scale=1]{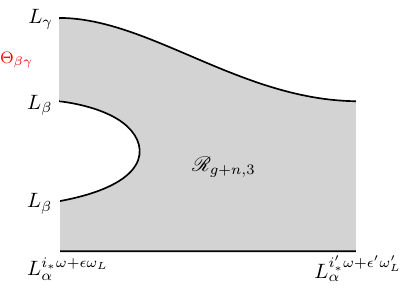}
	\caption{Triangles counted by the $2$-handle map $\CSI(W_2, \Omega)$.}
	\label{fig:2-handlemap}
\end{figure}

With the effect of $\Omega$ relegated to the Lagrangian for the $\alpha$-handlebody, proving that $\SI(W_2, \Omega)$ is independent of the choice of bouquet for $\bbL$, the ordering of the components of $L$, and handleslides among the components of $L$ proceeds exactly as usual, as the proofs for the $\Omega = 0$ case only used the algebra structure of $\HF(L_\beta, L_\gamma)$. Hence we conclude the following.

\begin{thm}
The map $\SI(W_2, \Omega): \SI(Y, \omega) \longrightarrow \SI(Y', \omega')$ induced by adding $2$-handles to a framed link $\bbL$ in $Y$ is independent of the choice of bouquet $B(\bbL)$ used to define it. Furthermore, if $\bbL_1$ and $\bbL_2$ are two framed links in $Y$ and $\Omega_k \in H_2(W(\bbL_k), \partial W(\bbL_k); \F_2)$ $(k = 1, 2)$ are mod $2$ homology classes with $\Omega_1|_{Y(\bbL_1)} = \Omega_2|_{Y(\bbL_1)}$, then
\[
	\SI(W(\bbL_1 \cup \bbL_2), \Omega_1 \cup \Omega_2) = \SI(W(\bbL_2), \Omega_2) \circ \SI(W(\bbL_1), \Omega_1)
\]
and if $\bbL'$ differs from $\bbL$ by handleslides amongst the components of $L$, then $\SI(W(\bbL'), \Omega) = \SI(W(\bbL), \Omega)$.
\end{thm}

Now, given an arbitrary $4$-dimensional cobordism $W: Y \longrightarrow Y'$, we may represent it as a relative handlebody built on $Y$, with handles of index $1$, $2$, and $3$ added in order of increasing index. Hence we may write $W = W_1 \cup W_2 \cup W_3$, where each $W_k$ is a cobordism consisting entirely of $k$-handles. Furthermore, given a homology class $\Omega \in H_2(W, \partial W; \F_2)$, there exist homology classes $\Omega_k \in H_2(W_k, \partial W_k; \F_2)$ such that
\begin{align}
	\label{eqn:Omega-decomp}
	\Omega & = (i_1)_\ast \Omega_1 + (i_2)_\ast \Omega_2 + (i_3)_\ast \Omega_3.
\end{align}
We then define
\[
	\SI(W, \Omega) = \SI(W_3, \Omega_3) \circ \SI(W_2, \Omega_2) \circ \SI(W_1, \Omega_1).
\]
To prove that $\SI(W, \Omega): \SI(Y, \omega) \longrightarrow \SI(Y', \omega')$ as defined above is an invariant of $(W, \Omega)$, we must show that Kirby moves on the handles of $W$ as well as different choices of the $\Omega_k$ satisfying Equation (\ref{eqn:Omega-decomp}) result in the same homomorphism.

First we show independence of the decomposition of $\Omega$ as in Equation (\ref{eqn:Omega-decomp}). Recall from the discussion of $1$- and $3$-handle maps that we necessarily have $\Omega_1 = \omega \times [0,1]$ for $\omega = \Omega|_Y$ and $\Omega_3 = \omega' \times [0,1]$ for $\omega' = \Omega|_{Y^\prime}$. This forces the decomposition of Equation (\ref{eqn:Omega-decomp}) to be unique, so that it causes no trouble for the well-definedness of $\SI(W, \Omega)$.

As for invariance under Kirby moves, we have already checked those Kirby moves which involve handles all of the same index, and it only remains to show that cancelling pairs of handles does not change $\SI(W, \Omega)$. The proofs of the corresponding results when $\Omega = 0$ apply word-for-word here, so we omit them.

\begin{lem}
\label{lem:12handlecancel}
Let $W_1: Y \longrightarrow Y'$ be a cobordism corresponding to a single $1$-handle attachment, and let $W_2: Y' \longrightarrow Y''$ be a cobordism corresponding to a $2$-handle attachment along a framed knot $\bbK$ in $Y'$ that cancels the $1$-handle from $W_1$. Then for any $\Omega \in H_2(W_1 \cup W_2, \partial(W_1 \cup W_2); \F_2)$, $\SI(W_1 \cup W_2, \Omega)$ is the identity map.
\end{lem}

\begin{lem}
Let $W_2: Y \longrightarrow Y'$ be a cobordism corresponding to attaching a $2$-handle along a framed knot $\bbK$ in $Y$, and let $W_3: Y' \longrightarrow Y''$ be a cobordism corresponding to attaching a $3$-handle to $Y'$ along some $2$-sphere such that the $3$-handle cancels the $2$-handle from $W_2$. Then for any $\Omega \in H_2(W_2 \cup W_3, \partial(W_2 \cup W_3); \F_2)$, $\SI(W_2 \cup W_3, \Omega)$ is the identity.
\end{lem}

As a result of the work done thus far in this section, we conclude the following.

\begin{thm}
For any compact, connected, oriented cobordism $W: Y \longrightarrow Y'$ of closed, connected, oriented $3$-manifolds $Y$ and $Y'$ and any homology class $\Omega \in H_2(W, \partial W; \F_2)$, there is a well-defined homomorphism $\SI(W, \Omega): \SI(Y, \Omega|_{Y}) \longrightarrow \SI(Y', \Omega|_{Y^\prime})$ that is an invariant of the pair $(W, \Omega)$.
\end{thm}

For our final effort of this section, we show that the assignment $(W, \Omega) \mapsto \SI(W, \Omega)$ is functorial, in the following sense:

\begin{thm}
Suppose $W: Y \longrightarrow Y'$ and $W': Y' \longrightarrow Y''$ are two compact, connected, oriented cobordisms of closed, connected, oriented $3$-manifolds and fix homology classes $\Omega \in H_2(W, \partial W; \F_2)$ and $\Omega' \in H_2(W', \partial W'; \F_2)$ such that $\Omega|_{Y^\prime} = \Omega'|_{Y^\prime}$. Then
\[
	\SI(W \cup W', \Omega \cup \Omega') = \SI(W', \Omega') \circ \SI(W, \Omega).
\]
\end{thm}

\begin{proof}
To prove the theorem, we must show that $1$- and $2$-handle maps commute with one another, and that $2$- and $3$-handle maps commute with one another. Once again, the proof for the $\Omega = 0$ case carries over directly.
\end{proof}
\section{Cobordism Maps in the Surgery Exact Triangle}

In this section, we show that the maps in the surgery exact triangle of \cite{horton1} are in fact $2$-handle cobordism maps.

\subsection{Review of the Exact Triangle}

The surgery exact triangle is induced on the chain level by the sequence of maps
\[
	\CF(L_0, \underline{L}, V) \otimes \CF(V^T, L_1) \xrightarrow{~C\Phi_0~} \CF(L_0, \underline{L}, L_1) \xrightarrow{~C\Phi_1~} \CF(L_0, \underline{L}, \tau_V L_1).
\]
$C\Phi_0$ is defined by a count of quilted triangles as in Figure \ref{fig:CPhi0}, and $C\Phi_1$ is defined by a count of pseudoholomorphic sections of a quilted Lefschetz fibration as in Figure \ref{fig:CPhi1}.

\begin{figure}[h]
\centering
\begin{minipage}{.45\textwidth}
% 	\hspace{.25in}
	\includegraphics[scale=.95]{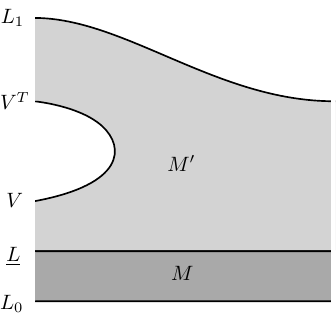}
	\caption{Quilts counted by the map $C\Phi_0$.}
	\label{fig:CPhi0}
\end{minipage}
\begin{minipage}{.45\textwidth}
% 	\vspace{.75in}\hspace{.25in}
	\includegraphics[scale=.95]{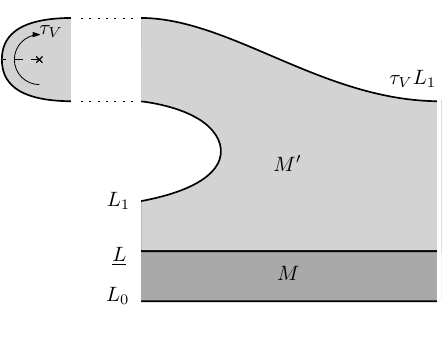}
	\caption{Quilted Lefschetz fibration defining $C\Phi_1$.}
	\label{fig:CPhi1}
\end{minipage}
% \caption{Maps appearing in the exact triangle.}
\end{figure}

To establish the exact triangle, we need a chain nullhomotopy of the composite $C\Phi_1 \circ C\Phi_0$. In our case, this explicit nullhomotopy is the map
\[
	h: \CF(L_0, \underline{L}, V) \otimes \CF(V^T, L_1) \longrightarrow \CF(L_0, \underline{L}, \tau_V L_1),
\]
\[
	h(x \otimes y) = \tilde{\mu}_2(x, k(y)) + \tilde{\mu}_3(x,y, c),
\]
where $\tilde{\mu}_2$ (respectively $\tilde{\mu}_3$) is the quilted triangle (respectively quilted rectangle) map pictured in Figure \ref{fig:mu2tilde} (resp. Figure \ref{fig:mu3tilde}), and $k$ is defined by counting pseudoholomorphic sections of the one-parameter family of Lefschetz fibrations interpolating between the two Lefschetz fibrations pictured in Figure \ref{fig:nullhomotopy-k}. Although in this section we only work with maps on the homology-level exact triangle, we bring up the chain nullhomotopy $h$ as it will appear when working with the link surgeries spectral sequence in Section \ref{sect:linksurgeries}.

\begin{figure}[h]
	\centering
\begin{minipage}{.45\textwidth}
	\hspace{.15in}
	\includegraphics{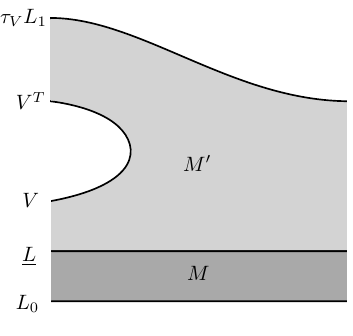}
	\caption{Quilted triangle map $\tilde{\mu}_2$.}
	\label{fig:mu2tilde}
\end{minipage}
\begin{minipage}{.45\textwidth}
	\hspace{.25in}
	\includegraphics{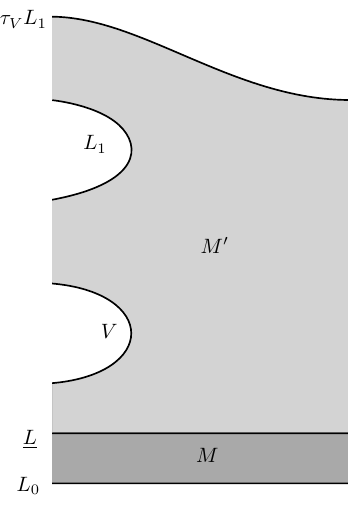}
	\caption{Quilted rectangle map $\tilde{\mu}_3$.}
	\label{fig:mu3tilde}
\end{minipage}
\end{figure}

\begin{figure}[h]
	\centering
	\includegraphics[scale=.9]{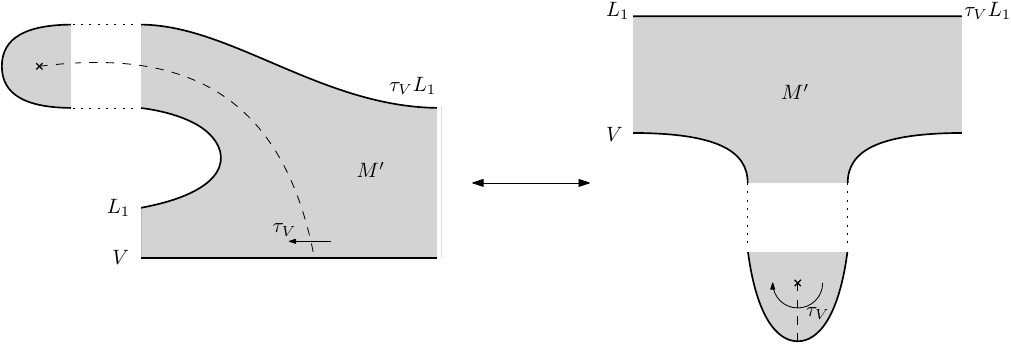}
	\caption{The $1$-parameter family of Lefschetz fibrations defining the chain nullhomotopy $k$.}
	\label{fig:nullhomotopy-k}
\end{figure}

We now relate the abstract setup of the exact triangle to the topology of Dehn surgery. Let $Y$ be a closed, oriented $3$-manifold with a framed knot $\bbK = (K, \lambda)$. Choose a Heegaard triple $(\Sigma_{g+1}, \bfalpha, \bfbeta, \bfgamma, z)$ subordinate to a bouquet for $\bbK$, and another triple $(\Sigma_{g+1}, \bfalpha, \bfbeta, \bfdelta, z)$ subordinate to $(K, \lambda + \mu)$. We then have a quadruple diagram $\calH = (\Sigma_{g+1}, \bfalpha, \bfbeta, \bfgamma, \bfdelta, z)$ satisfying the following:
\begin{itemize}
	\item $\beta_{g+1}$ is a meridian for $K$.
	\item $\gamma_{g+1}$ represents the framing $\lambda$ of $K$ and $\delta_{g+1}$ represents the framing $\lambda + \mu$ of $K$.
	\item $(\Sigma_{g+1}, \{\alpha_1, \dots, \alpha_{g+1}\}, \{\beta_1, \dots, \beta_{g}\}, z)$ represents the complement of $K$ in $Y$.
	\item $\gamma_k = \beta_k$ and $\delta_k = \beta_k$ for $k \neq g+1$.
	\item $\calH_{\alpha\beta}$ represents $Y$, $\calH_{\alpha\gamma}$ represents $Y_\lambda$, and $\calH_{\alpha\delta}$ represents $Y_{\lambda + \mu}$.
\end{itemize}

Let $\underline{L}$ be the Lagrangian correspondence coming from the first $g$ $\beta$- (equivalently, $\gamma$- or $\delta$-) handle attachments, $L_{\beta_{g+1}}$, $L_{\gamma_{g+1}}$, and $L_{\delta_{g+1}}$ be the Lagrangians in $\scrR_{1,3}$ corresponding to attaching the final handle of $\calH_{\alpha\beta}$, $\calH_{\alpha\gamma}$, or $\calH_{\alpha\delta}$, and let
\[
	V = \{[\rho] \in \scrR_{1,3} \mid \rho(\beta_{g+1}) = -I\}.
\]
Then on the level of Lagrangians $L_{\delta_{g+1}} = \tau_V L_{\gamma_{g+1}}$ and the exact triangle translates to
\[
	\xymatrix{\SI(Y, \omega_K) \ar[rr]^{\Phi_0} & & \SI(Y_\lambda(K)) \ar[dl]^{\Phi_1} \\
	 & \SI(Y_{\lambda + \mu}(K)) \ar[ul] & }
\]
where $\omega_K$ is the mod $2$ homology class of $K$ in $Y$. The goal of this section is to identify each map in this triangle with the $2$-handle cobordism map corresponding to each surgery.

\subsection{From Pseudoholomorphic Sections to Pseudoholomorphic Polygons}

In this Section, we identify the maps in the surgery exact triangle with the homomorphisms induced by the relevant $2$-handle attachments. Since one of the maps involves counting pseudoholomorphic sections of a Lefschetz fibration, we will need to show that this is equal to the count of pseudoholomorphic triangles defining our cobordism maps. Furthermore, it will also be useful for us later to interpret the chain homotopy $h$ appearing in the proof of the exact triangle as a count of certain pseudoholomorphic rectangles.

We will use the setup from the previous subsection where all Lagrangians are defined using the Heegaard quadruple $\calH = (\Sigma_{g+1}, \bfalpha, \bfbeta, \bfgamma, \bfdelta, z)$, although some of the results here apply in more general situations.

\begin{figure}[h]
	\centering
	\includegraphics[scale=1.25]{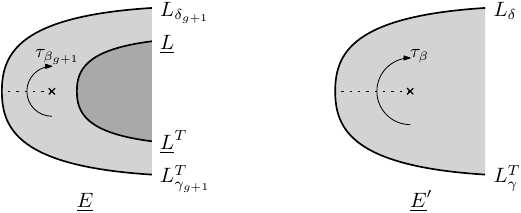}
	\caption{Two quilted Lefschetz-Bott fibrations.}
	\label{fig:lefschetz-unit}
\end{figure}

All results in this section rely on the following technical lemma:

\begin{lem}
\label{lem:lefschetz-unit}
Let $\underline{E}$ and $\underline{E}'$ be the two quilted Lefschetz-Bott fibrations pictured in Figure \ref{fig:lefschetz-unit}. Then the relative invariant $\Phi_{\underline{E}}$ is identified with the relative invariant $\Phi_{\underline{E}^\prime}$ under the isomorphism $\HF(L_{\gamma_{g+1}}^T, \underline{L}^T, \underline{L}, L_{\delta_{g+1}}) \cong \HF(L_{\gamma_{g+1}}^T \circ \underline{L}^T, \underline{L} \circ L_{\delta_{g+1}}) = \HF(L_\gamma^T, L_\delta)$. In other words, the following diagram commutes:
\[
	\xymatrix{ & \HF(L_{\gamma_{g+1}}^T, \underline{L}^T, \underline{L}, L_{\delta_{g+1}}) \ar[dd]^{~\substack{\mathrm{geometric} \\ \mathrm{composition}}} \\ \HF(\mathrm{pt}) \ar[ur]^{\Phi_{\underline{E}}} \ar[dr]_{\Phi_{\underline{E}'}} & \\ & \HF(L_\gamma^T, L_\delta) }
\]
\end{lem}

\begin{proof}
First, we determine $\Phi_{\underline{E}^\prime}(\text{pt})$. Equip $\underline{E}'$ with a horizontal almost complex structure. Of all possible horizontal sections of $\underline{E}'$, only the one corresponding to $\Theta_{\gamma\delta}$ has index zero. By monotonicity, any non-horizontal section of $\underline{E}'$ will have strictly positive index and hence does not contribute to $\Phi_{\underline{E}^\prime}(\text{pt})$. Therefore
\begin{align}
	\Phi_{\underline{E}^\prime}(\text{pt}) = \Theta_{\gamma\delta}.
\end{align}

\begin{figure}[h]
	\centering
	\includegraphics[scale=1.25]{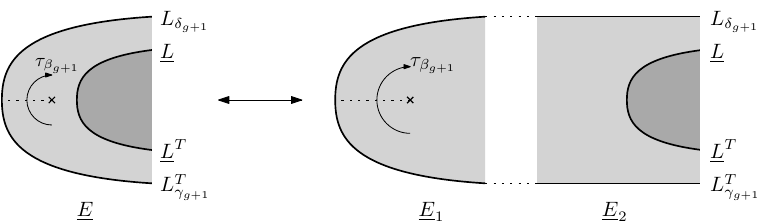}
	\caption{Breaking the quilted Lefschetz fibration $\underline{S}$ into simpler pieces.}
	\label{fig:lefschetz-unit-proof}
\end{figure}

To determine $\Phi_{\underline{E}}(\text{pt})$, we consider $\underline{E}$ as a gluing of two quilted Lefschetz fibrations $\underline{E}_1$ and $\underline{E}_2$ as indicated in Figure \ref{fig:lefschetz-unit-proof}. For sufficiently long gluing lengths, $C\Phi_{\underline{E}}(\text{pt}) = (C\Phi_{\underline{E}_2} \circ C\Phi_{\underline{E}_1})(\text{pt})$ on the chain level, so this decomposition allows us to count pseudoholomorphic sections on each piece of the gluing separately. By reasoning analogous to the computation of $\Phi_{\underline{E}^\prime}(\text{pt})$, we have that $\Phi_{\underline{E}_1}(\text{pt}) = \theta \in \HF(L_{\gamma_{g+1}}^T, L_{\delta_{g+1}})$. On the other hand, by the discussion in Section \ref{sect:S3-S2xS1}, $\Phi_{\underline{E}_2}(\theta) = \theta \times \Theta_{\gamma\delta} \in \HF(L_{\gamma_{g+1}}^T, \underline{L}^T, \underline{L}, L_{\delta_{g+1}})$. Therefore
\begin{align}
	\Phi_{\underline{E}}(\text{pt}) = \theta \times \Theta_{\gamma\delta}.
\end{align}

To complete the proof, simply note that under geometric composition, $\theta \times \Theta_{\gamma\delta}$ maps to $\Theta_{\gamma\delta}$. Hence the desired triangle commutes.
\end{proof}

With the above lemma in place, we can begin interpreting the chain maps $C\Phi_0$ and $C\Phi_1$ in the surgery exact triangle as counts of pseudoholomorphic triangles with certain specified vertices, and we can also show that the chain homotopy $h$ can be considered as a count of pseudoholomorphic rectangles with certain specified vertices. Let us start with $C\Phi_1$. A basic manipulation of quilted surfaces (see Figure \ref{fig:CPhi1-triangle}) shows that there is a commutative diagram
\[
	\xymatrix{\CF(L_\alpha, \underline{L}, L_{\gamma_{g+1}}) \ar[d]_{\cong} \ar[r]^{C\Phi_1} & \CF(L_\alpha, \underline{L}, L_{\delta_{g+1}}) \ar[d]^{\cong} \\
	\CF(L_\alpha, L_\gamma) \ar[r]_{\mu_2(\cdot,c')} & \CF(L_\alpha, L_\delta)}
\]
where the vertical maps are induced by the geometric compositions $\underline{L} \circ L_{\gamma_{g+1}} = L_\gamma$ and $\underline{L} \circ L_{\delta_{g+1}} = L_\delta$. Now $c' \in \CF(L_\gamma, L_\delta)$ is the unit for the algebra $\CF(L_\gamma, L_\delta) \cong H^{3-\ast}(\text{pt}; \Z) \otimes H^{3-\ast}(S^3; \Z)^{\otimes g}$, and is hence identified with $\Theta_{\gamma\delta}$. Furthermore, $\mu_2(\cdot, \Theta_{\gamma\delta})$ is precisely the map induced by the standard $2$-handle cobordism from $Y_\lambda(K)$ to $Y_{\lambda + \mu}(K)$. We conclude the following:

\begin{figure}
	\centering
	\includegraphics[scale=.9]{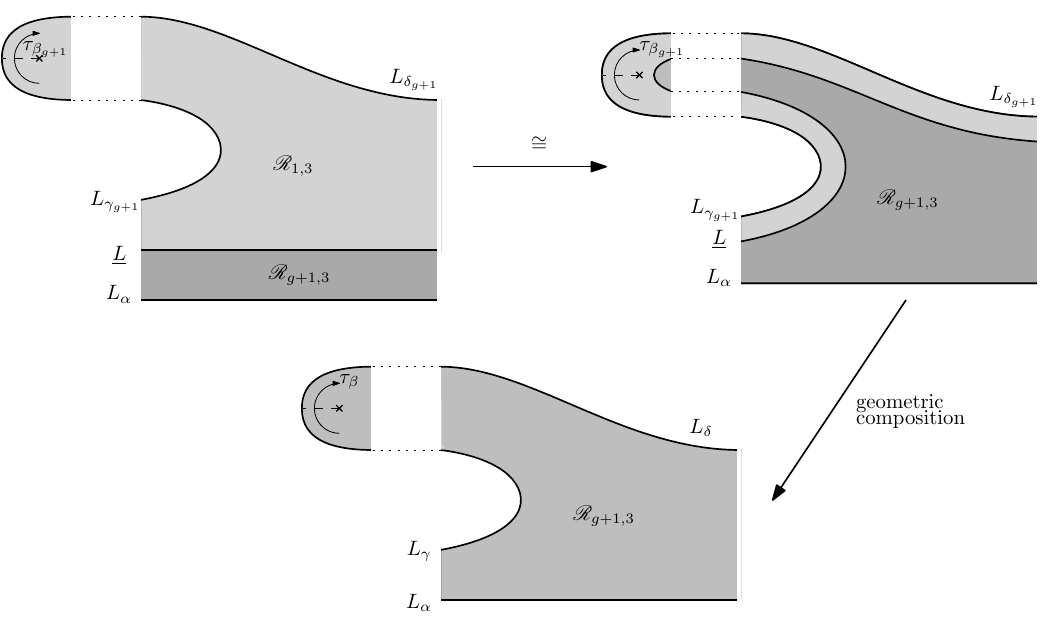}
	\caption{Identifying $C\Phi_1$ with a classical triangle map.}
	\label{fig:CPhi1-triangle}
\end{figure}

\begin{prop}
The map $C\Phi_1$ in the surgery exact triangle may be identified with the cobordism map $F_{W_{\lambda+\mu}}: \CF(L_\alpha, L_\gamma) \longrightarrow \CF(L_\alpha, L_\delta)$, where $W_{\lambda+\mu}$ is the standard $2$-handle cobordism between $Y_\lambda(K)$ and $Y_{\lambda + \mu}(K)$. In other words, there is a commutative diagram
\[
	\xymatrix{\CF(L_\alpha, \underline{L}, L_{\gamma_{g+1}}) \ar[d]_{\cong} \ar[r]^{C\Phi_1} & \CF(L_\alpha, \underline{L}, L_{\delta_{g+1}}) \ar[d]^{\cong} \\
	\CF(L_\alpha, L_\gamma) \ar[r]_{F_{W_{\lambda+\mu}}} & \CF(L_\alpha, L_\delta)}
\]
where the vertical maps are induced by geometric composition.
\label{Phi1Triangle}
\end{prop}

Now $C\Phi_0$ is already obviously a (quilted) triangle map, but \emph{a priori} it is not of the same form as one of the classical triangle maps associated to a $2$-handle cobordism and it also seems to lack a fixed vertex. Note that the upper left corner of the triangle in Figure \ref{fig:CPhi0} must be mapped to an intersection point of $L_{\gamma_{g+1}}$ and $V^T$ in $\scrR_{1,3}$. But these Lagrangians intersect only at the representation $[\rho_0] = [-I, \pm I, \bfi, \bfj, -\bfk]$ (where the $\pm$ depends on the framing $\lambda$ of the knot $K \subset Y$). Therefore the given vertex is \emph{forced} to be fixed.

\begin{figure}[t]
	\centering
	\includegraphics{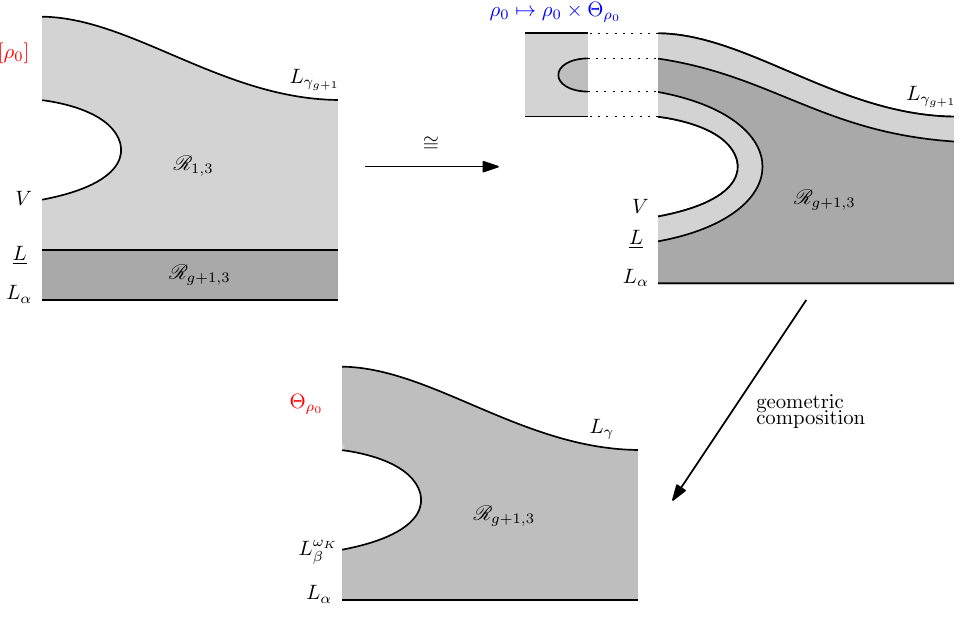}
	\caption{Identifying $C\Phi_0$ with a classical triangle map.}
	\label{fig:CPhi0-triangle}
\end{figure}

We still wish to perform some quilt manipulations similar to those done for $C\Phi_1$ to pass from a quilted triangle to just a triangle mapping to $\scrR_{g+1,3}$. The required manipulations are pictured in Figure \ref{fig:CPhi0-triangle}. The first indicated move sufficiently stretches the quilt so that we may consider it as a composition of the two pictured quilt maps in the upper right corner of the figure. The first quilt in the composition is the map $\rho_0 \mapsto \rho_0 \times \Theta_{\rho_0}$ discussed at the end of Section \ref{sect:S3-S2xS1}. The input for the upper left corner of the second quilt in the composition is forced to be $\rho_0 \times \Theta_{\rho_0}$ since $V^T \cap L_{\gamma_{g+1}} = \{[\rho_0]\}$. Finally, under the geometric composition $\CF(V^T, \underline{L}^T, \underline{L}, L_{\gamma_{g+1}}) \cong \CF(L_\beta^{\omega_K}, L_\gamma)$, $\rho_0 \times \Theta_{\rho_0}$ is mapped to $\Theta_{\rho_0}$. Therefore we can conclude the following:

\begin{prop}
The map $C\Phi_0$ in the surgery exact triangle may be identified with the cobordism map $F_{W_{\lambda}}: \CF(L_\alpha, L_\beta^{\omega_K}) \longrightarrow \CF(L_\alpha, L_\gamma)$, where $W_{\lambda}$ is the standard $2$-handle cobordism between $Y$ and $Y_{\lambda}(K)$. In other words, there is a commutative diagram
\[
	\xymatrix{\CF(L_\alpha, \underline{L}, V) \otimes \CF(V^T, L_{\gamma_{g+1}}) \ar[d]_{\cong} \ar[rr]^{\phantom{aaaaaaa}C\Phi_0} & & \CF(L_\alpha, \underline{L}, L_{\gamma_{g+1}}) \ar[d]^{\cong} \\
	\CF(L_\alpha, L_\beta^{\omega_K}) \ar[rr]_{F_{W_\lambda}} & & \CF(L_\alpha, L_\gamma)}
\]
where the vertical maps are induced by geometric composition.
\label{Phi0Triangle}
\end{prop}

% There is a map
% \[
% 	\psi: \CF(L_0, \tau_V L_0) \otimes \CF(\underline{L}, L_0) \longrightarrow \CF(\underline{L}, \tau_V L_0).
% \]
% defined by counting pseudoholomorphic sections of the five $1$-parameter family of quilted Lefschetz fibrations as indicated in \todo{Figure ???}. From the construction of $\psi$, we see that
% \[
% 	\mu_1(\psi(\cdot,\cdot)) + \psi(\mu_1(\cdot),\cdot) + \psi(\cdot, \mu_1(\cdot)) + \kappa(\mu_2(\cdot, \cdot)) + h(\cdot, \cdot) + \mu_2(k(\cdot), \cdot) + \mu_3(c, \cdot, \cdot) = 0.
% \]
% We claim that this identity shows that the map induced by $h(b, \cdot)$ on homology is the same as the map induced by $\mu_3(c, b, \cdot)$ in homology (recall that $b$ is the unique intersection point of $V$ and $L_0$). To see this, it suffices to show that $\kappa(\mu_2(b,\cdot)) + \mu_2(k(b), \cdot)$ is the zero map in homology. 

Recall that the map $h: \CF(L_\alpha, \underline{L}, V) \otimes \CF(V^T, L_{\gamma_{g+1}}) \longrightarrow \CF(L_\alpha, \underline{L}, L_{\delta_{g+1}})$ is defined by
\[
	h(x, y) = \tilde{\mu}_2(x, k(y)) + \tilde{\mu}_3(x,y,c),
\]
where $c \in \CF(L_{\gamma_{g+1}}, L_{\delta_{g+1}})$ is the Floer chain defined by the standard Lefschetz fibration over the disk with monodromy $\tau_{\beta_{g+1}}$, and $k: \CF(V^T, L_{\gamma_{g+1}}) \longrightarrow \CF(L_{\beta_{g+1}}, L_{\delta_{g+1}})$ is the chain nullhomotopy of $\mu_2(\cdot, c)$ described in Figure \ref{fig:nullhomotopy-k} above.

The first thing to notice is that the first term, $\tilde{\mu}_2(x, k(y))$ is identically zero in this context. This is because $k$ is a map of degree $1$, but the $\Z/2$-graded complexes $\CF(V^T, L_{\gamma_{g+1}}) \cong \Z \langle \rho_0 \rangle$ and $\CF(L_{\gamma_{g+1}}, L_{\delta_{g+1}}) \cong \Z \langle c \rangle$ are both supported only in degree $1 \text{ mod } 2$. Hence we simply have
\[
	h(x,y) = \tilde{\mu}_3(x,y,c).
\]
The fact that $\CF(V^T, L_{\gamma_{g+1}}) \cong \Z \langle \rho_0 \rangle$ also forces $h$ to depend only on $x \in \CF(L_\alpha, L_\beta^{\omega_K})$, so that
\[
	h(x) = \tilde{\mu}_3(x, \rho_0, c).
\]
At this point, showing that $h$ corresponds to a classical rectangle map follows by a combination of the proofs that $C\Phi_0$ and $C\Phi_1$ correspond to classical triangle maps, as suggested in Figure \ref{fig:h-rectangle}.

\begin{figure}[h]
	\centering
	\includegraphics[scale=.9]{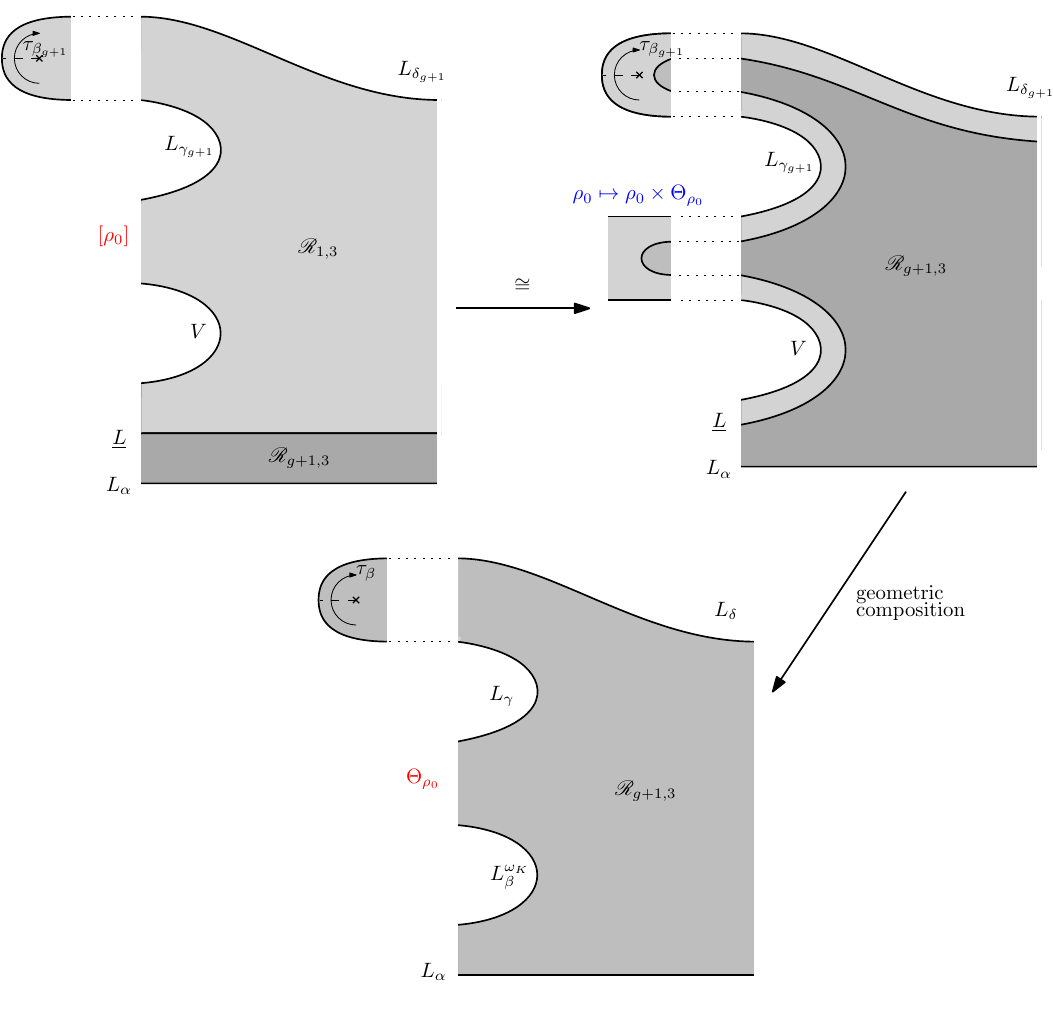}
	\caption{Identifying $h$ with a classical rectangle map.}
	\label{fig:h-rectangle}
\end{figure}

\begin{prop}
The map $h$ used in the proof of the surgery exact triangle may be identified with the rectangle map $\mu_3(\cdot,\Theta_{\beta\gamma}, \Theta_{\gamma\delta}): \CF(L_\alpha, L_\beta^{\omega_K}) \longrightarrow \CF(L_\alpha, L_\delta)$. In other words, there is a commutative diagram
\[
	\xymatrix{\CF(L_\alpha, \underline{L}, V) \ar[d]_{\cong} \ar[rr]^{h} & & \CF(L_\alpha, \underline{L}, L_{\delta_{g+1}}) \ar[d]^{\cong} \\
	\CF(L_\alpha, L_\beta^{\omega_K}) \ar[rr]_{\mu_3(\cdot,\Theta_{\beta\gamma}, \Theta_{\gamma\delta})} & & \CF(L_\alpha, L_\delta)}
\]
where the vertical maps are induced by geometric composition.
\label{hRectangle}
\end{prop}

It remains to identify the connecting map $\SI(Y_{\lambda + \mu}(K)) \longrightarrow \SI(Y, \omega_K)$ in the surgery exact triangle as being induced by a triangle map. Since the construction of this map was indirect (by homological algebra methods), instead of trying explicitly determine this map, we will show that it can be \emph{replaced} by a triangle map while still preserving exactness of the surgery triangle. Nevertheless, we still conjecture that the map (before replacement) can be shown to be a triangle map.

\begin{figure}
	\centering
	\includegraphics[scale=.5]{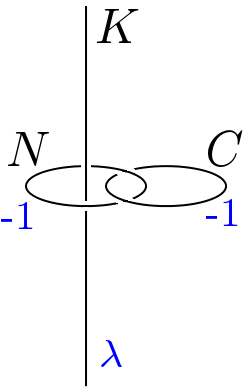}
	\caption{Framed knots inducing cobordisms in a surgery triad.}
	\label{fig:surgery-triad}
\end{figure}

The argument relies on basic Kirby calculus and the cyclic symmetry of surgery triads. Figure \ref{fig:surgery-triad} depicts the framed knots that induce the standard $2$-handle cobordisms between manifolds in the surgery triad $(Y, Y_\lambda(K), Y_{\lambda + \mu}(K))$, \emph{i.e.} attaching a $2$-handle to $Y$ along $K$ with framing $\lambda$ gives $Y_\lambda(K)$, attaching a $2$-handle to $Y_\lambda(K)$ along $N$ with framing $-1$ gives $Y_{\lambda + \mu}(K)$, and attaching a $2$-handle to $Y_{\lambda + \mu}$ along $C$ with framing $-1$ gives $Y$. Here we think of $N$ as the boundary of a normal disk to $K$ (using $\lambda$ to think of the normal bundle to $K$ as $K \times D^2$), given a framing of $-1$ relative to the normal disk. $C$ is a small meridian of $N$ which is also given a framing of $-1$ with respect to the normal disk bounded by $N$.

Since $(Y, Y_\lambda(K), Y_{\lambda + \mu}(K))$ is a surgery triad, so are $(Y_\lambda(K), Y_{\lambda + \mu}(K), Y)$ and $(Y_{\lambda + \mu}(K), Y,\allowbreak Y_\lambda(K))$. Hence we have two exact sequences
\[
	\SI(Y_\lambda(K), \omega_N \cup \omega_N \cup \omega_K) \xrightarrow{~\SI(W_\lambda, \Omega_1)~} \SI(Y_{\lambda + \mu}(K), \omega_N \cup \omega_K) \xrightarrow{~\SI(W_{\lambda + \mu}, \Omega_2)~} \SI(Y, \omega_N \cup \omega_K),
\]
\[
	\SI(Y_{\lambda+\mu}(K), \omega_C \cup \omega_C \cup \omega_K) \xrightarrow{~\SI(W_{\lambda+\mu}, \Omega_3)~} \SI(Y, \omega_C \cup \omega_K) \xrightarrow{~\SI(W, \Omega_4)~} \SI(Y_\mu(K), \omega_C \cup \omega_K),
\]
where we use the nontrivial bundle $\omega_N \cup \omega_K$ (resp. $\omega_C \cup \omega_K$) throughout in the first (resp. second) exact sequence. Note that $\omega_N \cup \omega_N \cup \omega_K = \omega_C = 0 \in H_1(Y_\lambda(K); \F_2)$, $\omega_N \cup \omega_K = \omega_C \cup \omega_C \cup \omega_K = 0 \in H_1(Y_{\lambda + \mu}(K);\F_2)$, and $\omega_N = \omega_C = 0 \in H_1(Y; \F_2)$, so that the two exact sequences above can be simplified to
\begin{align}
	\label{eqn:SES-surgery-N}
	\SI(Y_\lambda(K)) \xrightarrow{~\SI(W_\lambda)~} \SI(Y_{\lambda + \mu}(K)) \xrightarrow{~\SI(W_{\lambda + \mu}, \Omega_K^\prime)~} \SI(Y, \omega_K)
\end{align}
\begin{align}
	\label{eqn:SES-surgery-C}
	\SI(Y_{\lambda+\mu}(K)) \xrightarrow{~\SI(W_{\lambda+\mu}, \Omega_K^\prime)~} \SI(Y, \omega_K) \xrightarrow{~\SI(W, \Omega_K)~} \SI(Y_\mu(K))
\end{align}
Exactness of sequences (\ref{eqn:SES-surgery-N}) and (\ref{eqn:SES-surgery-C}) imply that
\[
	\ker(\SI(W_{\lambda + \mu}, \Omega_K^\prime)) = \im(\SI(W_\lambda)),
\]
\[
	\im(\SI(W_{\lambda + \mu}, \Omega_K^\prime)) = \ker(\SI(W, \Omega_K)),
\]
and therefore the triangle
\[
	\xymatrix{\SI(Y, \omega_K) \ar[rr]^{\SI(W, \Omega_K)} & & \SI(Y_\mu(K)) \ar[dl]^{\phantom{p}\SI(W_\lambda)} \\
	 & \SI(Y_{\lambda + \mu}(K)) \ar[ul]^{\SI(W_{\lambda+\mu}, \Omega_K^\prime)\phantom{pp}}}
\]
is exact.
\section{A Spectral Sequence for Link Surgeries}

\label{sect:linksurgeries}

We now show how to generalize the exact triangle for Dehn surgery on a knot to a spectral sequence for Dehn surgeries on a link. Our approach is heavily inspired by proofs of similar spectral sequences for Heegaard Floer homology \cite{oz-sz-branched}, singular instanton homology \cite{Kh-detector}, monopole Floer homology \cite{bloom-surgery}, and framed instanton homology \cite{scaduto}.

\subsection{The Link Surgeries Complex}

Let $Y$ be a closed, oriented $3$-manifold and $\bbL = (L_1 \cup \cdots \cup L_m, \lambda)$ be an oriented link with an enumeration of its components. An element $v = (v_1, \dots, v_m) \in \{0,1,\infty\}^m$ will be called a {\bf multi-framing} of $\bbL$. $Y_v$ will denote the $3$-manifold obtained from $Y$ by performing $v_k$-surgery on $L_k$ for eack $k = 1, \dots, m$ (with respect to the framing $\lambda$).

Recall that for each multi-framing $v$ of $\bbL$, we can construct a genus $g + m$ Heegaard triple $\calH_v = (\Sigma_{g+m}, \bfalpha, \bfbeta, \bfeta(v), z)$ satisfying the following conditions:
\begin{itemize}
	\item Attaching handles to $\{\alpha_1, \dots, \alpha_{g+m}\}$ and $\{\beta_1, \dots, \beta_g\}$ gives the complement of $L$ in $Y$.
	\item $\beta_{g+k}$ is a meridian for $L_k$.
	\item $\eta(v)_{g+k}$ represents the $v_k$-framing of $L_k$ with respect to the given framing $\lambda$.
	\item $\eta(v)_k = \beta_k$ for all $k = 1, \dots, g$.
	\item $(\Sigma_{g+m}, \bfalpha, \bfeta(v), z)$ is a Heegaard diagram for $Y_v$.
\end{itemize}
The same choice of curves $\bfalpha$ and $\bfbeta$ can be made for all $v \in \{0,1,\infty\}^m$, so that the various triple diagrams $\calH_v$ differ only in the framing curves $\bfeta(v)$. For any two multi-framings $v, w$, the Heegaard diagram $\{\Sigma_{g+m}, \bfeta(v), \bfeta(w), z\}$ represents a connected sum of copies of $S^2 \times S^1$ (the number $n_{vw}$ of which depends on how many components $v$ and $w$ differ in), so that we have as usual a distinguished element of maximal degree $\Theta_{vw} \in \HF(L_{\eta(v)}, L_{\eta(w)}) \cong H^{3-\ast}(S^3)^{\otimes n_{vw}}$ which acts as a unit for the triangle product.

For notational convenience, we will conflate the indices $\infty$ and $-1$ and henceforth consider multi-framings as elements of $\{-1,0,1\}^m$, with the understanding that to obtain $Y_v$, we perform $\infty$-surgery on $L_k$ if $v_k = -1$. With this change in notation, it now makes sense to define the {\bf weight} of $v$ by the formula
\[
	|v| = \sum_{k = 1}^m |v_k|.
\]
In other words, $|v|$ is the number of entries of $v$ not equal to $0$. Note that $\{-1,0,1\}^m$ is a lattice with the natural ordering given by
\[
	v \leq w \iff v_k \leq w_k \text{ for all } k = 1, \dots, m.
\]
When $|w - v| = 1$ and $v \leq w$, we say that $w$ is an {\bf immediate successor} of $v$.

If $w$ is an immediate successor of $v$, then $v$ and $w$ differ in just one of the components' framings, and $Y_w$ is obtained from $Y_v$ by attaching a single $2$-handle. In this case, write $\partial_{vw}: \CSI(Y_v) \longrightarrow \CSI(Y_w)$ for the map induced by this $2$-handle attachment. More generally, when $|w - v| = k + 1$ and $v \leq w$, given a sequence $v < u^1 < \cdots < u^k < w$ of immediate successors, define
\[
	\partial_{v<u^1<\cdots<u^k<w}(\xi) = \mu_{k+2}^{\alpha\eta(v)\eta(u^1)\cdots\eta(u^k)\eta(w)}(\xi \otimes \Theta_{\eta(v)\eta(u^1)} \otimes \cdots \otimes \Theta_{\eta(u^k)\eta(w)}),
\]
a count of pseudoholomorphic $(k+3)$-gons. To compactify the notation, write
\[
	\partial_{vw}: \CSI(Y_v) \longrightarrow \CSI(Y_w),
\]
\[
	\partial_{vw} = \sum_{v < u^1 < \cdots < u^k < w} \partial_{v < u^1 < \cdots < u^k < w},
\]
where the sum is over all ``paths'' in $\{-1, 0, 1\}^m$ from $v$ to $w$ through sequences of immediate successors. For all $v$, take $\partial_{vv}$ to just be the usual Floer differential on $\CSI(Y_v)$, and if $v > w$, define $\partial_{vw} \equiv 0$.

We will shortly need the following generalization of the K\"unneth principle to all polygons:

\begin{thm}
\label{thm:kunneth-polygon}
\textup{(K\"unneth Principle for Polygons)} If $\boldsymbol{\gamma}$ is a set of attaching curves respecting the connected sum decomposition $\Sigma_{g+m} = \Sigma_g \# \Sigma_n$, then the polygon map
\[
	\mu_{k+2}^{\gamma\eta(v)\eta(u^1)\cdots\eta(u^k)\eta(w)}(\xi \otimes \Theta_{\eta(v)\eta(u^1)} \otimes \cdots \otimes \Theta_{\eta(u^k)\eta(w)}): \CF(L_\gamma, L_{\eta(v)}) \longrightarrow \CF(L_\gamma, L_{\eta(w)})
\]
corrresponds exactly to a map of the form
\[
	\id \otimes \mu_{k+2}^{\gamma\eta(v)\eta(u^1)\cdots\eta(u^k)\eta(w)}(\xi' \otimes \Theta_{\eta(v)\eta(u^1)}^\prime \otimes \cdots \otimes \Theta_{\eta(u^k)\eta(w)}^\prime): \CF(L_{\gamma;g}^\prime, L_{\beta;g}^\prime) \otimes \CF(L_{\gamma;m}^\prime, L_{\eta(v)}^\prime)
\]
\[
	\longrightarrow \CF(L_{\gamma;g}^{\prime}, L_{\beta;g}^\prime) \otimes \CF(L_{\gamma;m}^\prime, L_{\eta(w)}^\prime).
\]
\end{thm}

The various Lagrangians appearing in the equivalent map are defined similarly to the case of triangles, based on the order in which we decide to attach handles. The proof is accomplished in exactly the same way as for triangles.

Now, define a chain complex
\[
	\widetilde{\bfC}(Y,L) = \bigoplus_{v \in \{-1,0,1\}^m} \CSI(Y_v), \quad\quad \bfD = \sum_{v \leq w} \partial_{vw}.
\]

\begin{prop}
$\bfD^2 \equiv 0$, so that $(\widetilde{\bfC}(Y,L), \bfD)$ is indeed a chain complex.
\end{prop}

\begin{proof}
The idea is to show that for any $\xi \in \widetilde{\bfC}(Y,L)$, the quantity $\bfD^2 \xi$ is precisely a count of degenerations of pseudoholomorphic polygons which coincides with an $A_\infty$ associativity relation that is known by general theory to be zero. The relevant $A_\infty$ relation here is the $A_k$ relation,
\[
	\sum_{0 \leq i < j \leq k} \mu^{\alpha \eta^1 \cdots \eta^i \eta^j \cdots \eta^k}_{k-j+i}( \xi \otimes \Theta_{\alpha \eta^1} \otimes \cdots \otimes \mu^{\eta^i \cdots \eta^j}_{j-i}(\Theta_{\eta^i \eta^{i+1}} \otimes \cdots \otimes \Theta_{\eta^{j-1}\eta^j}) \otimes \cdots \otimes \Theta_{\eta^{k-1} \eta^k}) = 0,
\]
for $0 \leq k \leq m$. A simple computation shows that for $v, w \in \{0,1,\infty\}^m$ with $|v - w| = k + 1$ and $v < w$, the $vw$-component of $\bfD^2 \xi$ is
\[
	(\bfD^2 \xi)_{vw} = \sum_{v < u^1 < \cdots < u^{k} < w} \sum_{0 < j \leq k} \mu^{\alpha \eta^j \cdots \eta^k}_{k-j}(\mu^{\alpha \eta^0 \cdots \eta^j}_{j}(\xi \otimes \cdots \otimes \Theta_{\eta^{j-1}\eta^j}) \otimes \cdots \otimes \Theta_{\eta^{k} \eta^{k+1}}),
\]
which is a sum over the $i = 0$ parts of various $A_{k+2}$ relations (here $\bfeta^0 = \bfeta(v)$, $\bfeta^{k + 1} = \bfeta(w)$, and for $1 \leq \ell \leq k$, $\bfeta^\ell = \bfeta(u^\ell)$). Hence if we can show the sum of the remaining terms of the appropriate $A_{k + 2}$ relations with $i \neq 0$ are all zero, we will have established that $\bfD^2 \equiv 0$.

We will in fact show that for any $v,w \in \{-1,0,1\}^m$ with $|w - v| = k + 1$ and $v < w$, we have
\[
	\sum_{v < u^1 < \cdots < u^k < w} \mu^{\eta(v) \eta(u^1) \cdots \eta(u^k)\eta(w)}_{k+1}(\Theta_{\eta(v) \eta(u^1)} \otimes \cdots \otimes \Theta_{\eta(u^k)\eta(w)}) = 0,
\]
where the sum is over all paths from $v$ to $w$ through sequences of immediate successors. This will clearly imply the desired result. We consider several cases:

\underline{$k = 0$:} In this case we want to show that
\[
	\mu_1^{\eta(v)\eta(w)}(\Theta_{\eta(v)\eta(w)}) = 0.
\]
But this is true because $\Theta_{\eta(v)\eta(w)} \in \CSI(L_{\eta(v)}, L_{\eta(w)})$ is a cycle and $\mu_1$ is the Floer differential.

\underline{$k = 1$:} There are two subcases here: either $v$ and $w$ differ in a single framing, or $v$ and $w$ differ in exactly two framings. In the first case, it is easy to see that the associated Heegaard triple diagram $(\Sigma_{g+m}, \bfeta(v), \bfeta(u^1), \bfeta(w), z)$ has a torus connect summand that corresponds to the cobordism map induced by a blowup. By \cite[Theorem 8.20]{horton1}, this map is zero, and hence $\mu_2^{\eta(v)\eta(u^1)\eta(w)}(\Theta_{\eta(v)\eta(u^1)} \otimes \Theta_{\eta(u^1)\eta(w)}) = 0$.

In the case where $v$ and $w$ differ in exactly two framings, there are two possible paths $v < u < w$ and $v < u' < w$ from $v$ to $w$ through immediate successors. Both of the associated Heegaard triple diagrams contain a genus $2$ connect summand which corresponds to the nearest point isomorphism $\CSI(S^2 \times S^1) \otimes \CSI(S^3) \cong \CSI(S^3)$, and the rest of Heegaard diagram corresponds to multiplying the unit of $\SI(\#^{g+m-2} S^2 \times S^1)$ with itself, so that
\[
	\mu_2^{\eta(v)\eta(u)\eta(w)}(\Theta_{\eta(v)\eta(u)} \otimes \Theta_{\eta(u)\eta(w)}) = \pm \Theta_{\eta(v)\eta(w)} = \mu_2^{\eta(v)\eta(u')\eta(w)}(\Theta_{\eta(v)\eta(u')} \otimes \Theta_{\eta(u')\eta(w)}).
\]
When summing over paths of immediate successors, these two summands therefore cancel each other out modulo $2$. In fact, they even cancel out with integer coefficients, but we will not use this fact in this article.

\underline{$k \geq 2$:} For the rest of cases we proceed by induction. Note that for any path $v < u^1 < \cdots < u^k < w$ through immediate successors, there will always be a torus connect summand of $\Sigma_{g + m}$ where $\eta(v) = \eta(u^1) = \cdots = \eta(u^k)$ and $\eta(w)$ meets these curves transversely in a single point. We will show that the polygon count from this particular summand is zero, which will force the overall polygon count to be zero by the K\"unneth principle.

When $k = 2$, let $v < u^1 < u^2 < w$ be a path from $v$ to $w$ through immediate successors. By adding an auxiliary Lagrangian $L_{\eta(u^{-1})}$ which is a small Hamiltonian isotope of $L_{\eta(v)}$, we see that
\[
	\mu_3(\Theta_{\eta(v)\eta(u^1)} \otimes \Theta_{\eta(u^1)\eta(u^2)} \otimes \Theta_{\eta(u^2)\eta(w)}) = \mu_3(\mu_2(\Theta_{\eta(u^{-1})\eta(v)} \otimes \Theta_{\eta(v)\eta(u^1)}) \otimes \Theta_{\eta(u^1)\eta(u^2)} \otimes \Theta_{\eta(u^2)\eta(w)}).
\]
On the other hand, the $A_\infty$ relations tell us that
\[
	\mu_3(\mu_2(\Theta_{\eta(u^{-1})\eta(v)} \otimes \Theta_{\eta(v)\eta(u^1)}) \otimes \Theta_{\eta(u^1)\eta(u^2)} \otimes \Theta_{\eta(u^2)\eta(w)})\phantom{lololololololololololololol}
\]
\[
	\phantom{lolololololololol} = \mu_3(\Theta_{\eta(u^{-1})\eta(v)} \otimes \mu_2(\Theta_{\eta(v)\eta(u^1)} \otimes \Theta_{\eta(u^1)\eta(u^2)}) \otimes \Theta_{\eta(u^2)\eta(w)})
\]
\[
	 \phantom{lololololololololololol =} + \mu_3(\Theta_{\eta(u^{-1})\eta(v)} \otimes \Theta_{\eta(v)\eta(u^1)} \otimes \mu_2(\Theta_{\eta(u^1)\eta(u^2)} \otimes \Theta_{\eta(u^2)\eta(w)})).
\]
It is clear that
\[
	\mu_2(\Theta_{\eta(v)\eta(u^1)} \otimes \Theta_{\eta(u^1)\eta(u^2)}) = \Theta_{\eta(v)\eta(u^2)}, \quad\quad \mu_2(\Theta_{\eta(u^1)\eta(u^2)} \otimes \Theta_{\eta(u^2)\eta(w)}) = \Theta_{\eta(u^1)\eta(w)},
\]
and therefore the two terms on the right-hand side above are equal since $L_{\eta(u^j)}$ is a small Hamiltonian isotope of $L_{\eta(u^\ell)}$ for any $-1, 0, \dots, k$. Hence with $\F_2$ coefficients, $\mu_3(\Theta_{\eta(v)\eta(u^1)} \otimes \Theta_{\eta(u^1)\eta(u^2)} \otimes \Theta_{\eta(u^2)\eta(w)}) = 0$. (As usual, one can establish this result with $\Z$ coefficients as well if appropriate signs are added)

For the inductive step, suppose that we have established that
\[
	\mu_k(\Theta_{\eta(v)\eta(u^1)} \otimes \cdots \otimes \Theta_{\eta(u^k)\eta(w)}) = 0.
\]
We want to show that
\[
	\mu_{k+1}(\Theta_{\eta(u^{-1})\eta(v)} \otimes \Theta_{\eta(v)\eta(u^1)} \otimes \cdots \otimes \Theta_{\eta(u^k)\eta(w)}) = 0,
\]
where $L_{\eta(u^{-1})}$ is a small Hamiltonian isotope of $L_{\eta(v)}$. Let $\mathcal{M}(k+3)$ denote the moduli space of conformal structures on the $(k+3)$-gon; this space is homeomorphic to $\R^k$. For any homotopy class of Whitney $(k+3)$-gons
\[
	\phi \in \pi_2(\Theta_{u^{-1}v}, \Theta_{vu^1}, \dots, \Theta_{u^k w},\Theta_{wv}),
\]
there is a forgetful map
\[
	G: \mathcal{M}(\phi) \longrightarrow \mathcal{M}(k+3)
\]
which keeps track of the conformal class of the domain. By Gromov compactness, this map is proper.

Note that the Lagrangian boundary conditions for the $(k+3)$-gon we are considering involve exactly $k$ non-convex corners, each of which contributes a $1$-parameter family of deformations of any pseudoholomorphic polygon by varying the length of the ``slit'' at the corner, or equivalently varying the conformal structure on the domain. It follows that $\mathcal{M}(\phi)$ at dimension at least $k$, since any pseudoholomorphic representative $u$ of $\phi$ automatically lies in a $k$-parameter family of representatives.

\begin{figure}
	\centering
	\includegraphics[scale=.5]{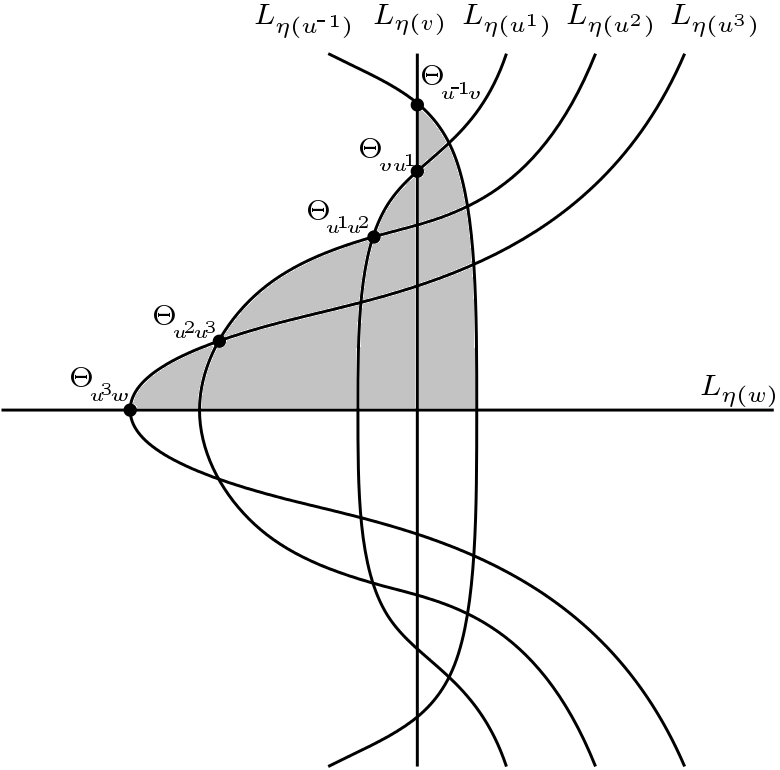}
	\caption{Schematic of polygons counted by $\mu_{k+1}(\Theta_{\eta(u^{-1})\eta(v)} \otimes \Theta_{\eta(v)\eta(u^1)} \otimes \cdots \otimes \Theta_{\eta(u^k)\eta(w)})$, $k = 3$.}
\end{figure}

\begin{figure}
	\centering
	\includegraphics[scale=.5]{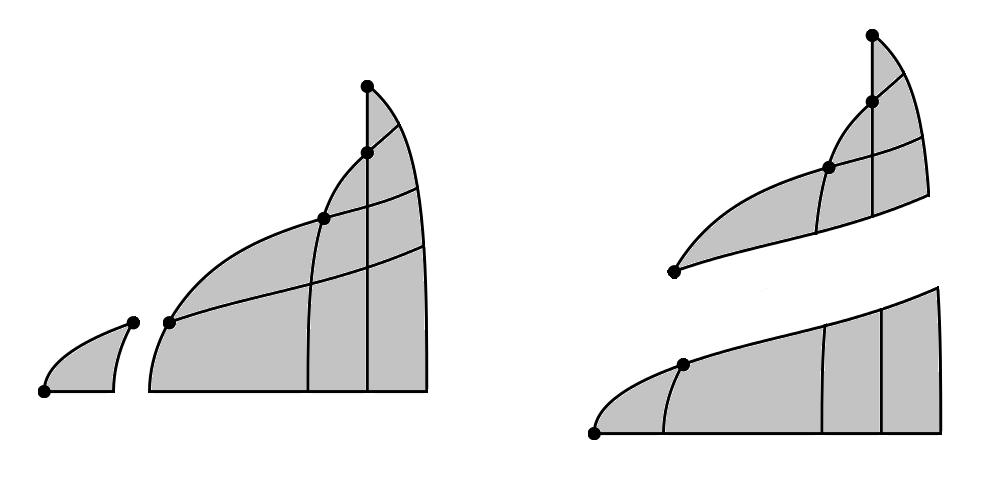}
	\caption{Schematic of the limiting pairs of polygons obtained by the two possible slits at a non-convex corner.}
	\label{fig:d2-degeneration}
\end{figure}

Suppose $\dim \mathcal{M}(\phi)$ is exactly $k$. Then for any conformal structure $c \in \mathcal{M}(k+3)$, $G^{-1}(c)$ is a finite set of points, since $G$ is proper. Near the ends of the moduli space $\mathcal{M}(\phi)$, the standard gluing theory for pseudoholomorphic polygons implies that the count $\# G^{-1}(c)$ will correspond to a count of degenerate pseudoholomorphic polygons where the appropriate ``slits'' in the image have become deep enough to meet another Lagrangian. In particular, by making a deep slit at the appropriate corner, we degenerate into a $(k+2)$-gon of the type we have already considered (in the inductive hypothesis) and a triangle (see Figure \ref{fig:d2-degeneration} for a schematic). The inductive hypothesis therefore implies that $\#G^{-1}(c)$ is even for $c$ near the appropriate end of $\mathcal{M}(k+3)$, and hence must be even for \emph{any} $c \in \mathcal{M}(k+3)$.

Now note that the only way to have a rigid pseudoholomorphic polygon coming from the entire Heegaard $(k+3)$-diagram under consideration is to have the other connect summands determine the conformal structure on the special torus connect summand investigated above, and there will only be a discrete set of pseudoholomorphic polygons contributed by this connect summand for a fixed conformal structure if the dimension of the relevant moduli space is exactly $k$. Since we just determined that the count of such polygons is zero modulo $2$, the K\"unneth principle then implies that the overall count is zero modulo $2$, so that
\[
	\mu_{k+1}(\Theta_{\eta(u^{-1})\eta(v)} \otimes \Theta_{\eta(v)\eta(u^1)} \otimes \cdots \otimes \Theta_{\eta(u^k)\eta(w)}) = 0.
\]
The desired result hence follows by induction.
\end{proof}

Write
\[
	\bfC(Y,L) = \bigoplus_{v \in \{0,1\}^m} \CSI(Y_v) \subset \widetilde{\bfC}(Y,L).
\]
We call the subcomplex $(\bfC(Y,L), \bfD)$ of $(\widetilde{\bfC}(Y,L), \bfD)$ the {\bf link surgeries complex} of $(Y,L)$.

We may define a filtration on $\bfC(Y,L)$ by
\[
	F^i \bfC(Y,L) = \bigoplus_{|v| \geq i} \CSI(Y_v).
\]
Since $\bfD$ is a sum of $\partial_{vw}$'s with $v \leq w$, it is clear that $\bfD$ preserves the filtration: $\bfD F^i \bfC(Y,L) \subseteq F^i \bfC(Y,L)$. Hence $(\bfC(Y, L), \bfD)$ is a filtered chain complex and we obtain an associated spectral sequence $E^r_{p,q}(Y,L)$ with
\[
	E^1 = \bigoplus_{v \in \{0,1\}^m} \SI(Y_v), \quad\quad d^1 = \sum_{\substack{v \leq w \\ |w - v| = 1}} \partial_{vw}.
\]
We will call the spectral sequence $E^r_{p,q}(Y,L)$ the {\bf link surgeries spectral sequence} for $(Y,L)$.

\subsection{Convergence of the Spectral Sequence}

Our goal now is to identify the $E^\infty$ page of the link surgeries spectral sequence $E^r_{p,q}(Y,L)$ with the symplectic instanton homology $\SI(Y)$. We achieve this by the technique of ``dropping a component,'' which we explain in detail below.

For $i = -1, 0, 1$, define $(\bfC_i, \bfD_i)$ to be the link surgeries complex for $(Y_i(L_1), L \setminus L_1)$, where $Y_i(L_1)$ denotes the result of $i$-surgery on $L_1 \subset Y$ and we recall that if $i = -1$, this means $\infty$-surgery. Note that for $v', w' \in \{-1,0,1\}^{m-1}$, we have that $(\bfD_i)_{v'w'} = \partial_{vw}$, where $v = (i, v')$ and $w = (i, w')$. Therefore we may consider $(\bfC_0, \bfD_0)$ and $(\bfC_1, \bfD_1)$ as subcomplexes of $(\bfC(Y,L), \bfD)$. $(\bfC_{-1}, \bfD_{-1})$ is the complex corresponding to ``dropping a component,'' since no surgery is performed on $L_1$ in this complex.

Define a map
\[
	\mathbf{f}_0: \bfC_0 \longrightarrow \bfC_1,
\]
\[
	\mathbf{f}_0 = \sum_{\substack{v, w \in \{0,1\}^m \\ v_1 = 0, w_1 = 1}} \partial_{vw}.
\]
$\mathbf{f}_0$ is simply the sum of components of $\bfD$ involving surgery on the component $L_1$. $\mathbf{f}_0$ is a (anti-)chain map for the same reason $\bfD^2 = 0$, and by construction it is clear that $(\bfC(Y,L), \bfD) = \Cone(\mathbf{f}_0)$, \emph{i.e.}
\[
	\bfC(Y,L) = \bfC_0 \oplus \bfC_1, \quad\quad \bfD = \begin{pmatrix} \bfD_0 & 0 \\ \mathbf{f}_0 & \bfD_1 \end{pmatrix}.
\]

Now consider the map
\[
	\mathbf{F}: \bfC_{-1} \longrightarrow \bfC(Y,L),
\]
\[
	\mathbf{F} = \sum_{\substack{v_1 = -1 \\ w_1 \in \{0,1\}}} \partial_{vw}.
\]
$\mathbf{F}$ is a chain map (anti-chain map if we work with $\Z$-coefficients), by a variant of the proof that $\bfD^2 = 0$:
\[
	\mathbf{F} \bfD_{-1} + \bfD \mathbf{F} = \sum_{\substack{v_1 = u_1 = -1 \\ w_1 \in \{0,1\}}} \partial_{uw} \partial_{vu} + \sum_{\substack{v_1 = -1 \\ u_1, w_1 \in \{0,1\}}} \partial_{uw} \partial_{vu} = 0.
\]
Consider the filtrations on $(\bfC, \bfD)$ and $(\bfC_{-1}, \bfD_{-1})$ given by the sum of all surgery coefficients except the one for $L_1$, $\sum_{i = 2}^m |v_i|$. $\mathbf{F}$ respects these filtrations, and in grading $p$ on the $E^0$-page of the induced spectral sequence, $\mathbf{F}$ restricts to a map of the form
\[
	\mathbf{F}_p^0: \bigoplus_{\substack{v_1 = -1 \\ \sum_{i = 2}^m |v_i| = p}} \CSI(Y_v) \longrightarrow \bigoplus_{\substack{v_1 \in \{0,1\} \\ \sum_{i = 2}^m |v_i| = p}} \CSI(Y_v).
\]
The restriction of $\mathbf{F}_p^0$ to a single direct summand has the form
\[
	\mathbf{F}_p^0|_{\CSI(Y_v)} = \partial_{vv^0} \oplus \partial_{vv^1}
\]
where $v^0$ and $v^1$ differ from $v$ only in the first component: $v^0_1 = 0$ and $v^1_1 = 1$ ($v_1 = -1$ by the definition of $\bfC_{-1}$). But in homology $\partial_{vv^0}$ is the same as the map $\Phi_0$ in the surgery exact triangle (by Proposition \ref{Phi0Triangle}), and $\partial_{vv^1}$ is the same as the map $h$ which serves as a chain nullhomotopy of $C\Phi_1 \circ C\Phi_0$ (by Proposition \ref{hRectangle}). It therefore follows from the Double Mapping Cone Lemma (see \cite[Lemma 5.4]{symp-gysin} or \cite[Lemma 5.9]{fiberedtriangle}) that $\mathbf{F}^0$ 
% \[
% 	\mathbf{F}^0: \bigoplus_{v_1 = -1} \CSI(Y_v) \longrightarrow \bigoplus_{v_1 \in \{0,1\}} \CSI(Y_v)
% \]
is a quasi-isomorphism, and hence so is
\[
	\mathbf{F}: \bfC_{-1} \longrightarrow \bfC(Y,L).
\]
In other words, the complex $\bfC_{-1}$ obtained from $\bfC(Y,L)$ obtained by ``dropping a component'' of $L$ (\emph{i.e.} not performing surgery on the component $L_1$) is in fact quasi-isomorphic to $\bfC(Y,L)$.

The above argument readily generalizes to dropping an arbitrary number of components of $L$. By dropping all $m$ components of $L$, we see that $\bfC(Y,L)$ is quasi-isomorphic to $\CSI(Y, \omega_L)$, and we conclude the following:

\begin{thm}
The link surgeries spectral sequence $E^r(Y, L)$ for an $m$-component link $L$ in a closed, oriented $3$-manifold $Y$ converges by the $m^\text{th}$ page to the symplectic instanton homology $\SI(Y, \omega_L)$.
\end{thm}

As was the case for the surgery exact triangle, a nontrivial $\SO(3)$-bundle on $Y$ may be incorporated into the link surgeries spectral seqeuence without changing the argument in any way. Indeed, any $\omega \in H_1(Y; \F_2)$ may be represented by a knot lying entirely in the $\alpha$-handlebody of the Heegaard splittings used in the proof of the link surgeries spectral sequence, and for any multiframing $v \in \{0,1, \infty\}^m$, $\omega$ induces an obvious homology class $\omega_v \in H_1(Y_v; \F_2)$. In this case, the link surgeries spectral sequence has $E^1$-page
\[
	E^1 = \bigoplus_{v \in \{0,1\}^m} \SI(Y_v, \omega_v)
\]
and converges by the $m^\text{th}$ page to $\SI(Y, \omega \cup \omega_L)$.

\iffalse
\todo{Examples from Baldwin's kitty-corner surgery.}

\todo{Gradings in the spectral sequence.}
\fi
\section{Khovanov Homology and Symplectic Instanton Homology of Branched Double Covers}

\label{sect:Khovanov}

As an application of the link surgeries spectral sequence, we exhibit a spectral sequence from the Khovanov homology of (the mirror of) a link in $S^3$ to the symplectic instanton homology of its branched double cover. The method of proof is inspired by Ozsv\'ath-Szab\'o's work on the corresponding spectral sequence in Heegaard Floer theory \cite{oz-sz-branched}.

\subsection{Khovanov Homology}

Let us give a definition of Khovanov homology suitable for our applications. Suppose $X \subset \R^2$ is a disjoint union of $k$ simple closed curves, $X = S_1 \cup \cdots \cup S_k$. Define $Z(X)$ to be the $\F_2$-vector space generated by the $S_i$, and write
\[
	V(X) = \wedge^\ast Z(X)
\]
for the corresponding exterior algebra.

Suppose $X' = S_1^\prime \cup \cdots \cup S_{k-1}^\prime$ is obtained from $X$ by merging two components $S_i$ and $S_j$. Then $Z(X') = Z(X)/(S_j - S_i)$ and there are natural isomorphisms
\[
	\alpha: (S_j - S_i) \wedge V(X) \xrightarrow{~\cong~} V(X'), \quad\quad \beta: V(X') \xrightarrow{~\cong~} V(X)/((S_j - S_i) \wedge V(X)).
\]
We may define a multiplication map by
\[
	m: V(X) \longrightarrow V(X'),
\]
\[
	m(\xi) = \alpha((S_j - S_i) \wedge \xi)
\]
and a comultiplication map
\[
	\Delta: V(X') \longrightarrow V(X),
\]
\[
	\Delta(\xi) = (S_j - S_i) \wedge \beta(\xi).
\]

Now, given a link $L$ in $S^3$, choose an oriented link diagram $\calD$ for $L$. Write $n$ for the total number of crossings in $\calD$, and $n_+$ (or $n_-$) for the number of positive (or negative) crossings of $\calD$. Furthermore, choose an enumeration ($1, \dots, n$) of the crossings. For any $I = (\epsilon_1, \dots, \epsilon_n) \in \{0,1\}^n$, let $\calD(I)$ denote the result of resolving all crossings of $\calD$, where the $i^\text{th}$ crossing is replaced with its $\epsilon_i$-resolution (where $0$- and $1$-resolutions are as defined in Figure \ref{fig:skein}). $\calD(I)$ is then a disjoint union of circles in the plane, so we can define $Z(\calD(I))$ and $V(\calD(I))$ for all $I \in \{0,1\}^n$. The {\bf $i^\text{th}$ Khovanov chain group} is given by
\[
	\CKh(\calD, i) = \bigoplus_{\substack{I \in \{0,1\}^n \\ |I| + n_+ = i}} V(\calD(I))
\]
The Khovanov differential
\[
	d: \CKh(\calD, i) \longrightarrow \CKh(\calD, i - 1)
\]
is defined by
\[
	d|_{V(\calD(I))} = \sum_{\substack{I^\prime \text{ is an immediate} \\ \text{successor of $I$}}} d_{I < I^\prime},
\]
where each $d_{I < I^\prime}: V(\calD(I)) \longrightarrow V(\calD(I'))$ is either a multiplication or comultiplication map, depending on whether $V(\calD(I'))$ differs from $V(\calD(I))$ by the merging or splitting of circles. It is easily checked that $d^2 = 0$, so that the {\bf Khovanov homology}
\[
	\Kh(\calD) = H_\ast\left( \bigoplus_i \CKh(\calD, i), d\right)
\]
is defined. $\Kh(\calD)$ is independent of the diagram $\calD$ of $L$ up to isomorphism, and therefore we will write $\Kh(L)$ to denote this isomorphism class of groups.

\begin{figure}
	\centering
	\includegraphics[scale=1.5]{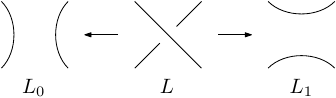}
	\caption{$0$- and $1$-resolutions of a crossing in a link.}
	\label{fig:skein}
\end{figure}

We will actually need a reduced variant of Khovanov homology. Fix a basepoint $p \in L$. Then a generic diagram $\calD$ for $L$ will have an induced basepoint, also denoted $p$, away from all crossings in $\calD$. It follows that for any $I \in \{0,1\}^n$ (where $n$ is the number of crossings in $\calD$), there is a disinguished component $S_I \in Z(\calD(I))$ containing $p$. Let us write
\[
	\tilde{V}(\calD(I)) = S_I \wedge V(\calD(I)).
\]
Then the {\bf $i^\text{th}$ reduced Khovanov chain group} is then defined by
\[
	\CKhr(\calD, i) = \bigoplus_{\substack{I \in \{0,1\}^n \\ |I| + n_+ = i}} \tilde{V}(\calD(I)),
\]
and the {\bf reduced Khovanov homology} is the homology of this subcomplex:
\[
	\Khr(\calD) = H_\ast\left( \bigoplus_i \CKhr(\calD, i), d\right).
\]
$\Khr(\calD)$ (with $\F_2$ coefficients) is independent of the choice of diagram $\calD$ for $L$ as well as the choice of basepoint $p \in L$ up to isomorphism, and we will write $\Khr(L)$ for this isomorphism class of groups.

\subsection{Branched Double Covers of Unlinks}

It turns out that the branched double cover of any link in $S^3$ is obtained from the branched double cover of some unlink by adding finitely many $2$-handles. In this section, we explain why this is true and give a natural identification of the symplectic instanton homology of the branched double cover of an unlink.

Let $L \subset S^3$ be any link and fix a diagram\footnote{It is not necessary to use a diagram for $L$ in this section, but we do it for simplicity and also because we will eventually relate to Khovanov homology, which does use a link diagram in its chain-level definition.} $\calD$ for $L$. A {\bf Conway sphere} for $L$ is an embedded $2$-sphere $S$ in $S^3$ intersecting $L$ transversely in exactly $4$ points. Such a sphere divides $S^3$ into two $3$-balls; let us suppose that one of these $3$-balls, $B$, contains exactly one crossing of the diagram $\calD$. Write $\tilde{Y}$ for the branched double cover of $S^3 \setminus B$ branched over $L \setminus (L \cap B)$. Note that the branched double cover of a $3$-ball branched over two arcs is a solid torus (via the hyperelliptic involution), and therefore $\tilde{Y}$ is $\Sigma(L)$ with a solid torus removed. The meridian $\gamma$ of this torus (in $\Sigma(L)$) is precisely the branched double cover of the pair of arcs from $L$.

Now consider the $0$- and $1$-resolutions $L_0$ and $L_1$ of $L$ (as pictured in Figure \ref{fig:skein}) at the single crossing inside $B$. Again, $\tilde{Y}$ is $\Sigma(L_0)$ minus a solid torus and also $\Sigma(L_1)$ minus a solid torus. However, the meridians $\gamma_0$ (resp. $\gamma_1$) of these solid tori correspond to the branched double covers of the pair of arcs $L_0 \cap B$ (resp. $L_1 \cap B$) in $\Sigma(L_0)$ (resp. $\Sigma(L_1)$). It is easily checked that $\#(\gamma \cap \gamma_0) = \#(\gamma_0 \cap \gamma_1) = \#(\gamma_1 \cap \gamma) = -1$. It therefore follows that $(\Sigma(L), \Sigma(L_0), \Sigma(L_1))$ is a surgery triad.

By resolving all the crossings of some diagram $\calD$ for $L \subset S^3$, we obtain an unlink with $k+1$ components, for some non-negative integer $k$. The particular choice of resolutions tells us how to attach $2$-handles to $\Sigma(L)$ to obtain $\Sigma(U \amalg \cdots \amalg U)$. The branched double cover of the two component unlink is $S^2 \times S^1$, and more generally the branched double cover of the $k+1$-component unlink is $\#^k S^2 \times S^1$. We already know the symplectic instanton homology of $\#^k S^2 \times S^1$ (as a unital algebra, no less), but we wish to have a geometric interpretation of it related to the fact that $\#^k S^2 \times S^1$ is the branched double cover of the $k+1$ component unlink. The following proposition establishes such an interpretation.

\begin{prop}
\label{prop:H1-module}
$\SI(\#^k S^2 \times S^1)$ is a free $\Lambda^\ast H_1(\#^k S^2 \times S^1)$-module of rank $1$, generated by the element $\Theta_k$ corresponding to the usual highest index intersection point of the relevant Lagrangians. Furthermore, the $\Lambda^\ast H_1(\#^k S^2 \times S^1; \F_2)$-module structure of $\SI(\#^k S^2 \times S^1)$ is natural with respect to $2$-handle cobordism maps, meaning the following:
\begin{itemize}
	\item[(1)] If $K$ is a knot dual to one of the $\mathrm{pt} \times S^1$'s in a summand of $\#^k S^2 \times S^1$, then $0$-surgery on $K$ gives $\#^{k-1} S^2 \times S^1$, where the component whose $S^2 \times S^1$ summand $K$ was dual to has been removed. If $\pi: H_1(\#^k S^2 \times S^1)/[K] \longrightarrow H_1(\#^{k-1} S^2 \times S^1)$ denotes the identification naturally induced by this $0$-surgery and $W: \#^k S^2 \times S^1 \longrightarrow \#^{k-1} S^2 \times S^1$ denotes the $2$-handle cobordism induced by the $0$-surgery, then
	\[
		F_W(\xi \cdot \Theta_k) = \pi(\xi) \cdot \Theta_{k-1}
	\]
	for all $\xi \in \Lambda^\ast H_1(\#^k S^2 \times S^1)$.
	\item[(2)] If $K$ is the unknot in $\#^k S^2 \times S^1$ and $W: \#^k S^2 \times S^1 \longrightarrow \#^{k+1} S^2 \times S^1$ is the $2$-handle cobordism induced by $0$-surgery on $K$, then
	\[
		F_W(\xi \cdot \Theta_k) = (\xi \wedge [K_W]) \cdot \Theta_{k+1}
	\]
	for all $\xi \in \Lambda^\ast H_1(\#^k S^2 \times S^1)$, where $[K_W]$ is the generator of the kernel of $i_\ast: H_1(\#^k S^2 \times S^1) \longrightarrow H_1(W)$.
\end{itemize}
\end{prop}

\begin{proof}
We already know that $\SI(\#^k S^2 \times S^1) \cong H^{3 - \ast}(S^3)^{\otimes k}$ as a unital algebra. We will write $H^{3-\ast}(S^3) = \F_2\langle\Theta, \theta\rangle$, with $\Theta$ the unit. $H_1(\#^k S^2 \times S^1; \F_2)$ has as basis the homology classes of the $\text{pt} \times S^1$'s in each connect summand; denote these classes by $X_1, \dots, X_k$. We define a $\Lambda^\ast H_1(\#^k S^2 \times S^1; \F_2)$-module structure on $\F_2 \langle \Theta, \theta\rangle^{\otimes k}$ as follows. On $\Theta^{\otimes k}$, the action is
\begin{align*}
	X_i \cdot \Theta^{\otimes k} & \mapsto \Theta^{\otimes (i - 1)} \otimes \theta \otimes \Theta^{\otimes (k - i - 1)}, \\
	(X_i \wedge X_j) \cdot \Theta^{\otimes k} & \mapsto \Theta^{\otimes (i - 1)} \otimes \theta \otimes \Theta^{\otimes (j - i - 1)} \otimes \theta \otimes \Theta^{\otimes (k - j - i - 2)} \quad \quad (i < j) \\
	 & \vdots \\
	(X_1 \wedge \cdots \wedge X_k) \cdot \Theta^{\otimes k} & \mapsto \theta^{\otimes k}.
\end{align*}
We see that every element $\eta \in \F_2\langle \Theta, \theta \rangle^{\otimes k}$ can be written as $\xi_\eta \cdot \Theta^{\otimes k}$ for a unique $\xi_\eta \in \Lambda^\ast H_1(\#^k S^2 \times S^1; \F_2)$. We extend the action to all of $\F_2\langle \Theta, \theta\rangle^{\otimes k}$ by setting
\[
	\xi \cdot \eta = (\xi \wedge \xi_\eta) \cdot \Theta^{\otimes k}.
\]
It is clear that with this action, $\F_2\langle \Theta, \theta\rangle^{\otimes k} \cong \Lambda^\ast H_1(\#^k S^2 \times S^1; \F_2)\langle\Theta^{\otimes k}\rangle$ as a $\Lambda^\ast H_1(\#^k S^2 \times S^1; \F_2)$-module. Now we prove the naturality with respect to $2$-handle cobordism maps, as in statements (1) and (2) of the Proposition.

(1) Thinking of the fundamental group of the torus as generated by the standard meridian and longitude $\mu$ and $\lambda$, consider the Heegaard triple diagrams $\calH = (\Sigma_1, \lambda, \lambda, \lambda)$ and $\calH' = (\Sigma_1, \lambda, \lambda, \mu)$. It is then easy to see that the triple diagram $\#^{k-1} \calH \# \calH'$ represents $0$-surgery in $\#^k S^2 \times S^1$ along the knot $K = \text{pt} \times S^1$ coming from the $S^2 \times S^1$ summand corresponding to the $\alpha$- and $\beta$-curves of $\calH'$. Direct inspection then reveals that
\[
	F_W(\xi \cdot \Theta_k) = \pi(\xi) \cdot \Theta_{k-1};
\]
if $\xi$ contains a wedge factor of $X_k$ then $F_W(\xi \cdot \Theta_k) = 0$ for Maslov index reasons, and if $\xi$ does not contain a wedge factor of $X_k$ then the claimed formula holds by considering $F_W$ being a closest point map on the factor corresponding to $\calH'$.

(2) Continuing the notation above, consider a third Heegaard triple $\calH'' = (\Sigma_1, \mu, \lambda, \lambda)$. Then the triple diagram $\#^k \calH \# \calH''$ represents $0$ surgery on $\#^k S^2 \times S^1$ along the unknot $K$ in the $S^3$ summand corresponding to the $\alpha$- and $\beta$- curves of $\calH''$. Similar arguments to the above establish that
\[
	F_W(\xi \cdot \Theta_k) = (\xi \wedge [K_W]) \cdot \Theta_{k+1};
\]
$F_W$ can be thought of as inducing a closest point map on all factors, and the $[K_W]$ appears because the factor $\calH''$ always gives the trivial connection $\theta$.
\end{proof}

\subsection{A Spectral Sequence from Khovanov Homology}

We now explain how Khovanov homology relates to the symplectic instanton homology of $\#^k S^2 \times S^1$ and deduce that the $E^2$-page of a certain link surgery spectral sequence associated to a branched double cover $\Sigma(L)$ is isomorphic to $\Khr(m(L); \F_2)$.

Let $L \subset S^3$ be a link and fix an $n$-crossing diagram $\calD$ for the \emph{mirror} link $m(L)$. As before, given any $I \in \{0,1\}^n$, we may construct the associated resolved diagram $\calD(I)$.

\begin{prop}
\label{prop:E2-Kh}
For any $I \in \{0,1\}^n$, there is an isomorphism $\Psi_I: \tilde{V}(\calD(I)) \longrightarrow \SI(\Sigma(\calD(I)))$ which is natural in the following sense: if $I' \in \{0,1\}^n$ is an immediate successor of $I$ and $F_{I < I^\prime}: \SI(\Sigma(\calD(I))) \longrightarrow \SI(\Sigma(\calD(I')))$ is the map induced by the associated $2$-handle cobordism on the branched double covers, then the following diagram commutes:
\[
	\xymatrix{\tilde{V}(\calD(I)) \ar[r]^{d_{I < I^\prime}} \ar[d]_{\Psi_I} & \tilde{V}(\calD(I')) \ar[d]^{\Psi_{I^\prime}} \\
	\SI(\Sigma(\calD(I))) \ar[r]_{F_{I < I^\prime}} & \SI(\Sigma(\calD(I')))}
\]
\end{prop}

\begin{proof}
The proof is really just a matter of understanding the relationship between merging/splitting unlinks and surgeries on the branched double cover, and then applying Proposition \ref{prop:H1-module}.

Fix a basepoint $p \in L$ and write $\calD(I) = S(I)_1 \amalg \dots \amalg S(I)_{k(I)}$, where we order the circle components so that $S(I)_1$ contains the basepoint $p$, and similarly write $\calD(I') = S(I')_1 \amalg \cdots \amalg S(I')_{k(I')}$. For convenience, we will assume further that the components are ordered so that the Khovanov differential $d_{I<I'}$ only possibly involves $\{S(I)_1, S(I)_2, S(I)_3\} \subset \calD(I)$ and $\{S(I')_1, S(I')_2, S(I')_3\} \subset \calD(I')$.

As in the proof of Proposition \ref{prop:H1-module}, $H_1(\Sigma(\calD(I)); \F_2)$ has as basis the homology classes $X_1, \dots, \allowbreak X_{k(I)}$ which arise as $\text{pt} \times S^1$ in each summand of $\#^{k(I)} S^2 \times S^1 \cong \Sigma(\calD(I))$. If $\gamma_i$ denotes an arc in $S^3$ whose only points in $\calD(I)$ are its endpoints, one of which lies on $S(I)_1$ and the other of which lies on $S(I)_{i+1}$, then $X_i$ is mod $2$ homologous to the branched double cover of $\gamma_i$.

We have an identification of $\SI(\Sigma(\calD(I)))$ as a rank $1$ free $\Lambda^\ast H_1(\Sigma(\calD(I)); \F_2)$-module from Proposition \ref{prop:H1-module}. It follows that the identification $S(I)_i \mapsto X_i$ induces an isomorphism $\Psi_I: \tilde{V}(\calD(I)) \longrightarrow \SI(\Sigma(\calD(I)))$.

\begin{figure}[h]
	\centering
	\includegraphics[scale=1.25]{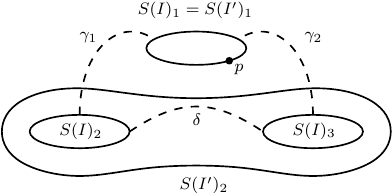}
	\caption{Merging two circles and the effect on homology.}
	\label{fig:branched-surgery}
\end{figure}

Suppose $\calD(I')$ differs from $\calD(I)$ by the merging of $S(I)_2$ and $S(I)_3$ into $S(I')_2$ (see Figure \ref{fig:branched-surgery} for a schematic). Then it is clear that $X_1^\prime, X_2^\prime \in H_1(\Sigma(\calD(I')); \F_2)$ are homologous, since after merging, both $\gamma_1$ and $\gamma_2$ will have an endpoint on $S(I')_2$. One then sees that $\Sigma(\calD(I'))$ is obtained from $\Sigma(\calD(I))$ by $0$-surgery along the branched double cover of the pictured arc $\delta$, which is just an unknot in $\Sigma(\calD(I))$. It follows from case (2) of Proposition \ref{prop:H1-module} that in this case, $F_{I < I^\prime} \circ \Psi_I = \Psi_{I^\prime} \circ d_{I < I^\prime}$, where we note that merging is corresponding to comultiplication on $\tilde{V}(\calD(I))$ because $\calD$ is a diagram for the mirror of $L$, and mirroring swaps the roles of multiplication and comultiplication in the Khovanov differential.

Dually, if $\calD(I')$ is obtained from $\calD(I)$ by splitting $S(I)_2$ into $S(I')_2$ and $S(I')_3$, a similar argument shows that $X_1^\prime$ and $X_2^\prime$ are homologous in $H_1(\Sigma(\calD(I')); \F_2)$ and $\Sigma(\calD(I'))$ is obtained from $\Sigma(\calD(I))$ by $0$-surgery along $X_1$. Case (1) of Proposition \ref{prop:H1-module} then implies that $F_{I < I^\prime} \circ \Psi_I = \Psi_{I^\prime} \circ d_{I < I^\prime}$.

There are two additional cases to consider, namely merges and splits involving the circle with the marked point $p$. But these are dealt with in the same way as before, if we introduce $\gamma_0 = 0$ and $X_0 = 0$.
\end{proof}

We can now derive the main result of this section.

\begin{thm}
\label{thm:Kh-SS}
For any link $L$ in $S^3$, there is a spectral sequence with $E^2$ page given by $\Khr(m(L); \F_2)$ abutting to $\SI(\Sigma(L); \F_2)$.
\end{thm}

\begin{proof}
Fix a diagram $\calD$ for $m(L)$ with $n$ crossings and a basepoint $p$. At the $i^\text{th}$ crossing of $\calD$, let $\mu_i \in H_1(\Sigma(L); \F_2)$ be the mod $2$ homology class represented by the branched double cover of the arc in $S^3$ connecting the two strands of $L$ at that crossing. Write $\mu = \mu_1 + \cdots + \mu_n \in H_1(\Sigma(L); \F_2)$. Given any $I \in \{0,1\}^n$, we have that $I$-framed surgery along the link $\mu = \mu_1 \amalg \cdots \amalg \mu_n$ results in $\Sigma(\calD(I))$, for the same reason that $(\Sigma(J), \Sigma(J_0), \Sigma(J_1))$ is a surgery triad for any link $J$.

From the previous paragraph and the link surgeries spectral sequence for symplectic instanton homology, it follows that we have a spectral sequence with
\[
	E^1 = \bigoplus_{I \in \{0,1\}^n} \SI(\Sigma(\calD(I))), \quad\quad d^1|_{\SI(\Sigma(\calD(I)))} = \sum_{\substack{\text{$I'$ an immediate} \\ \text{successor of $I$}}} F_{I < I^\prime}
\]
converging to $\SI(\Sigma(L), \mu)$. By Proposition \ref{prop:E2-Kh}, it follows that the $E^2$-page of this spectral sequence is isomorphic to $\Khr(m(L); \F_2)$.

To complete the proof, it remains to show that $\mu$ is nullhomologous. We construct an explicit surface in $\Sigma(L)$ bounded by $\mu$, as explained in \cite[Section 2]{scaduto}. The surface will be constructed as a lift of a surface in $S^3$ to the branched double cover; the idea is to apply a variant of the Seifert algorithm to the diagram $\calD$, where instead of connecting disks with twisted bands, we want to get a surface from attaching ``half-shaded'' twisted bands. This requires a modification to the usual algorithm, as we need to make sure each band connects a shaded region to a non-shaded region, and therefore we cannot use all of the disks the traditional seifert algorithm does. In any case, it is clear that the lift of the surface in $S^3$ defined by the region obtained by this construction (whose boundary consists of segments of $\calD$ and arcs connecting the two local components of $\calD$ at a crossing) lifts to a surface in $\Sigma(L)$ whose boundary is precisely $\mu$.

\begin{figure}[h]
	\centering
	\includegraphics[scale=1.5]{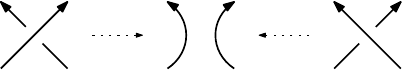}
	\caption{The oriented resolution of a crossing.}
	\label{fig:ori-res}
\end{figure}

To show that the desired surface exists, choose an orientation of $L$ and let $\calD'$ be the diagram obtained from $\calD$ by performing a oriented resolution (see Figure \ref{fig:ori-res}) of each crossing. $\calD'$ will consist of $N$ oriented circles $S_1^\prime, \dots, S_N^\prime$ in the plane. To each $S_k^\prime$, assign two signs $a_k$ and $b_k$:
\begin{itemize}
	\item $a_k = +1$ if $S_k^\prime$ is oriented counterclockwise; $a_k = -1$ otherwise.
	\item $b_k = (-1)^M$, where $M$ is the number of circles in $\calD'$ that surround $S_k^\prime$.
\end{itemize}
Fill in each disk bounded by a $S_k^\prime$ for which $a_kb_k = +1$ (one may wish to think as each $S_k^\prime$ as lying in the plane $z = k$ in $\R^3$, as some regions will be colored multiple times). To get the original link $L$ back, we attach bands with a half twist to connect the circles where we originally performed the oriented resolutions. The condition that $S_k^\prime$ be filled only if $a_kb_k = +1$ ensures that for any two circles joined by a band, exactly one of them has its interior filled, so that we may fill the appropriate half of the band to obtain our desired surface, which has boundary consisting of segements from $\calD$ and for each crossing of $\calD$, an arc connecting the two pieces of $L$ near that crossing.

The lift of the surface in $S^3$ constructed in the previous paragraph to the branched double cover $\Sigma(L)$ is a surface whose boundary is precisely $\mu$, and therefore $\mu$ is nullhomologous. It follows that $\SI(\Sigma(L), \mu) = \SI(\Sigma(L))$, and therefore our spectral sequence converges to $\SI(\Sigma(L))$, as claimed.
\end{proof}

As a quick corollary, we obtain a rank inequality for the symplectic instanton homology of a branched double cover.

\begin{cor}
For any link $L$ in $S^3$, $|\det(L)| \leq \rk \SI(\Sigma(L); \F_2) \leq \rk \Khr(m(L); \F_2)$.
\end{cor}

\begin{proof}
This follows from the fact that $\chi(\SI(\Sigma(L))) = |H_1(\Sigma(L))| = |\det(L)|$ and that there is a spectral sequence $\Khr(m(L)) \Rightarrow \SI(\Sigma(L))$.
\end{proof}
\section{Nontrivial Bundles on Branched Double Covers}

\label{sect:TQA}

The spectral sequence for a branched double cover discussed in the previous section can also be studied for nontrivial $\SO(3)$-bundles using a homology related to a ``twisted'' version of Khovanov homology, as first explained in the gauge theory context by Scaduto and Stoffregen \cite{TQA}.

\subsection{Two-Fold Marking Data}

Let $L = L_1 \amalg \cdots \amalg L_N \subset S^3$ be a link. A {\bf two-fold marking datum for $L$} is a function
\[
	\omega: \{L_1, \dots, L_N\} \longrightarrow \Z/2
\]
such that $\omega(L_1) + \cdots + \omega(L_N) \equiv 0 \text{ mod } 2$. One can think of $\omega$ as assigning a collection of points $\mathbf{p} = \{p_1, \dots, p_m\}$ to $L$, $m < N$, with one point on $L_k$ if $\omega(L_k) = 1$, such that the total number of points is even. The pointed link $(L, \mathbf{p})$ is easily seen to correspond to an element of $H_1(\Sigma(L); \F_2)$: if one takes a collection of arcs in $S^3$ whose interiors are disjoint from $L$ and whose endpoints are precisely the basepoints $\mathbf{p}$, then the lift of this collection of arcs to $\Sigma(L)$ will be a well-defined mod $2$ homology class. In fact, all elements of $H_1(\Sigma(L); \F_2)$ arise from some such collection of basepoints $\mathbf{p} \subset L$, so that we have the following:

\begin{prop}
There is a one-to-one correspondence between two-fold marking data $\omega$ for $L$ and elements of $H_1(\Sigma(L); \F_2)$.
\end{prop}

In particular, each two-fold marking datum $\omega$ corresponds to a nontrivial $\SO(3)$-bundle on $\Sigma(L)$, which by abuse of notation we will also refer to as $\omega$.

\subsection{Twisted Khovanov Homology and Dotted Diagram Homology}

We now wish to introduce a variant of Khovanov homology which takes into account two-fold marking data. To do this, we will first need to introduce a compatibility relation between the two-fold marking data and the link diagrams we use. Let $L \subset S^3$ be a link and $\calD$ be a diagram for $L$. An {\bf arc} of $\calD$ is a strand of $\calD$ that descends to an edge of the associated $4$-valent graph.  Let $\Gamma$ denote the set of all arcs in $\calD$, and given a component $L_k$ of $L$, let $\Gamma(L_k)$ denote the set of all arcs contained in the image of $L_k$ in $\calD$. Given a two-fold marking datum $\omega$ for $L$, we say that an assignment $\check{\omega}: \Gamma \longrightarrow \Z/2$ is {\bf compatible} with $\omega$ if
\[
	\sum_{\gamma \in \Gamma(L_k)} \check{\omega}(\gamma) \equiv \omega(L_k) \text{ mod } 2.
\]
The pair $(\calD, \check{\omega})$ will be called a {\bf two-fold marked diagram}. Note that the $0$- and $1$-resolutions of any crossing in $\calD$ have well-defined induced two-fold markings, giving two-fold marked diagrams $(\calD_0, \check{\omega}_0)$ and $(\calD_1, \check{\omega}_1)$.

Given a two-fold marked link $(L, \omega)$, fix a compatible two-fold marked diagram $(\calD, \check{\omega})$ with $n$ crossings, and fix an auxiliary basepoint $p_0 \in \calD$ (which has nothing to do with $\check{\omega}$). For any $I = (\epsilon_1, \dots, \epsilon_n) \in \{0, 1\}^n$, we then obtain a resolved diagram $(\calD(I), \check{\omega}(I))$ consisting of $k(I)$ pointed circles:
\[
	(\calD(I), \check{\omega}(I)) = (S(I)_1, \check{\omega}(I)_1) \cup \cdots \cup (S(I)_{k(I)}, \check{\omega}(I)_{k(I)}).
\]
Exactly one of these circles, which we will denote by $S_I$, contains the auxiliary basepoint $p_0$. As in the untwisted case, we define
\[
	\CKh(\calD, \check{\omega}) = \bigoplus_{I \in \{0,1\}^n} S_I \wedge V(\calD(I)),
\]
\emph{i.e.} $\CKh(\calD, \check{\omega})$ as a chain group is just the reduced Khovanov chain group with respect to the basepoint $p_0$, and we incorporate $\check{\omega}$ into the differential as follows. The differential $d$ on $\CKh(\calD, \check{\omega})$ is a sum of ``horizontal'' and ``vertical'' differentials,
\[
	d = d_h + d_v.
\]
The horizontal differential $d_h$ is just the usual Khovanov differential, defined in terms of merge/split maps. The vertical differential $d_v$ is defined by its restriction to a direct summand: 
\[
	d_v \xi = \sum_{j = 1}^{k(I)} \check{\omega}(I)_j S(I)_j \wedge \xi \text{ for } \xi \in S_I \wedge V(\calD(I)).
\]
One may check that $d^2 = 0$, and the {\bf twisted Khovanov homology} of $(L, \omega)$ is defined as
\[
	\Kh(L, \omega) = H_\ast(\CKh(\calD, \check{\omega}), d).
\]
$\Kh(L, \omega)$ may be thought of as a generalization or deformation of $\Khr(L)$: if $\omega_0$ is the trivial two-fold marking data (which assigns $0$ to all components of $L$), then $\Kh(L,\omega_0) \cong \Khr(L)$.

The horizontal and vertical differentials on $\CKh(\calD, \check{\omega})$ are easily seen to commute, and therefore $\CKh(\calD, \check{\omega})$ admits the structure of a double complex. The homology with respect to the horizontal differential $d_h$ is the reduced Khovanov homology of $L$, and the homology with respect to the vertical differential $d_v$ results in the subcomplex of $(\CKh(\calD, \check{\omega}), d_h)$ consisting of summands corresponding to $I \in \{0,1\}^n$ with all $\check{\omega}(I)_j$'s even. We call the homology of this subcomplex with respect to $d_h$ the {\bf dotted diagram homology} of $(\calD, \check{\omega})$ and denote it by
\[
	\Hd(\calD, \check{\omega}) = H_\ast(H_\ast(\CKh(\calD, \check{\omega}), d_v), d_h).
\]
We remark that $\Hd(\calD, \check{\omega})$ is \emph{not} an invariant of $(L,\omega)$, but it naturally appears as the $E^2$-page of a link surgeries spectral sequence for a nontrivial $\SO(3)$-bundle on the branched double cover of $L$, as we explain in the next section. Note that the spectral sequence for the double complex $(\CKh(\calD, \check{\omega}), d_v, d_h)$ gives another spectral sequence $\Hd(\calD, \check{\omega}) \Rightarrow \Kh(L, \omega)$.

\begin{rem}
\label{rem:Hd-complex}
It is easy to see that $d_v|_{\tilde{V}(\calD(I))}$ is an isomorphism if $\check{\omega}(I)$ is nonzero on any component of $\calD(I)$, and zero otherwise. It follows that $H_\ast(\CKh(\calD, \check{\omega}), d_v)$ is just the subcomplex of $\CKhr(\calD)$ consisting of $I$-summands with $\check{\omega}(I) \equiv 0$.
\end{rem}

\subsection{Spectral Sequence for Nontrivial Bundles on Branched Double Covers}

We now explain the relevance of dotted diagram homology to the symplectic instanton homology of nontrivial $\SO(3)$-bundles over branched double covers of links. 

\iffalse
First, we put a mild condition on the types of two-fold marked diagrams we consider.

Up to \todo{dotted moves}, we can represent $(\calD, \check{\omega})$ by a dotted diagram $(\calD, \mathbf{p})$, $\mathbf{p} = \{p_1, \dots, p_m\} \subset \calD$, which satisfies the condition
\[
	\sum_{\gamma \in \Gamma(L_k)} \check{\omega}(\gamma) \equiv \#\{p_\ell \in \mathbf{p} \mid p_\ell \text{ corresponds to a point of } L_k\} \text{ mod } 2
\]
for all components $L_k$ of $L$. We call such a dotted diagram $(\calD, \mathbf{p})$ a {\bf dotted surgery diagram} if the following are satisfied:
\begin{itemize}
	\item[(1)] Each $p_\ell \in \mathbf{p}$ lies in a small neighborhood of a crossing of $\calD$.
	\item[(2)] Near each crossing of $\calD$, there are exactly $0$ or $2$ points from $\mathbf{p}$.
	\item[(3)] For crossings which there are $2$ points from $\mathbf{p}$ in a small neighborhood $B$ of the crossing, the two points do not lie on the same component of $\calD \cap B$.
\end{itemize}
It is easy to see that we can find a dotted surgery diagram $(\calD, \mathbf{p})$ representing any given two-fold marked link $(L, \omega)$. The significance of dotted surgery diagrams will become relevant in the proof of the following theorem, which is the main result of this section.
\fi

\begin{thm}
\label{thm:SSdotted}
Let $(L, \omega)$ be a two-fold marked link in $S^3$, and suppose $(\calD, \check{\omega})$ is a compatible two-fold marked diagram for the mirror $(m(L), \omega)$. Then there is a spectral sequence with $E^2$-page isomorphic to the dotted-diagram homology $\Hd(\calD, \check{\omega})$ converging to $\SI(\Sigma(L), \omega)$, where $\omega$ denotes the $\SO(3)$-bundle on $\Sigma(L)$ induced by the two-fold marking data $\omega$.
\end{thm}

\begin{proof}
\iffalse
Fix a dotted surgery diagram $(m(\calD), \mathbf{p})$ compatible with $(m(L), \omega)$. For each crossing of $m(\calD)$ where there are two points from $\mathbf{p}$ nearby, we can consider the loop in $\Sigma(L)$ which is the branched double cover of the obvious arc connecting the two points. Let $\mu \in H_1(\Sigma(L); \F_2)$ denote the mod $2$ homology class of the disjoint union of all loops in $\Sigma(L)$ obtained from $(m(\calD), \mathbf{p})$ in this way.
\fi

Let $\mu$ denote the $n$-component link in $\Sigma(L)$ obtained by lifting the $n$ arcs in $S^3$ connecting the two local components of each crossing of $\calD$. As in the proof of Theorem \ref{thm:Kh-SS}, for any $I \in \{0,1\}^n$, the result of $I$-framed surgery on $\mu = \mu_1 \amalg \cdots \amalg \mu_n$ results in $\Sigma(\calD(I))$. Letting $\omega \in H_1(\Sigma(L); \F_2)$ denote the mod $2$ homology class corresponding to the two-fold marking data $\omega$ (by abuse of notation), we get induced homology classes $\omega(I) \in H_1(\Sigma(\calD(I)); \F_2)$ for each $I \in \{0,1\}^n$. The link surgery spectral sequence for $(\Sigma(L), \omega)$ and the link $\mu$ then has
\[
	E^1 = \bigoplus_{I \in \{0,1\}^n} \SI(\Sigma(\calD(I)), \omega(I)), \quad\quad d^1|_{\SI(\Sigma(\calD(I)), \omega(I))} = \sum_{\substack{\text{$I^\prime$ an immediate} \\ \text{successor of $I$}}} F_{I < I^\prime}
\]
and converges to $\SI(\Sigma(L), \omega \cup \mu)$.

We now identify the $E^1$- and $E^2$-pages of the spectral sequence. First recall that the mod $2$ homology class $\omega(I) \in H_1(\Sigma(\calD(I)); \F_2)$ is obtained from the two-fold marked diagram $(\calD(I), \check{\omega}(I))$ by first choosing a compatible collection of dots $\mathbf{p}(I) \subset \calD(I)$ and family of arcs in $S^3$ whose interiors are disjoint from $\calD(I)$ and whose endpoints lie in $\mathbf{p}(I)$ in such a way that each point in $\mathbf{p}(I)$ occurs as the endpoint of one of these arcs exactly once. $\omega(I)$ is then the mod $2$ homology class of lift of this collection of arcs to $\Sigma(\calD(I))$. Aside from the marked circle $S_I \subset \calD(I)$, each circle in $\calD(I)$ corresponds to an $S^2 \times S^1$-summand under the identification $\Sigma(\calD(I)) \cong \#^{k(I)} S^2 \times S^1$, and it is easy to see that the lift of an arc connecting two circles in $\calD(I)$ is mod $2$ homologous to the disjoint union of the copies of $\text{pt} \times S^1$ in the $S^2 \times S^1$-summand the circles correspond to. It therefore follows that $\omega(I)$ is trivial on $S^2 \times S^1$-summands of $\Sigma(\calD(I))$ corresponding to circles in $\calD(I)$ that $\check{\omega}(I)$ is even on, and nontrivial on $S^2 \times S^1$-summands of $\Sigma(\calD(I))$ corresponding to circles in $\calD(I)$ that $\check{\omega}(I)$ is odd on. Since $\SI(S^2 \times S^1, [\text{pt} \times S^1]) = 0$ by \cite[Proposition 9.5]{horton1}, we can therefore use Proposition \ref{prop:E2-Kh} and our understanding of the vertical homology of the twisted Khovanov complex (cf. Remark \ref{rem:Hd-complex}) to conclude that
\[
	E^1 \cong H_\ast(\CKh(\calD, \check{\omega}), d_v), \quad\quad d^1 = d_h.
\]
By the definition of dotted diagram homology, it immediately follows that $E^2 \cong \Hd(\calD, \check{\omega})$. As argued in the proof of Theorem \ref{thm:Kh-SS}, the link $\mu$ in $\Sigma(L)$ is nullhomologous, so that the spectral sequence in fact converges to $\SI(\Sigma(L), \omega)$, as desired.
\end{proof}

\appendix

\bibliographystyle{plain}
\bibliography{traceless}

\end{document}